\documentclass{amsart}

\usepackage{amsmath,amssymb,amsthm,enumerate} %% default
\usepackage{amsfonts} %% for mathbb
\usepackage[top=30mm,bottom=30mm,left=25mm,right=25mm]{geometry}
\usepackage[dvipdfmx]{graphicx}
\usepackage{color}
\usepackage{cite}

\theoremstyle{plain}
\newtheorem{pretheo}{Theorem}[section]
\newtheorem{preassu}[pretheo]{Assumption}
\newtheorem{precoro}[pretheo]{Corollary}
\newtheorem{predefi}[pretheo]{Definition}
\newtheorem{preexam}[pretheo]{Example}
\newtheorem{prelemm}[pretheo]{Lemma}
\newtheorem{preprop}[pretheo]{Proposition}
\newtheorem{prerema}[pretheo]{Remark}

\newenvironment{theo}{\begin{pretheo}}{\end{pretheo}}

\newenvironment{defi}{\begin{predefi}}{\end{predefi}}

\newenvironment{lemm}{\begin{prelemm}}{\end{prelemm}}
\newenvironment{prop}{\begin{preprop}}{\end{preprop}}
\newenvironment{rema}{\begin{prerema}\rm}{\end{prerema}}

\DeclareMathOperator{\di}{div}

\DeclareMathOperator{\Di}{Div}
\newcommand{\jump}[1]{\ensuremath{[\![#1]\!]}}

\newcommand{\pa}{\partial}

\newcommand{\wh}[1]{\widehat{#1}}
\newcommand{\wt}[1]{\widetilde{#1}}

\newcommand{\BBA}{\mathbb{A}}
\newcommand{\BBB}{\mathbb{B}}
\newcommand{\BBC}{\mathbb{C}}
\newcommand{\BBD}{\mathbb{D}}
\newcommand{\BBH}{\mathbb{H}}
\newcommand{\BBI}{\mathbb{I}}

\newcommand{\BBN}{\mathbb{N}}
\newcommand{\BBR}{\mathbb{R}}
\newcommand{\BBT}{\mathbb{T}}

\usepackage{amsbsy}
\newcommand{\0}{\boldsymbol {0}}
\newcommand{\Ba}{\boldsymbol {a}}
\newcommand{\Bb}{\boldsymbol {b}}
\newcommand{\Bg}{\boldsymbol {g}}
\newcommand{\Bf}{\boldsymbol {f}}
\newcommand{\Bh}{\boldsymbol {h}}
\newcommand{\Bk}{\boldsymbol {k}}
\newcommand{\Bn}{\boldsymbol {n}}
\newcommand{\Bp}{\boldsymbol {p}}
\newcommand{\Bu}{\boldsymbol {u}}
\newcommand{\Bv}{\boldsymbol {v}}
\newcommand{\Bw}{\boldsymbol {w}}
\newcommand{\Bx}{\boldsymbol {x}}

\newcommand{\BR}{\boldsymbol {R}}
\newcommand{\BU}{\boldsymbol {U}}

\newcommand{\BW}{\boldsymbol {W}}
\newcommand{\BX}{\boldsymbol {X}}
\newcommand{\BY}{\boldsymbol {Y}}

\newcommand{\CA}{\mathcal{A}}
\newcommand{\CC}{\mathcal{C}}
\newcommand{\CD}{\mathcal{D}}
\newcommand{\CG}{\mathcal{G}}
\newcommand{\CH}{\mathcal{H}}
\newcommand{\CI}{\mathcal{I}}
\newcommand{\CK}{\mathcal{K}}
\newcommand{\CL}{\mathcal{L}}
\newcommand{\CR}{\mathcal{R}}
\newcommand{\CS}{\mathcal{S}}
\newcommand{\CT}{\mathcal{T}}
\newcommand{\CY}{\mathcal{Y}}
\newcommand{\CZ}{\mathcal{Z}}

\newcommand{\Fh}{\mathfrak{h}}

\newcommand{\Fp}{\mathfrak{p}}
\newcommand{\Fq}{\mathfrak{q}}

\newcommand{\FA}{\mathfrak{A}}
\newcommand{\FL}{\mathfrak{L}}
\newcommand{\FP}{\mathfrak{P}}

\newcommand{\ga}{\gamma}

\newcommand{\ep}{\varepsilon}

\newcommand{\Ga}{\Gamma}

\newcommand{\Om}{\Omega}

\newcommand{\dOm}{\dot \Omega}
\newcommand{\BuL}{\boldsymbol u_{_L} }

\newcommand{\tCA}{\widetilde{\mathcal A}}

\newcommand{\tE}{\widetilde{E}}
\newcommand{\tJ}{\widetilde{J}}
\newcommand{\tBf}{\widetilde{\boldsymbol f}}
\newcommand{\tBh}{\widetilde{\boldsymbol h}}
\newcommand{\tBk}{\widetilde{\boldsymbol k}}
\newcommand{\tg}{\widetilde{g}}
\newcommand{\tBu}{\widetilde{\boldsymbol u}}

\newcommand{\tBR}{\widetilde{\boldsymbol R}}
\newcommand{\tBuL}{\widetilde{\boldsymbol u}_{_L}}
\newcommand{\tBU}{\widetilde{\boldsymbol U}}
\newcommand{\tBW}{\widetilde{\boldsymbol W}}
\newcommand{\tW}{\widetilde{ W}}
\newcommand{\tBBD}{\widetilde{\mathbb{D}}}
\newcommand{\tBBH}{\widetilde{\mathbb{H}}}
\newcommand{\tBBK}{\widetilde{\mathbb{K}}}

\newcommand{\tp}{\widetilde{p}}
\newcommand{\tq}{\widetilde{q}}
\newcommand{\tQ}{\widetilde{Q}}

\newcommand{\tsigma}{\widetilde{\sigma}}

\newcommand{\oBf}{\overline{\boldsymbol f}}
\newcommand{\oBh}{\overline{\boldsymbol h}}
\newcommand{\oBk}{\overline{\boldsymbol k}}
\newcommand{\oBn}{\overline{\boldsymbol n}}
\newcommand{\oFn}{\overline{\mathfrak n}}

\newcommand{\og}{\overline{g}}

\newcommand{\oBR}{\overline{\boldsymbol R}}

\newcommand{\oBuL}{\overline{\boldsymbol u}_{_L}}

\newcommand{\Bnu}{\boldsymbol{\nu}}

\newcommand{\Balpha}{\boldsymbol {\alpha}}

\usepackage{mathrsfs}
\newcommand{\sA}{\mathscr{A}}
\newcommand{\sB}{\mathscr{B}}
\newcommand{\sE}{\mathscr{E}}

\usepackage{accents}
\newcommand{\ubar}[1]{\underaccent{\bar}{#1}}

\newcommand{\JDt}{\langle D_t \rangle}

\newcommand{\BUC}{\mathcal{BUC}}
\newcommand{\FpL}{\mathfrak{p}_{_L}}
\newcommand{\FqL}{\mathfrak{q}_{_L}}
\newcommand{\ET}{E_{_{(T)}}}

\def\d{\partial}

\usepackage{dsfont}%%for mathds
\numberwithin{equation}{section} 

\usepackage{lipsum}

\begin{document}
\title[Two-phase problem with free surface]{Some free boundary problem for two phase inhomogeneous incompressible flow}

\author{Hirokazu Saito}
\address[H.~Saito]{Faculty of Industrial Science and Technology, Liberal Arts\\
Tokyo University of Science\\
102-1 Tomino, Oshamambe-cho, Yamakoshi-gun, Hokkaido, 049-3514, Japan}
\email{hsaito@rs.tus.ac.jp}

\author{Yoshihiro Shibata}
\address[Y.~Shibata]{Waseda Research Institute for Science and Engineering\\ 
             Waseda University\\
 3-4-1 Ohkubo, Shinjuku-ku, Tokyo, 169-8555, Japan}
\address{Department of Mathematics\\
             Faculty of Science and Engineering\\ 
             Waseda University,Japan}
\address{Adjunct faculty member\\ 
 Department of Mechanical Engineering and Matherials Science, University of Pittsburgh, USA}   
\email{yshibata@waseda.jp}    
  
\author{Xin Zhang}
\address[X.~Zhang]{Waseda Research Institute for 
Science and Engineering\\ Waseda University\\
 3-4-1 Ohkubo, Shinjuku-ku, Tokyo, 169-8555, Japan}      
\email{xinzhang@aoni.waseda.jp}

\subjclass[2010]{Primary: 35Q30; Secondary: 76D05.}

\keywords{Two-phase problem, Inhomogeneous incompressible Navier-Stokes equations, $L_p-L_q$ maximal regularity, analytic semigroup}

%\thanks{}

%\date{}

%\dedicatory{}

\begin{abstract}
In this paper, we establish some local and global solutions for the two phase 
 incompressible inhomogeneous flows with moving interfaces in $L_p-L_q$ maximal regularity class. 
Compared with previous results obtained by V.A.Solonnikov and by Y.Shibata \& S.Shimizu, we find the local solutions in $L_p-L_q$ class in some general uniform $W^{2-1\slash r}_r$ domain in $\BBR^N$ by assuming  $(p, q)\in ]2,\infty[\times ]N,\infty[$ or  $(p, q)\in ]1,2[ \times ]N,\infty[$ satisfying $1 \slash p + N \slash q>3\slash 2.$ 
In particular, less regular initial data are allowed by assuming $p<2.$
In addition, if the density and the viscosity coefficient are piecewise constant, we can construct the long time solution from the small initial states in the case of the bounded droplet. This is due to some decay property for the corresponding linearized problem.
\end{abstract}

\maketitle

\renewcommand{\thefootnote}{\fnsymbol{footnote}}

\tableofcontents

\clearpage

%%%%%%%%%%%%%%%%%%%%%%%%%%%%
\section{Introduction}\label{sec:intro}
%%%%%%%%%%%%%%%%%%%%%%%%%%%%

\subsection{Description of the problem}
In this paper, we consider the following Cauchy problem in $\BBR^N$ ($N\geq 2$), 
\begin{equation}\label{eq:INS}\tag{$INS_{\pm}$}
	\left\{\begin{aligned}
		\pa_t(\rho \Bv)+\Di (\rho \Bv \otimes \Bv)  - \Di \BBT(\Bv,\Fp) = \rho \Bf  &\quad\mbox{in}\quad \dOm_t,   \\
		\pa_t \rho + \di (\rho \Bv) =0,  \,\,\,\di\Bv = 0    &\quad\mbox{in}\quad \dOm_t,  \\
		\jump{\BBT(\Bv,\Fp)\Bn_t} = \jump{\Bv}=\0,\,\,\, V_t=\Bv\cdot \Bn_t &\quad\mbox{on}\quad \Ga_t,  \\
		\BBT(\Bv_+,\Fp_+)\Bn_{+,t} = \0, \,\,\, V_{+,t}= \Bv_+\cdot \Bn_{+,t}&\quad\mbox{on}\quad \Ga_{+,t}, \\
				\Bv_- = \0 &\quad\mbox{on}\quad \Ga_{-}, \\
		(\rho, \Bv)|_{t=0}=(\rho_0, \Bv_0) &\quad\mbox{in}\quad \dOm,
	\end{aligned}\right.
\end{equation}
which describes the motion of two immiscible viscous incompressible liquids at time instant $t$ in some domain 
$\Omega_t:=\dOm_t \,\cup\, \Gamma_t := \Omega_{+,t} \,\cup\, \Omega_{-,t} 
\,\cup\,\Gamma_t$ 
surrounded by free surface $\Gamma_{+,t}$ and fixed boundary $\Gamma_-$ without taking the surface tension into account.
For simplicity, we adopt the notations 
$\Omega := \dOm \cup \Gamma
:= \Omega_+\cup \Omega_- \cup \Gamma$ for the initial domain with boundaries $\Gamma_\pm$ by dropping off the subscript $t=0.$

In fact, there are (at least) three typical physical situations as follows characterized by the model \eqref{eq:INS} (also see the figure below),
 but we will mainly concentrate on $(\Omega_1)$ for the long time issue later.
\begin{enumerate}
\item[$(\Omega_1)$:] $\Omega_t$ is some bounded droplet surrounded by the free surface $\Gamma_{+,t}$ with setting $\Gamma_- =\emptyset;$
\item[$(\Omega_2)$:] $\Omega_t$ is some bounded container with solid  boundary $\Gamma_-$ by assuming $\Gamma_{+,t} = \emptyset;$
\item[$(\Omega_3)$:] $\dOm_t = \Omega_{+,t} \,\cup\, \Omega_{-,t}$ stand for two infinite layers with some rigid bottom $\Gamma_-.$
\end{enumerate}
Besides, $\Bn_t$ and $\Bn_{+,t}$ are outwards unit normals subject to the moving interface $\Gamma_t$ between two bulks $\Omega_{\pm,t}$ and the free surface $\Gamma_{+,t}$ respectively.

\includegraphics[height=5cm]{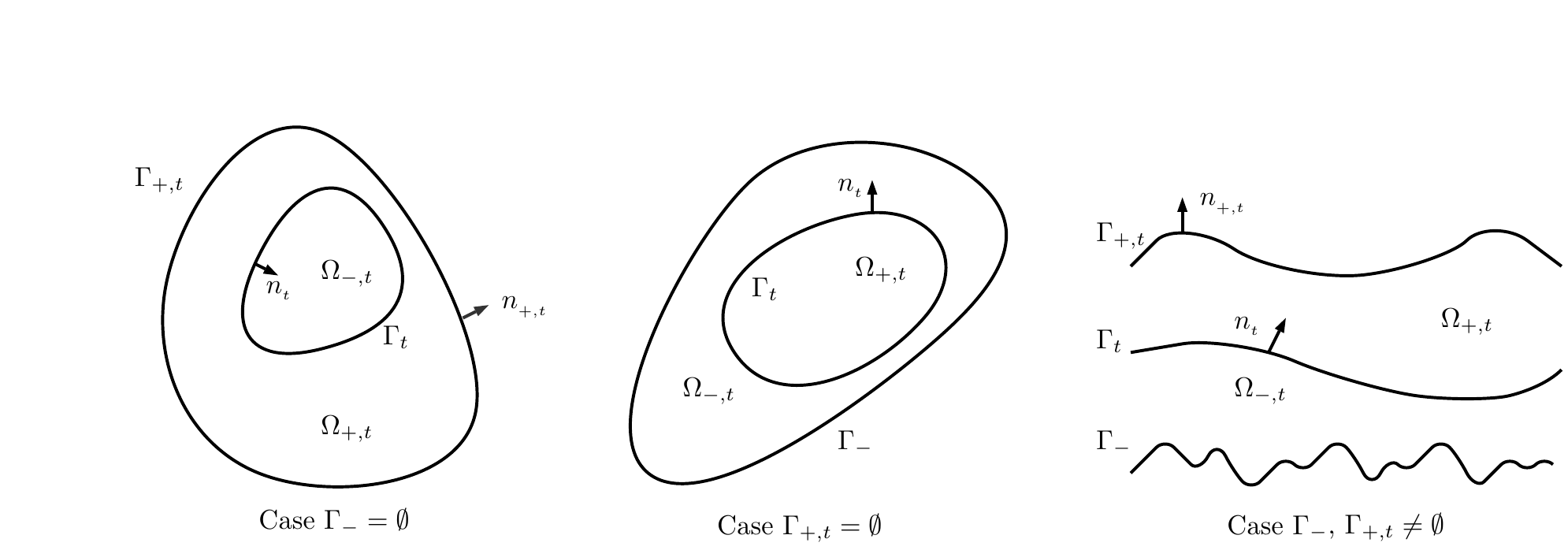}

With above settings on domains $\dOm_t,$ our aim is to determine the unknowns $(\rho,\Bv, \Fp,\dOm_t)$ in \eqref{eq:INS}: the density, the velocity field, the pressure and the domain with free interface, whenever the external force $\Bf$ and initial states $(\rho_0, \Bv_0)$ are given. In addition, the standard stress tensor $\BBT(\Bv,\Fq)$ is defined by
\begin{equation*}
\BBT(\Bv,\Fq) (x,t) : = \mu \big(\rho(x,t)\big) \BBD (\Bv) (x,t) -\Fq (x,t)\BBI,
\end{equation*} 
and the double deformation tensor $\BBD(\Bv)$ is given by
\begin{equation*}
\BBD(\Bv) := \nabla_{x}^{\top} \Bv + \nabla_x \Bv^{\top} \quad  \hbox{and} \quad  
       \big(\nabla^{\top}_{x} \Bv \big)^{j}_{k} = \big(\nabla_x \Bv^{\top} \big)^{k}_{j} := \pa_{x_k} v^{j}
        \quad \hbox{for}\quad j,k=1,...,N.
\end{equation*}
In \eqref{eq:INS}, the following standard notations are also utilized.
For any two vectors $\Bu, \Bv$ in $\BBR^N,$ the tensor product $\Bu \otimes \Bv$
stands for a $N\times N$ matrix with the $(j,k)$-entry 
$(\Bu \otimes \Bv)^j_k := u^j v^k$ ($1\leq j,k \leq N$). 
Additionally, for any $N\times N$ matrix $\BBA=\big(A^j_k(x)\big)_{N\times N},$ the quantity $\Di \BBA$ denotes an 
$N-$vector with $j^{\text{\tiny th}}$ component $(\Di \BBA)^j:=\sum_{k=1}^N \pa_{x_k} A^j_k.$ 
Lastly, the jump of the vector $\Bg$ 
across some surface $\CS$ is given by the following non-tangential limit
\begin{equation*}
\jump{\Bg}(x_0) := \lim_{\delta \rightarrow 0+} 
\Big(\Bg\big(x_0+\delta \Bnu (x_0)\big)
-\Bg\big(x_0-\delta \Bnu (x_0)\big)\Big)
\quad \forall\,\,\,x_0 \in \CS,
\end{equation*} 
where $\Bnu$ is the unit outward normal along the surface $\CS.$
Moreover, $V_t$ and $V_{+,t}$ stand for the normal velocities of $\Gamma_t$ and $\Gamma_{+,t}$ respectively.
\medskip

Although the goal in this paper is to investigate the solution of the two-phase model \eqref{eq:INS} within $L_p-L_q$ maximal regularity, 
our method below can be applied to the classical one-phase Navier-Stokes problem.
For convenience, we recall $\Omega$ within either of the following physical settings here:
\begin{enumerate}
\item[$(\Omega_4)$:] $\Omega$ is some moving (bounded) droplet without the solid boundary, i.e. $\Gamma_- = \emptyset;$
\item[$(\Omega_5)$:] $\Omega$ is some infinite layer with some finite-depth bottom $\Gamma_-.$
\end{enumerate}

Now let us review the history of the free boundary problem on the motion of viscous liquid for the cases $(\Omega_2)$- $(\Omega_5)$ before describing the main results in this context.  
The first breakthrough is due to \cite{Sol1977} by V.A. Solonnikov for the case $(\Omega_4)$, where the author first came up with Lagrangian coordinates approach and studied the classical solution of H\"older continuity. Indeed, V.A. Solonnikov in \cite{Sol1977} succeeded in establishing the short time existence of a unique solution in some bounded domain $\Omega_t$ with free surface $\Gamma_{+,t},$ as long as the given data satisfy
\begin{equation*}
\Bv_0 \in \CC^{2,\ep}(\Omega), \,\,\
\Gamma_+ \,\,\hbox{is of class}\,\, \CC^{2,\ep}
 \,\,\, \hbox{and}\,\,\,
\Bf, \nabla\Bf \in (\CC^{\ep\slash 2}_t \cap \CC^{\ep}_x )(\BBR^3\times]0,T[)
\end{equation*}  
for some $0<\ep <1$ and some $T >0.$ 
Later, Solonnikov in \cite{Sol1987} investigated the global solvability in the Sobolev space framework where he assumed that $\Gamma_+$ is $W^{2-1\slash r}_r$ regular for some $r>N.$
Compared with \cite{Sol1977,Sol1987}, V.A. Solonnikov studied the role of  
 of the surface tension in \cite{Sol1986,Sol1989}. 
\smallbreak

For the unbounded layer $(\Omega_5)$, J.T.Beale studied the (local and global) wellposedness issues within the $L^2$ framework in \cite{Bea1981} without taking surface tension into account, and in \cite{Bea1983} with surface tension involved.  Roughly speaking, the author in \cite{Bea1981,Bea1983} established the solutions from the initial state $\Bv_0$ in $W^s_2(\Omega)$ for some $s \in ]2,5\slash 2[.$ Furthermore, A.Tani in \cite{Tani1996} and A.Tani and N.Tanaka in \cite{TaniT1995} formulated the problem for $(\Omega_5)$ in fractional Sobolev-Slobodetski\v{\i} spaces. More recently, Y.Guo and I.Tice gave a series of works \cite{GuoT2013a,GuoT2013b,GuoT2013c} about the wellposedness issues and decay property of $(\Omega_5)$ based on the new energy method.
\smallbreak

Besides, for the bounded domain $(\Omega_2)$ occupied by two-phase inhomogeneous immiscible liquids, N.Tanaka in \cite{Tana1993,Tana1995} proved that the equilibrium state is stable in $L_2$ framework with including surface tension. Inspired by \cite{Tana1993,Tana1995},  L.Xu and Z.Zhang in \cite{XuZ2010} studied the double-layer case $(\Omega_3)$ with gravity additionally involved. 
However, for the piecewise constant density, the works \cite{Deni2007,DeniSol2011} by I.V.Denisova et al. implied the global solvability in H\"older space with or without surface for the domain $(\Omega_2).$
\smallbreak

Apart from the $L_2$ or H\"older framework, we also note that some recent contributions \cite{KPW2013,PruSim2010a,PruSim2009,PruSim2010b,PruSim2016} to the $L_p$ approach for two-phase problems by J. Pr\"uss and his collaborators, especially for case of the surface tension. For instance, the authors in \cite{PruSim2010a} showed that the interface between two immiscible liquids becomes instantaneously real analytic whenever the initial data lie in some Sobolev spaces $W^s_p$ for large enough $p.$ More recently, Y.Shibata in \cite{Shi2018b, Shi2018c} and H.Saito in \cite{Sai2018} studied one-phase problem including $(\Omega_4)$ and $(\Omega_5)$ in the $L_p-L_q$ framework with $2<p<\infty,$ $N<q<\infty$ and $2\slash p + N\slash q <1.$ 
\medskip

The solvability for the two-phase problem within $L_p-L_q$ maximal regularity is our main task, which will rely on the recent contributions \cite{ShiShi2011, MarSai2017} to the linearized model problem. In \cite{ShiShi2011, MarSai2017}, the authors employed the Multiplier Theorem  characterized by \emph{$\CR$-boundedness theory}  in \cite{Weis2001} to handle the resolvent problems. For $\CR$-boundedness theory, one may see \cite{DHP2003,KPW2013} for more discussions.

%%%%%%%%%%%%%%%%%%%%%%%%%%%%
\subsection{Main results}
%%%%%%%%%%%%%%%%%%%%%%%%%%%%
To state our main results in this context concerning \eqref{eq:INSL},
we firstly specify the assumptions on $\dot\Omega.$ 
\begin{defi}
We say that a connected open subset $\Omega$ in $\BBR^N$ ($N \geq 2$) is of class $W^{2-1\slash r}_r$ for some $1<r<\infty,$ if and only if for any point $x_0 \in \pa \Omega,$ one can choose a Cartesian coordinate system with origin $x_0$ (up to some translation and rotation) and coordinates $y=(y',y_{_N}):=(y_1,...,y_{_{N-1}},y_{_N}),$ as well as positive constants $\alpha, \beta, K$ and some $W^{2-1\slash r}_r$ function $h$ satisfying $\|h\|_{W^{2-1\slash r}_r} \leq K$ such that the neighborhood of $x_0$
\begin{equation*}
U_{\alpha,\beta, h}(x_0):= \{(y',y_{_N}): h(y')-\beta <y_{_N} < h(y')+\beta, |y'|<\alpha \}
\end{equation*}  
satisfies
\begin{equation*}
U^{-}_{\alpha,\beta, h}(x_0):= \{(y',y_{_N}): h(y')-\beta <y_{_N} < h(y'), |y'|<\alpha \}=\Omega \cap U_{\alpha,\beta, h}(x_0),
\end{equation*}
and $$\pa\Omega \cap U_{\alpha,\beta, h}(x_0)= \{(y',y_{_N}): y_{_N} = h(y'), |y'|<\alpha \}.$$
Above $\alpha, \beta, K, h$ may vary with respect to the different location on the boundary. 
Whenever the choices of $\alpha, \beta, K$ are independent of the position of $x_0,$ $\Omega$ is called uniform $W^{2-1\slash r}_r$ domain. Note that if the boundary $\pa \Omega$ is compact, then the uniformness is satisfied automatically.  Sometimes $\Omega$ is just called $W^{2-1\slash r}_r$ regular for simplicity. 
\end{defi}
\smallbreak

In general,  assume that $\Omega := \dOm \cup \Gamma
:= \Omega_\pm \cup \Gamma$ in $\BBR^N$ surrounded by two sharp surfaces $\Gamma_\pm$ are of uniform $W^{2-1\slash r}_r$ class hereafter.
For such domain, recall the unique solvability of the so-called \emph{weak elliptic transmission problem}, which plays a fundamental role in our results. To be exact, let us introduce some functional spaces and notations.
Firstly, $W^1_q(\Omega)$ and $ \widehat{W}^1_{q,\Gamma_+}(\Omega)$ are standard (nonhomogeneous and homogeneous) Sobolev spaces. Namely,
\begin{align*}
W^1_q(\Omega) &:= \{f\in L_q(\Omega): \|f\|_{W^1_q(\Omega)}
:=\|f\|_{L_q(\Omega)} + \|\nabla f\|_{L_q(\Omega)} <\infty \},\\
\widehat{W}^1_q(\Omega) &:= \{f\in L_{q,loc}(\Omega): 
\|f\|_{\widehat{W}^1_q(\Omega)}:=\|\nabla f\|_{L_q(\Omega)} <\infty \}.
\end{align*}
Next, the linear space $X^{1}_{q,\Gamma_+} (\Omega)$ for any $1<q<\infty$ is defined as below,
\begin{equation*}
X^{1}_{q,\Gamma_+} (\Omega):= 
\begin{cases}
\{f \in X^1_q(\Omega): f= 0\,\,\, \hbox{on}\,\,\, \Gamma_+ \}& \hbox{if}\,\,\, \Gamma_+ \not= \emptyset, \\
X^1_q(\Omega) & \hbox{if}\,\,\, \Gamma_+ = \emptyset,
\end{cases}
\end{equation*}
with the word $X \in \big\{W, \widehat{W}\big\}$ and $\|f\|_{X^{1}_{q,\Gamma_+} (\Omega)} := \|f\|_{X^{1}_{q} (\Omega)}.$ 
Moreover, for any vectors $\Bu$ and $\Bv$ defined in any domain
$G \subset \BBR^N,$ set $(\Bu,\Bv)_{G}:=\int_{G} \Bu\cdot \Bv \,dx
 = \sum_{j=1}^N\int_{G} u^j v^j \,dx.$
\begin{defi}\label{def:welliptic}
Consider some domain 
$\Omega$ as above with $\Gamma_+ \not= \emptyset$ and suppose that the step function 
$\eta:= \eta_+ \mathds{1}_{\Omega_+} +\eta_-\mathds{1}_{\Omega_-} $ for any constants $\eta_\pm>0.$
Then we say that the weak elliptic transmission problem is uniquely solvable on $\widehat{W}^1_{q,\Gamma_+}(\Omega)\,(1<q<\infty)$ for $\eta$ if the following assertions hold true:
For any $\Bf \in L_q(\Omega)^N,$ there is a unique $\theta \in \widehat{W}^1_{q,\Gamma_+}(\Omega)$ satisfying variational equations 
 as below,
\begin{equation*}
(\eta^{-1} \nabla \theta, \nabla \varphi)_{\dOm} = (\Bf,\nabla \varphi)_{\Omega}, 
\quad \hbox{for all} \,\,\, \varphi \in\widehat{W}^1_{q',\Gamma_+}(\Omega).
\end{equation*}
Moreover, there exists a constant $C$ independent on the choices of $\theta$, $\varphi$ and $\Bf$ such that 
$$\|\nabla \theta\|_{L_q(\Omega)} \leq C\|\Bf\|_{L_q(\Omega)}.$$
\end{defi}

Now let us give some comments on Definition \ref{def:welliptic}, which will shed light on the construction of Stokes operator for two phase problem later.
\begin{rema}\label{rmk:welliptic}
For some $1<q<\infty,$ we write 
\begin{equation*}
W^1_q(\dOm)+\widehat{W}^1_{q,\Gamma_+}(\Omega):=\{\theta_1+\theta_2:\theta_1 \in W^1_q(\dOm) 
\,\,\,\hbox{and}\,\,\,
\theta_2 \in \widehat{W}^1_{q,\Gamma_+}(\Omega)\}.
\end{equation*}
Suppose that the week elliptic transmission problem is uniquely solvable on $\widehat{W}^1_{q,\Gamma_+}(\Omega)$ for $\eta_\pm$ and $\Omega$ as in Definition \ref{def:welliptic}. Then for any 
$(\Balpha,\beta,\gamma) \in L_q(\dOm)^N 
\times W^{1-1\slash q}_q (\Gamma) \times W^{1-1\slash q}_q (\Gamma_+)$ 
and any test function $\varphi \in \widehat{W}^1_{q',\Gamma_+}(\Omega),$ there exists a unique $\theta \in W^1_q(\dOm)+ \widehat{W}^1_{q,\Gamma_+}(\Omega)$ satisfying
\begin{equation*}
(\eta^{-1}\nabla \theta, \nabla \varphi)_{\dOm} = (\Balpha, \nabla \varphi)_{\dOm},\quad \jump{\theta} =\beta \,\,\,\hbox{on}\,\,\,\Gamma \quad\hbox{and}\quad 
\theta = \gamma \,\,\,\hbox{on}\,\,\,\Gamma_+.
\end{equation*} 
In addition, there is positive constant $C$ independent on the choices of $\Balpha,$ $\beta,$ $\gamma$ and $\varphi$ such that,
\begin{equation*}
\|\nabla \theta\|_{L_q(\dOm)} \leq C \Big(\|\Balpha\|_{L_q(\dOm)} 
+ \|\beta\|_{W^{1-\frac{1}{q}}_q (\Gamma)}
+ \|\gamma\|_{W^{1-\frac{1}{q}}_q (\Gamma_+)}\Big).
\end{equation*}
For brevity, we write $\theta := \CK(\Balpha,\beta,\gamma),$ which satisfies above properties. 
\end{rema}
\smallbreak

By above discussions, let us summarize our hypotheses to investigate \eqref{eq:INS} as follows.
\begin{enumerate}
\item[$(\CH 1)$] The domain $\dOm$ is uniformly  $W^{2-1\slash r}_r$ regular for some $r>N$, i.e. $\Om_\pm$ are uniform $W^{2-1\slash r}_r$ domains; 
\item[$(\CH 2)$] The weak Elliptic transmission problem is uniquely solvable on $\widehat{W}^1_{q,\Gamma_+} (\Om)$ and $\widehat{W}^1_{q',\Gamma_+}(\Om)$ for $\eta_\pm>0$ and $q \in ]1,\infty[.$
\item[$(\CH 3)$] $\mu\big(\rho_0(x)\big)$ is a strictly positive function on $\dOm$ satisfying 
\begin{equation*}
\ubar{\mu}_{+} \mathds{1}_{\Omega_+}+\ubar{\mu}_{-}\mathds{1}_{\Omega_-}
  \leq \mu\big(\rho_0(\cdot)\big) \leq
  \bar{\mu}_{+} \mathds{1}_{\Omega_+}+  \bar{\mu}_{-}  \mathds{1}_{\Omega_-},
\end{equation*}
where the constants $ \ubar{\mu}_{\pm},$ $ \bar{\mu}_{\pm}>0.$ In addition, assume that $\mu \in C^{1}(\BBR_+;\BBR_+)$ and $r$ is given as in $(\CH 1).$
\end{enumerate}
\medskip

By assuming $(\CH 1)- (\CH 3),$ let us outline the main strategy to handle \eqref{eq:INS}.
Motivated by the pioneering work \cite{Sol1977} by V.A.Solonnikov, we shall take advantage of the so-called \emph{Lagrangian coordinates}.  Moreover, the fact that the surfaces $\Gamma_t,$ $\Gamma_{+,t}$ and $\Gamma_-$ consist of the exactly same fluid particles at all time instants $t,$ is taken for granted. 
Indeed, if we denote $\Bu(\xi, t) := \Bv\big( \BX_{u}(\xi, t) ,t\big)$ and consider
\begin{equation}\label{eq:Lagrange}
\BX_{u}(\xi ,t)  := \xi +\int^t_0 \Bu (\xi , \tau) \,d \tau 
\quad \hbox{for all}\,\,\, \xi  \in \Om \cup \Gamma_-, 
\end{equation}
then $(\Gamma_t,\Gamma_{+,t},\Gamma_-) = \BX_u\big((\Gamma_t,\Gamma_{+},\Gamma_-),t\big).$
In other words, the unknown regions $\Omega_{\pm,t}$ are the image of $\Omega_\pm$ respectively under the transformation $\BX_u(\cdot,t).$ 
Then, to rewrite \eqref{eq:INS} by \eqref{eq:Lagrange}, we adopt the following notations here and subsequently.
\begin{itemize}
\item For any $\CC^1$ vector $\BY(\xi)$ defined in $\dOm,$  write $\nabla_{\xi}^{\top}\BY$ for the Jacobian matrix of $\BY,$ i.e. $(\nabla_{\xi}^{\top} \BY)^j_k := \pa_{\xi_k} Y^j$ for $1\leq j,k\leq N,$ and
$\nabla_\xi \BY^{\top}:=(\nabla_{\xi}^{\top} \BY)^{\top}.$
\item  For simplicity, $\sA_u$ stands for the cofactor matrix of $\nabla_{\xi}^{\top} \BX_u.$
Moreover, the derivatives and stress tensors related to \eqref{eq:Lagrange} are given by
\begin{equation*}
\nabla_{u } := \sA_u \nabla_{\xi }, \,\,\, 
\di_{u} = \Di_{u}:= \nabla_{u }\cdot
\,\,\,\hbox{and}\,\,\,
 \BBT_{u }(\Bw , \Fq) := \mu \big(\rho_0 (\xi)\big) \BBD_{u}(\Bw ) - \Fq  \BBI.
\end{equation*}
Above $\BBD_{u}(\Bw):=\nabla_{\xi}^{\top} \Bw \cdot \sA_u^{\top} + \sA_u \cdot \nabla_\xi \Bw^{\top}$ for any smooth $\Bw$ and $\Fq.$ 
\item Suppose that $\Bn$ and $\Bn_+$ are the unit normal for $\Ga$ and $\Ga_+$ respectively. Set that
\begin{equation*}
(\oBn, \oBn_+)(\xi,t)
:=(\Bn_t, \Bn_{+,t} )\big(\BX_u (\xi,t)\big)
= \Big( \frac{\sA_{u} \Bn}{|\sA_{u}\Bn|},
\frac{\sA_{u_+} \Bn_+}{|\sA_{u_+}\Bn_+|}\Big)(\xi,t),
\quad \forall \, \xi \in \Gamma \cup \Gamma_+.
\end{equation*}
\item For any vector $\Bnu$ and $\Bh$ defined along some surface $\CS,$ introduce the operator 
\begin{equation*}
\CT_{\Bnu}\Bh := \Bh-(\Bh\cdot \Bnu)\Bnu,
\end{equation*}
which is a projection into the hyperplane orthogonal to $\Bnu.$
\end{itemize}

Now the mass conservation law in \eqref{eq:INS} is reduced to  $\rho\big( \BX_{u }(\xi,t) ,t\big) = \rho_0(\xi)$ by \eqref{eq:Lagrange}, and set that  $\Fq(\xi, t) := \Fp\big( \BX_{u }(\xi,t) ,t\big).$
Then it is not hard to verify that $( \Bu, \Fq)$ satisfies the following equations 
\begin{equation}\label{eq:INSL}\tag{$INS_{\pm}^{\FL}$}
	\left\{\begin{aligned}
\rho_{0}\pa_t\Bu- \Di_u \BBT_u (\Bu ,\Fq ) = \rho_0 \Bf \big( \BX_{u }(\xi, t) ,t\big), \,\,\, \di_{u} \Bu  = 0               
                        &\quad\mbox{in}\quad \dOm  \times ]0,T[, \\
		\jump{\BBT_u(\Bu,\Fq) \oBn} =\jump{\Bu}= \0  
		                &\quad\mbox{on}\quad \Ga\times ]0,T[,  \\
	    \BBT_{u_+}(\Bu_+,\Fq_+) \oBn_+  = \0  
	                    &\quad\mbox{on}\quad \Ga_+\times ]0,T[,  \\		
	   \Bu_- = \0  
	                    &\quad\mbox{on}\quad \Ga_{-}\times ]0,T[, \\
		\Bu|_{t=0}=\Bv_{0} 
		                &\quad\mbox{in}\quad \dOm.
	\end{aligned}\right.
\end{equation}
From the incompressibility condition $\eqref{eq:INSL}_2,$ we remark that $\sA_u = (\nabla_\xi \BX_u^{\top})^{-1}$ due to $\det (\nabla_{\xi}^{\top} \BX_u) =1$ and \emph{Liouville Theorem}.  
As our method for local solvability issue could be applied to the domain $(\Omega_3)$, we would like to keep the boundary condition on $\Gamma_-$ for a while.
\medskip

To seek suitable functional space for $\Bv_0$ in \eqref{eq:INSL}, recall the linear mapping $\CK$ in Remark \ref{rmk:welliptic}. For any $1<q<\infty$ and any vector $\Bu \in W^2_q(\dOm)^N,$  consider 
\begin{align*}
\Balpha_u &:=\eta^{-1} \Di \big(\mu \BBD(\Bu)\big)-\nabla \di \Bu,\\
\beta_u &:= \jump{\mu \BBD(\Bu)\Bn} \Bn - \jump{\di \Bu},\\
\gamma_u &:= \big(\mu \BBD(\Bu)\Bn_+\big)\Bn_+ -\di \Bu,
\end{align*}
and write $K(\Bu):= \CK(\Balpha_u,\beta_u,\gamma_u) \in W^1_q(\dOm)+ \widehat{W}^1_{q,\Gamma_+}(\Omega)$ for short. 
\smallbreak

Next, keeping $K(\Bu)$ in mind,  we introduce \emph{Stokes} operator for two phase problem
$$\CA_q \Bu := \eta^{-1} \Di \BBT\big(\Bu,K(\Bu)\big),$$  whose domain $\CD (\CA_q)$ is given by 
\begin{align*}
\CD (\CA_q) := \{\Bu \in W^{2}_q (\dOm)^N \cap J_q (\dOm) :\,\,
 & \jump{\Bu}|_{\Ga}=\jump{\CT_{\Bn} \big(\mu \BBD (\Bu) \Bn\big)}|_{\Ga}  =  \0,\\
&\CT_{\Bn_+} \big(\mu \BBD (\Bu) \Bn_+\big)|_{\Ga_+}
 =\0
 \quad \hbox{and}\quad  \Bu|_{\Ga_-}=0\}.
\end{align*}
Above the hydrodynamic Lebesgue space $J_q (\dOm) := \{ \Bf \in L_q(\Om)^N: (\Bf, \nabla \varphi )_{\dOm}=0 \quad \forall \varphi \in \widehat{W}^{1}_{q',\Gamma_+}(\Om)\}.$
Additionally, recall the real interpolation functor (see Section \ref{sec:fs}) and set
\begin{equation*}
\CD^{2-2\slash p}_{q,p}(\dOm) := \big(J_q(\dOm), \CD(\CA_q) \big)_{1- 1\slash p,p}
\,\,\,\hbox{and}\,\,\,
W^{2,1}_{q,p}(\dOm\times I) := L_p\big(I; W^2_q(\dOm)^N\big) \cap W^1_p\big(I;L_q(\dOm)^N\big)
\end{equation*}
for any $1<p<\infty$ and some time interval $I \subset \BBR.$
More discussions on $\CD^{2-2\slash p}_{q,p}(\dOm)$ are postponed to Appendix \ref{appendix:interpolation}.
Now our main result upon the local solvability of System \eqref{eq:INSL} for the cases $(\Omega_1)-(\Omega_3)$ reads as follows.
\begin{theo}\label{thm:main_local}
Let $(p,q)$ be in $(I) \cup (II)$ with the sets $(I)$ and $(II)$ given by 
\begin{equation*}
(I) :=\{(p,q)\in ]2,\infty[\times ]N,\infty[\}
\quad \hbox{and}\quad 
(II) :=\{(p,q)\in ]1,2[ \times ]N,\infty[\,: 1\slash p + N\slash q  >3\slash 2\}.
\end{equation*} 
Additionally, hypotheses $(\CH1)-(\CH3)$ are fulfilled for some $\eta:= \eta_+ \mathds{1}_{\Omega_+}+ \eta_- \mathds{1}_{\Omega_-}$ $(\eta_\pm>0)$ and $r \geq q.$
Assume that $\rho_0 \in \wh W^1_q(\dOm),$
$\Bv_0 \in \CD^{2-2\slash p}_{q,p}(\dOm)$ and $\Bf \in L_p\big(0,2;W^1_\infty(\BBR^N)^N\big).$ 
If, in addition, $\|\rho_0 -\eta\|_{L_\infty(\dOm)} \leq c$ for some constant $c \ll 1,$ then there are some time instant $T(<1)$ and constant $C,$ only depending on $p,$ $q,$ $\Bv_0$ and $\Bf,$ such that the System \eqref{eq:INSL} admits a unique solution $(\Bu,\Fq)$ satisfying
\begin{equation}\label{eq:uq_est}
\|\Bu\|_{W^{2,1}_{q,p}(\dOm\times ]0,T[)} 
+\|\nabla \Fq\|_{L_p(0,T;L_q(\dOm))} \leq C.
\end{equation} 
In addition, if $\mu$ is piecewise constant, we can relax the constrain $\rho_0 \in \wh W^1_q(\dOm)$ to $\rho_0 \in L_\infty(\dot\Omega).$
\end{theo}
\medskip

For the case of $(\Omega_1),$ the hypothese $(\CH 2)$ is fulfilled for any  
$\eta := \eta_+ \mathds{1}_{\Omega_+} + \eta_- \mathds{1}_{\Omega_-}$ ($\eta_\pm >0$)
due to \cite{Shi2016} by Y.Shibata. Our second result is about the unique long time solution of \eqref{eq:INSL} in the case of $(\Omega_1).$ 
To this end, let us introduce the rigid motion space
\begin{equation*}
\CR_d := \{ \Bp(x) = \BBA x + \Bb: \,\BBA\, 
\hbox{is an}\, N\times N \,\hbox{anti-symmetric matrix and}\, \, 
\Bb \in \BBR^N\}.
\end{equation*}
As $\BBA$ is anti-symmetric, it is easy to verify that 
$\di \Bp =0 \,\,\,\hbox{and}\,\,\, \BBD (\Bp)=0$ for any $\Bp \in \CR_d.$
Without loss of generality,  set $M:= \dim \CR_d  \in \BBN$ and then there exist a basis 
\begin{equation*}
\FP:=\{\Bp_\alpha \in \CR_d :  (\eta \, \Bp_\alpha, \Bp_\beta)_{\dOm} = \delta^{\alpha}_{\beta}, \,\,\,
\hbox{for any}\,\, 1 \leq \alpha, \beta \leq M\},
\end{equation*}
such that $\CR_d := span\{\,\Bp_\alpha \in \FP\}.$ 

\begin{theo}\label{thm:main_global}
Let $(p,q) \in (I) \cup (II)$ as in Theorem \ref{thm:main_local} and $\Omega$ be a bounded 
$W^{2-1\slash r}_r$ ($r \geq q$) domain satisfying $(\Omega_1).$
Assume that  $\rho_0 (\xi) =\eta = \eta_+ \mathds{1}_{\Omega_+} +\eta_- \mathds{1}_{\Omega_-} $ and 
$\mu = \mu_+ \mathds{1}_{\Omega_+} +\mu_- \mathds{1}_{\Omega_-} $ are piecewise constant for any $\eta_\pm, \mu _\pm>0.$ If $\| \Bv_0 \|_{\CD^{2-2\slash p}_{q,p}(\dOm)} \ll 1$ such that 
$(\eta \Bv_0, \Bp_\alpha )_{\dOm} =0$ for any $\Bp_\alpha \in \FP,$
then \eqref{eq:INSL} admits a unique global solution $(\Bu, \Fq).$
Moreover, there exists constant $\ep_0$ and $C$ such that
\begin{equation*}
\|e^{\ep_0 t} \Bu \|_{W^{2,1}_{q,p}(\dOm\times ]0,T[)} 
+\|e^{\ep_0 t}\nabla \Fq\|_{L_p(0,T;L_q(\dOm))} \leq C \| \Bv_0 \|_{\CD^{2-2\slash p}_{q,p}(\dOm)} 
\,\,\,\hbox{for any}\,\, T > 0.
\end{equation*} 
\end{theo}

\begin{rema} Now let give some comments on Theorem \ref{thm:main_global}.
\begin{enumerate}

\item Firstly, let us indicate that above results on \eqref{eq:INSL} yield the solvability for the system \eqref{eq:INS} in Euler coordinates. For instance, we consider the short time result in Theorem \ref{thm:main_local}.
Thanks to \eqref{eq:uq_est}, $\BX_u$ in \eqref{eq:Lagrange} is well defined as long as $T$ is small enough. Moreover, $\BX_u$ is a $\CC^1$ diffeomorphism from $\dOm$ onto $\dOm_t$ and measure preserving. 
Denote $\BX_u^{-1}.$ the inverse mapping of $\BX_u$ For any smooth function $h$ over $\dOm$ and $r \in [1,\infty],$ we have 
\begin{equation*}
\|h \cdot \BX_u^{-1}\|_{L_r(\dOm_t)} \lesssim \|h\|_{L_r(\dOm)}.
\end{equation*}
Recall that $(\rho, \Bv, \Fp) = (\rho_0, \Bu,\Fq)\circ \BX_u^{-1}$ and $\sA_u:= (\nabla_{\xi}^{\top} X_u)^{-1}.$ 
Then by Lemma \ref{lem:A_DA}, \eqref{eq:uq_est} and  
\begin{equation*}
\nabla_x (\rho, \Bv, \Fp) = \big(\sA_u \nabla_\xi(\rho_0,\Bu,\Fq)\big)\circ \BX_u^{-1}, 
\end{equation*}
we obtain that the components of $\nabla_x (\rho, \Bv, \Fp)$ belong to $L_p\big(0,T;L_q(\dOm_t)\big).$ 
On the other hand, we can bound $\d_t \Bv$ in $L_p\big(0,T;L_q(\dOm_t)\big)^N$ according to
$\d_t \Bv = (\d_t \Bu)\circ \BX_u^{-1}
  + (\Bu \circ \BX_u^{-1}) \cdot\nabla_x \Bv. $
Thus one can verify that the second order derivatives of $\Bv$ are bounded in 
$L_p\big(0,T;L_q(\dOm_t)\big)^N$ by using the momentum equations $\eqref{eq:INSL}_1.$
For simplicity, we omit the statements of the exact theorems concerning \eqref{eq:INS} here.

\item In fact, the index set $(II)$ allows us to handle less regular initial states due to the definition of $\CD^{2-2\slash p}_{q,p}(\dOm).$ On the other hand, compared with the recent contributions \cite{DanM2012, DanZhx2017b, LiaoZh2016, LiaoZh2016b, GG2018} on the density patch problem for inhomogeneous Navier-Stokes system, we prove that the global density patches (in Eulerian coordinates) with imposing the free boundary conditions and \emph{almost} critical initial velocity in $L_p-L_q$ framework.

\end{enumerate}
\end{rema}

 The rest of the paper unfolds as follows. In next section, we will recall the notations of functional spaces and linear estimates. Then in Section \ref{sec:local}, we will prove Theorem \ref{thm:main_local}. Section \ref{sec:decay} is devoted some decay property, which is important for our discussions on Theorem \ref{thm:main_global} in Section  \ref{sec:global}.

%%%%%%%%%%%%%%%%%%%%%%%%%%%%%%%%%%%%%%%%%%%%%%%%%%%%%%%%%%%%%%%%%%%%%%%%%%
%%%%%%%%%%%%%%%%%%%%%%%%%%%%%%%%%%%%%%%%%%%%%%%%%%%%%%%%%%%%%%%%%%%%%%%%%%
\section{Functional spaces and some linear estimates}\label{sec:fs}

\subsection{Functional spaces}
In this part,  we shall introduce some functional spaces used throughout this paper. For any domain $G \subset \BBR^m$ ($1\leq m\in \BBN$) and some Banach space $E,$  $W^k_p(G;E)$ stands for the standard $E-$valued Sobolev space
for $1 \leq p \leq \infty$ and $k\in \BBN.$ Whenever $E$ coincides with $\BBR$ or $\BBC,$ we just write $W^k_p(G)$ for the collection of scalar-valued Sobolev functions. Moreover, the $L_p(G)$ ($L_p(G;E)$) stands for usual ($E-$valued) Lebegue spaces. In addition, similar conventions for $W^k_{p,loc}(G;E)$ and the homogeneous spaces $\widehat{W}^k_p(G;E)$ are admitted.  
For some open interval $I$ in $\BBR$ and some constant $\gamma>0,$ we define the following exponentially weighted Lebesgue and Sobolev spaces,
\begin{align*}
L_{p, \ga}(I; E) &:= \{ f: I \rightarrow E : e^{-\ga t} f \in L_p (I; E)\},
\\
L_{p, 0}(\BBR; E) &:= \{ f \in L_{p} (\BBR; E):  f(\cdot, t) = 0\,\, \hbox{for}\,\, t<0 \},\\
L_{p, 0,\ga}(\BBR; E) &:= \{ f \in L_{p,\gamma} (\BBR; E):  f(\cdot, t) = 0\,\, \hbox{for}\,\, t<0 \},\\
W^m_{p, \ga}(I; E) &:= \{ f \in L_{p,\gamma} (I; E) : \pa_t^j f (\cdot, t)\in L_{p,\gamma} (I; E), \,\,  1\leq j \leq m \}, \\
W^m_{p, 0,\ga}(\BBR; E) &:= W^m_{p, \ga}(\BBR; E) \cap L_{p, 0,\ga}(\BBR; E). 
\end{align*}
Moreover, the norm of $W^m_{p,0,\ga}(\BBR; E)$ with $m\geq 0$ is given by
\begin{equation*}
\|f\|_{W^m_{p,0,\ga}(\BBR; E)} := \sum_{0 \leq j \leq m} \|e^{-\ga t}\pa_t^j f (\cdot, t)\|_{L_{p} (\BBR; E)}.
\end{equation*}
\smallbreak

Recall the notion of \emph{Japanese} bracket $\langle y \rangle:= (1 + \left| y \right|^2)^{1\slash 2}$ for any $y \in \BBR^m$ and $\langle D_y \rangle (:= (I-\Delta_y)^{1\slash 2})$ denotes the Fourier multiplier whose symbol is $\langle y \rangle.$ By such kind of multiplier, the standard and weighted Bessel potential spaces are defined for $s \geq 0$ as below,
\begin{align*}
H^{s}_{p}(\BBR; E):=& \{ f\in  L_{p}(\BBR; E):  (\langle D_t \rangle^s  f)(\cdot, t) \in L_p (\BBR; E) \}, \\
H^{s}_{p,0,\gamma}(\BBR; E):=& \{ f\in  L_{p,0,\gamma}(\BBR; E): e^{-\gamma t} (\langle D_t \rangle^s  f)(\cdot, t) \in L_p (\BBR; E) \}.
\end{align*}
The norm of $H^{s}_{p}(\BBR; E)$ is given by $\|f\|_{H^{s}_{p}(\BBR; E)}
:=\|\langle D_t \rangle f\|_{L_p (\BBR; E)}$
and the norm $\|\cdot\|_{H^{s}_{p,0,\gamma}(\BBR; E)}$ is defined similarly.
Furthermore, we introduce the following mixed derivative spaces for $ 1 \leq  q \leq \infty,$
\begin{align*}
W^{2,1}_{q,p}(G \times I) :=& L_p\big(I; W^2_q(G)\big) \cap W^1_p\big(I;L_q(G)\big),\\
W^{2,1}_{q,p,0,\gamma}(G \times \BBR) :=& L_{p,0,\gamma}\big(\BBR; W^2_q(G)\big) \cap W^1_{p,0,\gamma}\big(\BBR;L_q(G)\big),\\
H^{1,1\slash 2}_{q,p}(G \times \BBR):=& L_p\big(\BBR; W^1_q(G)\big) \cap H^{1\slash 2}_p \big(\BBR; L_q(G)\big),\\
H^{1,1\slash 2}_{q,p,0}(G \times \BBR):=& L_{p,0} \big( \BBR; L_q(G)\big) \cap  H^{1,1\slash 2}_{q,p}(G \times \BBR) ,\\
H^{1,1\slash 2}_{q,p,0,\gamma}(G \times \BBR):=& L_{p,0,\gamma}\big(\BBR; W^1_q(G)\big) \cap H^{1\slash 2}_{p,0,\gamma} \big(\BBR; L_q(G)\big).
\end{align*}
Here for any Banach spaces $E_0$ and $E_1$ embedded in tempered distribution space, the norm of $E_0 \cap E_1$ is understood by $\|\cdot \|_{E_0 \cap E_1} := \|\cdot\|_{E_0} +\|\cdot\|_{E_1}.$ Besides, $\CL(E_0, E_1)$ is the family of all bounded linear mapping from $E_0$ to $E_1.$ Moreover, we also employ the notations $(E_0, E_1)_{\theta, p}$ and  $(E_0, E_1)_{[\theta]}$ for the real and complex interpolation functors between the (interpolation) couple $E_0$ and $E_1$ respectively for any $\theta \in ]0,1[$ and $p\in [1,\infty]$ (for more details see \cite{BerLof1976}). 
\smallbreak

\subsection{Linear estimates}
The study of \eqref{eq:INSL} is based on the following linear two phase Stokes equtions in some fixed domain 
$\Om = \dOm \cup \Gamma
= \Omega_+ \cup \Omega_- \cup \Gamma,$
\begin{equation}\label{eq:S}\tag{$S_{\pm}$}
	\left\{\begin{aligned}
	  \pa_t\Bu - \eta^{-1}\Di \BBT(\Bu,\Fq) = \Bf,\,\,\,
         \di\Bu = g=\di \BR  		
		  &\quad\mbox{in}\quad \dOm \times ]0,T[,  \\
		\jump{\BBT(\Bu,\Fq)\Bn} = \jump{\Bh}, \,\,\,\jump{\Bu}=\0 
		&\quad\mbox{on}\quad \Gamma \times ]0,T[,  \\
		\BBT_+(\Bu_+,\Fq_+)\Bn_+ = \Bk &\quad\mbox{on}\quad \Ga_{+}\times ]0,T[, \\
				\Bu_- = \0 &\quad\mbox{on}\quad \Ga_{-}\times ]0,T[, \\
	\Bu|_{t=0}= \Bu_0 &\quad\mbox{in}\quad \dOm.
	\end{aligned}\right.
\end{equation}
where $\eta := \eta_+ \mathds{1}_{\Om_+} +\eta_- \mathds{1}_{\Om_-}$ for some $\eta_{\pm} >0.$ 
In \eqref{eq:S}, we assume that the viscosity coefficient $\mu$ satisfies $(\CH 3').$ Namely,
\begin{enumerate}
\item[$(\CH 3')$] For some $r$ given in $(\CH 1),$  take some $\mu(\cdot)$ in $W^{1}_{r}(\dOm)$ satisfying 
\begin{equation*}
\ubar{\mu}_{+} \mathds{1}_{\Omega_+}+\ubar{\mu}_{-}\mathds{1}_{\Omega_-}
  \leq \mu(\cdot) \leq
  \bar{\mu}_{+} \mathds{1}_{\Omega_+}+  \bar{\mu}_{-}  \mathds{1}_{\Omega_-},
\end{equation*}
with constants $ \ubar{\mu}_{\pm},$ $  \bar{\mu}_{\pm}>0.$ 
\end{enumerate}
\smallbreak

Now we denote the space 
\footnote{ Compared with the \cite[Sec. 1.2]{MarSai2017}, our definition of $\BW^{-1}_q(\Omega)$ here is slightly more  general in order to handle the domains with the exterior bulk.} 
 $\BW^{-1}_q(\Omega)$ for $1 <q < \infty$ by
\begin{equation*}
\BW^{-1}_q(\Omega):= \{ g \in L_q(\Omega) : \,\exists\, \BR \in L_q(\Omega)^N \,\,\, \hbox{such that}\,\, (g,\varphi)= - (\BR, \nabla \varphi)_{\Omega}, \,\, \forall \varphi \in W^{1}_{q',\Gamma_+}(\Omega)\}.
\end{equation*}
Moreover, we say such $\BR$ above is in $\CG(g)$ for some $g \in \BW^{-1}_q(\Omega).$
Then \cite[Theorem 2.8]{MarSai2017} concerning \eqref{eq:S} reads as follows.  
\begin{theo} \label{thm:stokes}
Let $(p,q) \in ]1,\infty[^2$ and $r \geq \max\{q, q'\}$ with $ q' := q \slash (q-1).$ 
Assume that $(\CH 1)$, $(\CH 2)$ and $(\CH 3')$ are fulfilled.
Then there exists some constants $\gamma_0 \geq 1$ and $C_{p,q,\ga_0}$  such that the following assertions hold true by taking $T= \infty$ in \eqref{eq:S}.
\begin{enumerate}
\item   For any $\Bu_0\in \CD^{2-2\slash p}_{q,p}(\dOm)$ and $(\Bf, \BR, \Bh,\Bk)= \0 \in \BBR^{4N},$ \eqref{eq:S}  admits a unique solution 
\begin{equation*}
(\Bu, \Fq) \in W^{2,1}_{q,p,\gamma_0} (\dOm \times \BBR_+) \times L_{p,\gamma_0}\big( \BBR_+;W^{1}_q(\dOm) + \widehat{W}^1_{q,\Gamma_+}(\Om)\big).
\end{equation*}
Moreover, we have
\begin{equation*}
\|e^{-\gamma_0 t} (\pa_t \Bu, \Bu,\nabla \Bu, \nabla^2 \Bu,\nabla \Fq)\|_{L_p(\BBR_+;L_q(\dOm))}
\leq C_{p,q,\gamma_0} \|\Bu_0\|_{\CD^{2-2\slash p}_{q,p}(\dOm)}.
\end{equation*}

\item Assume that $\Bu_0 =\0$ and  $(\Bf, \BR, g, \Bh,\Bk) \in \CY_{p,q,\ga_0}.$
In other words, $\Bf, \BR, g, \Bh,$ and $\Bk$ satisfy
\begin{gather*}
\Bf \in L_{p,0,\ga_0}\big(\BBR; L_q(\dOm)^N\big), \,\,\,
\BR \in W^1_{p,0,\ga_0}\big(\BBR; L_q(\dOm)^N\big),\\ 
g \in  L_{p,0,\ga_0}\big(\BBR; W^1_q(\dOm) \cap \BW^{-1}_q (\Om)\big) \cap H^{1\slash 2}_{p,0,\ga_0}\big( \BBR; L_q(\dOm)\big),\\
\Bh \in  L_{p,0,\ga_0}\big(\BBR; W^1_q(\dOm)^N\big) \cap H^{1\slash 2}_{p,0,\ga_0}\big( \BBR; L_q(\dOm)^N\big),\\
\Bk \in  L_{p,0,\ga_0}\big(\BBR; W^1_q(\Om_+)^N\big) \cap H^{1\slash 2}_{p,0,\ga_0}\big(\BBR; L_q(\Om_+)^N\big),
\end{gather*}
with $\BR \in \CG(g).$ 
Then \eqref{eq:S} admits a unique solution 
\begin{equation*}
(\Bu, \Fq) \in W^{2,1}_{q,p,0,\gamma_0}\big(\dOm \times \BBR \big) 
   \times L_{p,0,\gamma_0}\big(\BBR; W^1_q (\Om) + \widehat{W}^1_{q,\Gamma_+}(\Om) \big),  
\end{equation*}
possessing the following estimates
\begin{equation*}
 \| e^{-\ga_0 t}(\pa_t \Bu, \Bu, \JDt^{1\slash 2} \nabla \Bu,  \nabla^2 \Bu, \nabla \Fq)\|_{L_p(\BBR; L_q(\dOm))}
\leq C_{p,q,\ga_0} \|(\Bf, \BR, g, \Bh,\Bk)\|_{\CY_{p,q,\ga_0}}.
\end{equation*}
Above, the norm $\|\cdot\|_{\CY_{p,q,\ga_0}}$ is given by 
\begin{align*}
\|(\Bf, \BR, g, \Bh,\Bk)\|_{\CY_{p,q,\ga_0} }
& := \| e^{-\ga_0 t}(\Bf, \pa_t \BR)\|_{L_p(\BBR; L_q(\dOm))}  
+\| e^{-\ga_0 t} (g, \Bh)\|_{L_p(\BBR; W^1_q(\dOm))}\\
 & \quad +\| e^{-\ga_0 t} \Bk\|_{L_p(\BBR; W^1_q(\Om_+))}
+\| e^{-\ga_0 t} \JDt^{1\slash 2}(g, \Bh)\|_{L_p(\BBR; L_q(\dOm))} \\
&\quad +\| e^{-\ga_0 t} \JDt^{1\slash 2}\Bk\|_{L_p(\BBR; L_q(\Om_+))}  .
\end{align*}
\end{enumerate}
\end{theo}
\medskip

%%%%%%%%%%%%%%%%%%%%%%%%%%%%%%%%%%%%%%%%%%%%%%%%%%%%%%%%%%%%%%%%%%%%%%%%%%

\section{Local wellposedness of $\eqref{eq:INSL}$}   \label{sec:local}
This section is dedicated to the proof of Theorem \ref{thm:main_local}.
As the first step, we will reduce \eqref{eq:INSL} to some linear system and outline the main idea of the proof in Section \ref{ssec:local_1}. 
Next, to apply Theorem \ref{thm:stokes},  we shall derive some concrete estimates in Section \ref{ssec:local_2}. Finally, the stability of the reduced system will be verified in Section \ref{ssec:local_3}, which yields the uniqueness.
\subsection{Reduction and the main strategy}\label{ssec:local_1}
For convenience, we rewrite \eqref{eq:INSL} into "Stokes-like" form as follows:
\begin{equation}\label{eq:INSLL}
	\left\{\begin{aligned}
		\pa_t\Bu - \eta^{-1}\Di_\xi \BBT(\Bu,\Fq) = \Bf_{u,\Fq},\,\,\,
		\di_\xi \Bu = g_u=\di_\xi \BR_u    
		                         &\quad\mbox{in}\quad \dOm \times ]0,T[,  \\
\jump{\BBT(\Bu,\Fq)\Bn} = \jump{\Bh_{u,\Fq}}, \quad \jump{\Bu}=\0
                                &\quad\mbox{on}\quad \Ga\times ]0,T[,  \\		
\BBT(\Bu_+,\Fq_+)\Bn_+ = \Bk_{u_+,\Fq_+}  
                                &\quad\mbox{on}\quad \Gamma_+ \times ]0,T[,  \\	
				\Bu_- = \0 &\quad\mbox{on}\quad \Ga_{-}\times ]0,T[, \\
	\Bu|_{t=0}= \Bv_0 &\quad\mbox{in}\quad \dOm,
	\end{aligned}\right.
\end{equation}
where the nonlinear terms
$(\Bf_{u,\Fq}, g_u, \BR_u,  \Bh_{u,\Fq}, \Bk_{u_+\Fq_+})$ are defined by 
\begin{equation*}
 \eta \Bf_{u,\Fq}:= \rho_0 \Bf\big( \BX_{u}(\xi,t) ,t\big)  
+  (\eta-\rho_0) \pa_t \Bu 
-  \Di_\xi \big(\BBT(\Bu,\Fq)- \BBT_u(\Bu,\Fq)\sA_u\big) ,
\end{equation*}
\begin{equation*}
g_u:= \nabla_{\xi}^{\top}\Bu: (\BBI - \sA^{\top}_u ) ,\quad \BR_u:= (\BBI - \sA^{\top}_u ) \Bu,
\end{equation*}
\begin{equation*}
\Bh_{u,\Fq}:=  \BBT(\Bu,\Fq)\Bn - \BBT_u(\Bu,\Fq)\oBn
\quad\hbox{and}\quad
\Bk_{u_+,\Fq_+}:= \BBT(\Bu_+,\Fq_+)\Bn_+ - \BBT_{u_+}(\Bu_+,\Fq_+)\oBn_+.
\end{equation*}

As a starting point, we consider the following linear system with initial state $\Bv_0,$
\begin{equation}\label{eq:uq_L}
	\left\{\begin{aligned}
	 \pa_t\Bu_{_L} - \eta^{-1} \Di_\xi \BBT(\Bu_{_L},\Fq_{_L}) = \0, \,\,\,
		\di_\xi \Bu_{_L} = 0   &\quad\mbox{in}\quad \dOm \times \BBR_+,  \\
		\jump{\BBT(\Bu_{_L},\Fq_{_L})\Bn} = \jump{\Bu_{_L}}=\0 &\quad\mbox{on}\quad \Ga\times \BBR_+,  \\
		\BBT(\Bu_{_{L,+}},\Fq_{_{L,+}})\Bn_+ = \0 &\quad\mbox{on}\quad \Ga_{+}\times \BBR_+, \\
				\Bu_{_{L,-}} = \0 &\quad\mbox{on}\quad \Ga_{-}\times \BBR_+, \\
	\Bu_{_L}|_{t=0}= \Bv_0 &\quad\mbox{in}\quad \dOm.
	\end{aligned}\right.
\end{equation}
Thanks to Theorem \ref{thm:stokes}, we obtain a (unique) global solution $(\Bu_{_L}, \Fq_{_L})$ of \eqref{eq:uq_L} satisfying
\begin{equation*}
(\Bu_{_L}, \Fq_{_L}) \in W^{2,1}_{q,p,\ga_0}(\dOm \times \BBR_+) \times L_{p,\ga_0}\big(\BBR_+; W^1_q (\dOm) + \widehat{W}^1_{q,\Gamma_+}(\Om)\big)
\end{equation*}
for some $\gamma_0 \geq 1.$ Furthermore, there exists  some $C_{p,q,\gamma_0}>0$ such that
\begin{equation}\label{es:uq_L1}
\|e^{-\gamma_0 t} (\pa_t \Bu_{_L}, \Bu_{_L},\nabla \Bu_{_L}, \nabla^2 \Bu_{_L},\nabla \Fq_{_L})\|_{L_p(\BBR_+;L_q(\dOm))}
\leq C_{p,q,\gamma_0} \|\Bv_0\|_{\CD^{2-2\slash p}_{q,p}(\dOm)}.
\end{equation}
By the definition of $\CD^{2-2\slash p}_{q,p}(\dOm)$ and the classical embedding
\begin{equation}\label{eq:embedding}
L_p(I;E_1) \cap W^1_p (I;E_0) \hookrightarrow \BUC\big(I;(E_0,E_1)_{1-1\slash p, p}\big) \quad \forall \,p \in ]1,\infty[,
\end{equation}
we have 
\begin{equation*}
\|e^{-\ga_0 t}  \Bu_{_L}\|_{L_\infty(\BBR_+;\CD^{2-2\slash p}_{q,p}(\dOm))}  
\lesssim  \|e^{-\ga_0 t} \pa_t \Bu_{_L} \|_{L_p(\BBR_+;L_q(\dOm))} 
+\|e^{-\ga_0 t}  \Bu_{_L} \|_{L_p(\BBR_+;W^2_q(\dOm))}.
\end{equation*}
Then above inequality and \eqref{es:uq_L1} yield 
\begin{equation}\label{es:uq_L2}
\|\Bu_{_L}\|_{L_\infty(0,T;\CD^{2-2\slash p}_{q,p}(\dOm))}
\leq C_{p,q,\ga_0}e^{\ga_0 T} \|\Bv_0\|_{\CD^{2-2\slash p}_{q,p}(\dOm)}, 
\quad \hbox{for any finite} \,\,\, T>0.
\end{equation}
Combining \eqref{es:uq_L1} and \eqref{es:uq_L2}, 
we arrive for any finite $T>0,$
\begin{equation}\label{es:uq_L}
\|\Bu_{_L}\|_{L_\infty(0,T;\CD^{2-2\slash p}_{q,p}(\dOm))} +
\| (\pa_t \Bu_{_L}, \Bu_{_L},\nabla \Bu_{_L}, \nabla^2 \Bu_{_L},\nabla \Fq_{_L})\|_{L_p(0,T;L_q(\dOm))}
\leq C_{p,q,\ga_0} e^{\gamma_0 T} \|\Bv_0\|_{\CD^{2-2\slash p}_{q,p}(\dOm)}.
\end{equation}
\medbreak
%%%%%%%%%%%%%%%%%%%

With $(\Bu_{_L}, \Fq_{_L})$ in mind,  we are looking for some solution $(\Bu, \Fq)$ of \eqref{eq:INSLL} which coincides with $(\Bu_{_L} + \BU, \Fq_{_L} + Q)$ in some short time interval $]0,T[\subset ]0,1[.$ Thanks to  \eqref{eq:INSLL} and \eqref{eq:uq_L}, $(\BU, Q)$ is determined by the following system
\begin{equation}\label{eq:UQ}
	\left\{\begin{aligned}
	 \pa_t\BU -	\eta^{-1} \Di_\xi \BBT(\BU, Q) = \oBf_{_{U,Q}},\,\,\,
		\di_\xi \BU = \og_{_U} =\di_\xi \oBR_{_U} 
		  &\quad\mbox{in}\quad \dOm \times ]0,T[,  \\
		\jump{\BBT(\BU, Q)\Bn} = \jump{\oBh_{_{U,Q}} }, \,\,\,\jump{\BU}=\0 &\quad\mbox{on}\quad \Ga\times ]0,T[,  \\
		\BBT_+(\BU_+, Q_+)\Bn_+ = \oBk_{_{U_+,Q_+}}  &\quad\mbox{on}\quad \Ga_{+}\times ]0,T[, \\
				\BU_- = \0 &\quad\mbox{on}\quad \Ga_{-}\times ]0,T[, \\
	\BU|_{t=0}= \0 &\quad\mbox{in}\quad \dOm.
	\end{aligned}\right.
\end{equation}
In \eqref{eq:UQ},  $(\oBf_{_{U,Q}}, \og_{_U}, \oBR_{_U},  \oBh_{_{U,Q}}, \oBk_{_{U_+,Q_+}} ) = (\Bf_{u,\Fq}, g_u, \BR_u,  \Bh_{u,\Fq}, \Bk_{u_+\Fq_+})$ 
as we defined below \eqref{eq:INSLL}. 
Then the local solvability of System \eqref{eq:INSLL} is reduced to studying \eqref{eq:UQ}.
\smallbreak

For convenience, let us introduce suitable solution spaces within $L_p-L_q$ maximal regularity framework.
For any  $(p,q) \in ]1,\infty[^2$ and $\gamma_0>0,$ set that
\begin{equation*}
\sE(T):=  W^{2,1}_{q,p}(\dOm \times ]0,T[)
        \times L_p\big(0,T; W^1_q(\dOm)+ \widehat{W}^1_{q,\Gamma_+}(\Om) \big) 
        \times H^{1,1\slash 2}_{q,p,0,\ga_0}(\dOm \times \BBR)
        \times H^{1,1\slash 2}_{q,p,0,\ga_0}(\Omega_+ \times \BBR)
\end{equation*}
with the norm $\|\cdot\|_{\sE(T)}$ given by
\begin{align*}
\|(\Bw,\nabla  P, \Pi,\Pi_+)\|_{\sE(T)} 
&:=\|\Bw\|_{W^{2,1}_{q,p}  (\dOm \times ]0,T[)}
 + \|\nabla P\|_{L_p(0,T; L_q(\dOm))}\\
 &\quad \quad+\|\Pi\|_{H^{1,1\slash 2}_{q,p,0,\ga_0}(\dOm\times \BBR)} 
 +\|\Pi_+\|_{H^{1,1\slash 2}_{q,p,0,\ga_0}(\Omega_+\times \BBR)} .
\end{align*}
Moreover, we say $(\Bw, P,\Pi,\Pi_+)$ belongs to $\sE_{L}(T)$ for some $L>0,$ if and only if the following assertions hold.
\begin{itemize}
\item $(\Bw, P, \Pi,\Pi_+)\in \sE(T)$ such that
\begin{align*}
\Bw|_{\Gamma_-} = \0 \,\,\, \hbox{on}\,\,\, \Gamma_-, \quad 
\jump{(P -\Pi)\Bn}=\jump{\Bw}=\0\,\,\,\hbox{on}\,\,\, \Gamma \times ]0,T[,\\
\hbox{and} \quad 
P_+ \Bn_+ = \Pi_+ \Bn_+ \,\,\,\hbox{on}\,\,\, \Gamma_+ \times ]0,T[; \hspace{2cm}
\end{align*}
\item The norm of $(\Bw,\nabla  P, \Pi,\Pi_+)$ is bounded by $L,$ i.e.
$\|(\Bw,\nabla  P, \Pi,\Pi_+)\|_{\sE(T)}  \leq L.$
\end{itemize}

\medskip

Now let us return to \eqref{eq:UQ}. Thanks to  \eqref{es:uq_L}, choose parameter $L=L(p,q,\gamma_0,\Bv_0,\Bf)$ such that
\begin{multline*}
\|\Bu_{_L}\|_{L_\infty(0,2;\CD^{2-2\slash p}_{q,p}(\dOm)) \cap W^{2,1}_{q,p}  (\dOm \times ]0,2[)} 
 + \|(\Bf,\nabla \Fq_{_L})\|_{L_p(0,2; L_q(\dOm))}\\
 \leq C_{p,q,\ga_0} e^{2\gamma_0} \|\Bv_0\|_{\CD^{2-2\slash p}_{q,p}(\dOm)}
 +\|\Bf\|_{L_p(0,2; L_q(\dOm))}
\leq L.
\end{multline*}
Fix any $(\Bw, P,\Pi,\Pi_+) \in \sE_{L}(T)$ for some $q >N$ and  $L$ chosen as above, and then consider the following \emph{linearized} Stokes equations with respect to \eqref{eq:UQ},
\begin{equation}\label{eq:UQwP}
	\left\{\begin{aligned}
		 \pa_t\BU - \eta^{-1}\Di_\xi \BBT(\BU, Q) = \oBf_{w,P},\,\,\,
		\di_\xi \BU = \og_w =\di_\xi \oBR_w
		  &\quad\mbox{in}\quad \dOm \times ]0,T[,  \\
		\jump{\BBT(\BU, Q)\Bn} = \jump{\oBh_{w,P} }, \,\,\,\jump{\BU}=\0 &\quad\mbox{on}\quad \Ga\times ]0,T[,  \\
		\BBT_+(\BU_+, Q_+)\Bn_+ = \oBk_{w_+,P_+}  &\quad\mbox{on}\quad \Ga_{+}\times ]0,T[, \\
				\BU_- = \0 &\quad\mbox{on}\quad \Ga_{-}\times ]0,T[, \\
	\BU|_{t=0}= \0 &\quad\mbox{in}\quad \dOm.
	\end{aligned}\right.
\end{equation}
Thus our goal turns to the construction of the solution mapping
$$(\BU,Q,\Xi,\Xi_+):= \Phi(\Bw,P,\Pi,\Pi_+) \in \sE_L(T)$$
 with $(\BU,Q)$ fulfilling \eqref{eq:UQwP} in $]0,T[.$ 
 In fact, this will be a consequence of Theorem \ref{thm:stokes} by assuming the smallness of $T$ and 
$ \|\eta-\rho_0\|_{L_\infty(\dOm)}.$
 \smallbreak 
 
Now, let us sketch the strategy to building such $\Phi.$ 
In Section \ref{ssec:local_2}, we shall find suitable extensions  $(\tBf_{w,P}, \tg_w, \tBR_w,  \tBh_{w,P}, \tBk_{w_+,P_+} )$ over $\BBR$ such that
\begin{equation*}
 (\tBf_{w,P}, \tg_w, \tBR_w,  \tBh_{w,P}, \tBk_{w_+,P_+})|_{]0,T[}=
 (\oBf_{w,P}, \og_w, \oBR_w,  \oBh_{w,P}, \oBk_{w_+,P_+})
\end{equation*}
More importantly, the following bound will be checked for some $\gamma_0 >1$ and $q>N,$ 
\begin{equation}\label{spc:ext}
\|(\tBf_{w,P}, \tg_w, \tBR_w,  \tBh_{w,P}, \tBk_{w_+,P_+})\|_{\CY_{p,q,\ga_0}} < \infty.
\end{equation} 
Thanks to Theorem \ref{thm:stokes} and \eqref{spc:ext},  we can solve
\begin{equation}\label{eq:tUtQwP}
	\left\{\begin{aligned}
		 \pa_t\tBU - \eta^{-1}\Di_\xi \BBT(\tBU, \tQ) = \tBf_{w,P},\,\,\,
		\di_\xi \tBU = \tg_w =\di_\xi \tBR_w   
		 &\quad\mbox{in}\quad \dOm \times \BBR_+,  \\
		\jump{\BBT(\tBU, \tQ)\Bn} = \jump{\tBh_{w,P} }, \,\,\,\jump{\tBU}=\0 &\quad\mbox{on}\quad \Ga\times \BBR_+,  \\
		\BBT_+(\tBU_+, \tQ_+)\Bn_+ = \tBk_{w_+,P_+}  &\quad\mbox{on}\quad \Ga_{+}\times \BBR_+, \\
				\tBU_- = \0 &\quad\mbox{on}\quad \Ga_{-}\times \BBR_+, \\
	\tBU|_{t=0}= \0 &\quad\mbox{in}\quad \dOm.
	\end{aligned}\right.
\end{equation}
Then with the solution $(\tBU,\tQ)$ of \eqref{eq:tUtQwP}, $(U,Q):= (\widetilde{U},\widetilde{Q})|_{[0,T[}$ is exactly the local solution of \eqref{eq:UQwP} on $]0,T[.$
Next we set that
\begin{align}
\Xi &:= \big(\mu(\rho_0) \BBD(\tBU)\Bn \big) \cdot\Bn - \tBh_{w,P} \cdot \Bn\label{eq:Xi}\\
\Xi_+ &:= \big(\mu(\rho_0) \BBD(\tBU_+)\Bn_+ \big) \cdot\Bn_+ - \tBk_{w_+,P_+} \cdot \Bn_+. \label{eq:Xi+}
\end{align}
Thanks to above definitions and the smallness of $T,$ we shall see that $\Phi(\Bw,P,\Pi,\Pi_+):= (\BU,Q,\Xi,\Xi_+)$ is a contracting mapping from $\sE_L(T)$ to itself in Section \ref{ssec:local_2}. This will complete our proof for local existence of \eqref{eq:UQ}, as well as \eqref{eq:INSLL}, by standard fixed point arguments. Moreover, the local solution of \eqref{eq:UQ} is unique by the discussions in Section \ref{ssec:local_3}.

\subsection{Solution operator of \eqref{eq:UQwP}}\label{ssec:local_2}
To construct the solution operator $\Phi$, it is convenient to introduce the operator $E_{(t)}$ as in \cite[Theorem 3.2]{Shi2015}. 
For any (scalar- or vector-valued) mapping $\Fh$ defined on $]0,T[$ and any fixed parameter $t \in ]0,T],$
\begin{equation*}
E_{(t)} \Fh(\cdot, s) := 
\left\{\begin{aligned}
		\Fh(\cdot, s) &\quad\mbox{if}\quad s\in ]0,t[, \\
		\Fh(\cdot, 2t-s) &\quad\mbox{if}\quad  s\in ]t,2t[, \\
		0 &\quad\mbox{otherwise}.
	\end{aligned}\right.
\end{equation*}
For any Banach space $E$ and $(p,\ga) \in ]1,\infty[\times]0,\infty[,$  we have
$E_{(t)} \in \CL \big(L_p(0,T;E),  L_{p,0,\gamma}(\BBR;E)\big).$ 
Indeed,
\begin{equation}\label{es:Fh}
\|e^{-\ga s}E_{(t)}\Fh(\cdot, s) \|_{L_p (\BBR; E)} \leq 2 
\|\Fh(\cdot, s) \|_{L_p (0,T; E)} , 
\quad \hbox{for any}\,\,\, \ga \geq 0 \,\,\,\hbox{and} \,\,\, t \in ]0,T].
\end{equation}
If $\Fh(\cdot,0) =0,$ then it is clear that
\begin{equation*}
\pa_s E_{(t)} \Fh(\cdot, s) = 
\left\{\begin{aligned}
		\pa_s\Fh(\cdot, s) &\quad\mbox{if}\quad s\in ]0,t[, \\
		-(\pa_s \Fh)(\cdot, 2t-s) &\quad\mbox{if}\quad  s\in ]t,2t[, \\
		0 &\quad\mbox{otherwise}.
	\end{aligned}\right.
\end{equation*}
Thus $E_{(t)} \in \CL \big(\widehat{W}^{1}_{p,0}(0,T;E),  \widehat{W}^1_{p,0,\gamma}(\BBR;E)\big).$ In fact, we easily show that
\begin{equation}\label{es:DFh}
\|e^{-\ga s} \pa_s \big(E_{(t)} \Fh(\cdot, s)\big) \|_{L_p (\BBR; E)} \leq 2 
\|(\pa_s \Fh)(\cdot, s) \|_{L_p (0,T; E)} 
\,\,\,\hbox{for any}\,\,\, \ga \geq 0 \,\,\,\hbox{and} \,\,\, t \in ]0,T].
\end{equation}
\smallbreak

Next, let us derive some useful estimates.
For simplicity, define that $(\BW, \Theta):= (\Bu_{_L}+\Bw, \FpL + P).$
Thanks to \eqref{es:Du_decay} and the conventions on $(T,L),$ we have 
\begin{equation}\label{es:decayW}
\|\nabla \BW\|_{L_{1} (0,T; L_\infty(\dOm))} \leq C_N T^{1\slash{p'}}
\|(\nabla \Bw, \nabla \Bu_{_L})\|_{L_{p} (0,T; L_\infty(\dOm))} 
 \leq C_N T^{1\slash{p'}+ \sigma_{p,q}} L,
\end{equation} 
where, for any $1<p<\infty$ and $N<q<\infty,$  the non-negative index $\sigma_{p,q}$ is given by
\begin{equation*}
\sigma_{p,q}:= 
\begin{cases}
 \big(1-N\slash q\big)\slash 2, 
    & \hbox{for}\quad  2\slash p + N\slash q>1;\\
  0,&  \hbox{for} \quad  2\slash p + N\slash q \leq 1.
  \end{cases}
\end{equation*}
Thus $X_{_W},$ $\sA_{_W}$ and $\oBn_{_W}:= (\sA_{_W}\Bn) \slash |\sA_{_W}\Bn|$ are well defined with $T$ satisfying
\begin{equation}\label{cdt:Tsmall1}
 C_N T^{1\slash{p'} + \sigma_{p,q}} L \leq 1\slash 2.
\end{equation}
Combining this decay property and Condition \eqref{cdt:Tsmall1}, we infer from Lemma \ref{lem:A_DA} and Lemma \ref{lem:normal} that
\begin{align}\label{ob:key}
\|(\sA_{_W}-\BBI,\oBn_{_W}-\Bn)\|_{L_\infty(\dOm\times]0,T[)} 
&\lesssim T^{1\slash{p'}+\sigma_{p,q}}L,&
\|\nabla_\xi (\sA_{_W}-\BBI,\oBn_{_W}-\Bn)\|_{L_\infty(0,T;L_q(\dOm))}
&\lesssim T^{1\slash{p'}} L,\\ \nonumber
\|\pa_t (\sA_{_W}-\BBI,\oBn_{_W}-\Bn)\|_{L_p(0,T;L_\infty(\dOm))}
 &\lesssim T^{\sigma_{p,q}}L,&
\|\pa_t \nabla_\xi (\sA_{_W}-\BBI,\oBn_{_W}-\Bn)\|_{L_p(0,T;L_q(\dOm))}
&\lesssim  L.
\end{align}

In the rest of this subsection, we devote ourselves to verifying the bound \eqref{spc:ext} with keeping \eqref{es:Fh} and \eqref{es:DFh} in mind.
The fact that  $W^1_q (\dOm) (\hookrightarrow \CC_b(\dOm))$ is a Banach algebra for $N<q <\infty$ will be also constantly used without mention. 
%%%%%%%%%%%%%%%%%%%%%%%%%%%%%%%%%%%%%%%%%%%%%%%%%%%%%%%%%%%%%%%%%%%%%%%%%%%%%%%%%%%%%%%%%%
\subsubsection*{\underline{Bounds for $\tBf_{_{w,P}}$}}
Recall the definition of $\oBf_{_{w,P}}$ as follows,
\begin{align*}
\eta \oBf_{_{w,P}} &= \rho_0 \Bf\big( \BX_{_W}(\xi,t) ,t\big)+ (\eta-\rho_0) \pa_t \BW
                 -\Di_\xi \Big(\mu(\rho_0) \big(\BBH_{_W}
                 +\BBD(W)(\BBI-\sA_{_W}) \big) \Big)\\
              &\quad +\Di_\xi \big(\mu(\rho_0)\BBH_{_W}(\BBI-\sA_{_W})\big) 
                 +\Di_\xi \big(\Theta(\BBI-\sA_{_W})\big),
\end{align*}
where we adopt the notation
\begin{equation*}
\BBH_{_W} :=  \nabla_\xi^{\top}\BW (\BBI-\sA_{_W}^{\top}) +(\BBI-\sA_{_W})\nabla_\xi \BW^{\top}.  
\end{equation*} 

To seek some suitable extension of $\oBf_{w,P}$ in $L_{p,0,\ga_0}\big(\BBR;L_q(\dOm)\big),$
 let us first give some notations for convenience.
Assume that $\chi \in \CC^{\infty}(\BBR)$ is some cut-off function satisfying
$\chi (t) = 1$  for $|t| \leq 1$ and 
$\chi (t) = 0$ for $|t|\geq 2.$
Thanks to \eqref{es:uq_L}, define that
\begin{equation*}
\tBuL(\cdot,t) := \chi(t) e^{-|t| \CA_q} \Bv_0(\cdot) 
\quad\hbox{for}\,\,\, t\in \BBR.
\end{equation*}
From the definition of $\tBuL,$ it is obvious that 
$\tBuL(\cdot,t)= \BuL (\cdot,t)$ for any $t \in [0,1].$ Moreover, the mixed derivative theorem implies 
\begin{equation}\label{es:tBuL}
\|\tBuL\|_{H^{1\slash 2}_p(\BBR; W^1_q(\dOm))} \lesssim 
\|\tBuL\|_{W^{2,1}_{q,p}(\dOm \times \BBR)} \lesssim 
\|\BuL\|_{W^{2,1}_{q,p}(\dOm \times ]0,2[)} \lesssim L.
\end{equation}
Next with above $\tBuL$, consider $\tBW := E_{_{(T)}} \Bw + \tBuL$ which satisfies
\begin{equation*}
\tBW(\cdot,t) = \BW(\cdot, t) \,\,\,\hbox{for any}\,\,\,t\in [0,T[ 
\quad \hbox{and} \quad
\tBW(\cdot,t) = \0 \,\,\,\hbox{for any}\,\,\,t\notin ]-2,2[. 
\end{equation*}
Furthermore, \eqref{es:tBuL} and the mixed derivative theorem yield
\begin{equation}\label{es:tW}
\|\tBW\|_{H^{1\slash 2}_p(\BBR; W^1_q(\dOm))} \lesssim 
\|\tBW\|_{W^{2,1}_{q,p}(\dOm \times \BBR)} \lesssim 
\|E_{_{(T)}} \Bw\|_{W^{2,1}_{q,p}(\dOm \times ]0,2T[)}+\|\tBuL\|_{W^{2,1}_{q,p}(\dOm \times \BBR)} \lesssim L.
\end{equation}
Thanks to \eqref{es:decayW}, $\tBW$ has the similar decay property due to \eqref{es:tW}
\begin{equation}\label{es:decaytW}
\|\nabla_\xi \tBW\|_{L_{1} (0,2T; L_\infty(G))} 
\lesssim  T^{1\slash{p'}}
\|\nabla_\xi \tBW\|_{L_{p} (0,2T; L_\infty(G))}  
\lesssim T^{1\slash{p'} + \sigma_{p,q}} L.
\end{equation}
\medskip

Now, keep $\tBW$ in mind and  introduce the following matrices
\begin{align*}
\tBBH_{_W} &:=  \nabla_\xi^{\top} \tBW \cdot E_{_{(T)}}(\BBI-\sA_{_W}^{\top})
             + E_{_{(T)}}(\BBI-\sA_{_W}) \cdot \nabla_\xi \tBW^{\top},\\
 \tBBD_{_W} &:=\BBD(\tBW) \cdot  E_{_{(T)}}(\BBI -\sA_{_W}).         
\end{align*}
Then one desired extension of the source term is given by,
\begin{align}\label{eq:tfwP}
\eta \tBf_{_{w,P}} &:= \rho_0 \ET \Bf\big( \BX_{_W}(\xi,t) ,t\big)
                   + (\eta-\rho_0) \ET \pa_t \BW
                 -\Di_\xi \big(\mu(\rho_0) (\tBBH_{_W} +\tBBD_{_W} ) \big)\\
    &\quad +\Di_\xi \big(\mu(\rho_0)\tBBH_{_W}\ET (\BBI-\sA_{_W})\big) 
            +\Di_\xi \Big(   \ET \big(\Theta (\BBI-\sA_{_W}) \big)\Big).\nonumber
\end{align}
Obviously, $\tBf_{w,P}|_{t\in ]0,T[} =\oBf_{w,P}$ is fulfilled. More importantly,  we shall prove that
\begin{equation}\label{es:tfwP}
\|\tBf_{w,P}\|_{L_{p,0,\ga_0}(\BBR;L_q(\dOm))} 
\lesssim  c L +T^{1\slash{p'} + \sigma_{p,q}}L^2 \big(\|\mu\|_{L_\infty(\dOm)}
 + T^{\sigma_{p,q}} \|\nabla \mu\|_{L_q(\dOm)} \big) 
  +T^{1\slash{p'}}L^2,
\end{equation}
where $\nabla \mu$ stands for $\nabla \big(\mu (\rho_0)\big)$ in short.
Furthermore, $\|\nabla \mu\|_{L_q(\dOm)}< \infty$ by $(\CH 3)$ and $\rho_0 \in \wh W^1_q(\dOm).$
\medskip 

To verify \eqref{es:tfwP}, first note that Condition \eqref{cdt:Tsmall1} (up to the choice of $C_N$) yields, 
\begin{equation*}
\Big\|\det \big(\BBI + \int_0^t \nabla \BW(\cdot,\tau) d\tau\big)\Big\|_{L_\infty(0,T;L_\infty(\dOm))} \leq 1\slash 2\, .
\end{equation*}
Thus we have for some $\gamma_0 >0,$
\begin{equation}\label{es:tfwP1}
\big\|\ET \Bf\big( \BX_{_W}(\xi,t) ,t\big)\big\|_{L_{p,0,\ga_0}(\BBR;L_q(\dOm))} 
\lesssim \|\Bf\|_{L_p(0,T;L_q(\dOm))} \lesssim cL.
\end{equation}
\smallbreak

Next, the second term on the right hand side of \eqref{eq:tfwP}  is easy bounded by  
\begin{equation}\label{es:tfwP2}
\|(\eta-\rho_0) \ET \pa_t \BW\|_{L_{p,0,\ga_0}(\BBR;L_q(\dOm))}
\lesssim \|\eta-\rho_0\|_{L_\infty(\dOm)} L \lesssim cL.
\end{equation}
\smallbreak

To study the nonlinear terms in \eqref{eq:tfwP}, we know from \eqref{ob:key} and \eqref{es:decaytW} that for any $1\leq j,k,\ell,m\leq N,$
\begin{equation}\label{es:HD1}
\big\|\big(\ET (\BBI-\sA_{_W})\big)^j_k 
(\nabla_\xi \tBW)^\ell_m\big\|_{L_p(0,2T;W^1_q(\dOm))}
\lesssim  T^{1\slash{p'} + \sigma_{p,q}}L^2,       
\end{equation}
\begin{equation}\label{es:HD2}
\big\|\big(\ET (\BBI-\sA_{_W})\big)^j_k 
(\nabla_\xi \tBW)^\ell_m\big\|_{L_p(0,2T;L_\infty(\dOm))}
\lesssim T^{1\slash{p'} + 2\sigma_{p,q}}L^2. 
\end{equation}
Then combining the bounds \eqref{es:HD1} and \eqref{es:HD2}, we obtain that for any $\gamma_0 >0,$
\begin{equation}\label{es:tHtD1}
\|\mu(\rho_0)(\tBBH_{_W}, \tBBD_{_W})\|_{L_{p,0,\ga_0}(\BBR;W^1_q(\dOm))} 
\lesssim T^{1\slash{p'} + \sigma_{p,q}}L^2
  \big(\|\mu\|_{L_\infty(\dOm)} + T^{\sigma_{p,q}} \|\nabla \mu\|_{L_q(\dOm)} \big).
\end{equation}
Moreover, compared with \eqref{es:tHtD1}, $\mu(\rho_0) \tBBH_{_W} \cdot E_{_{(T)}}(\BBI - \sA_{_W})$ is higher order term for short time,
\begin{equation}\label{es:tHtD2}
\|\mu(\rho_0) \tBBH_{_W} \cdot E_{_{(T)}}(\BBI - \sA_{_W})\|_{L_{p,0,\ga_0}(\BBR;W^1_q(\dOm))} 
        \lesssim T^{2\slash{p'} + 2\sigma_{p,q}}L^3
                     \big(\|\mu\|_{L_\infty(\dOm)} + T^{\sigma_{p,q}} \|\nabla \mu\|_{L_q(\dOm)} \big). 
\end{equation}
Hence we can bound the nonlinear terms  by \eqref{cdt:Tsmall1}, \eqref{es:tHtD1} and \eqref{es:tHtD2},
\begin{multline}\label{es:tfwP3}
\|\Di_\xi \big(\mu(\rho_0) (\tBBH_{_W} +\tBBD_{_W} ) \big)
  -\Di_\xi \big(\mu(\rho_0)\ET (\BBI-\sA_{_W}^{\top})\tBBH_{_W}\big)\|_{L_{p,0,\ga_0}(\BBR;L_q(\dOm))}\\
\lesssim T^{1\slash{p'} + \sigma_{p,q}}L^2 \big(\|\mu\|_{L_\infty(\dOm)} 
  + T^{\sigma_{p,q}} \|\nabla \mu\|_{L_q(\dOm)} \big).
\end{multline}
The bound of last pressure term of $\eta \tBf_{w,P}$ is immediate from \eqref{ob:key},
\begin{equation}\label{es:tfwP4}
\Big\|\Di_\xi \Big( \ET \big(\Theta(\BBI-\sA_{_W})\big)\Big) \Big\|_{L_{p,0,\ga_0}(\BBR;L_q(\dOm))}
\lesssim T^{1\slash{p'}}L^2.
\end{equation} 
At last, putting together the bounds \eqref{es:tfwP1}, \eqref{es:tfwP2}, \eqref{es:tfwP3} and \eqref{es:tfwP4} completes the proof of our claim \eqref{es:tfwP}.
%%%%%%%%%%%%%%%%%%%%%%%%%%%%%%%%%%%%%%%%%%%%%%%%%%%%%%%%%%%%%%%%%%%%%%%%%%%%%%%%%%%%%%%%%%
\subsubsection*{\underline{Bound of $\tg_w$}}
Based on the expression
\begin{equation*}
\og_w =  \nabla_{\xi}^{\top} \BW:  (\BBI - \sA^{\top}_{_W} ) 
 = \di_\xi \oBR_w =\di_\xi  \big( (\BBI - \sA^{\top}_{_W} ) \BW \big),
\end{equation*}
let us consider the following extension
\begin{equation}\label{eq:tgwRw}
\tg_w := \ET \big( \nabla_{\xi}^{\top}\tBW: (\BBI -  \sA^{\top}_{_{W}})\big)
 = \di \Big(  E_{_{(T)}}\big((\BBI-\sA^{\top}_{_{W}}) \tBW \big) \Big)
 =: \di \tBR_w.
\end{equation}
By \eqref{eq:tgwRw}, we immediately know that $\tg_w (\cdot,t) = \og_w(\cdot, t)$ for $t \in [0,T[$ and vanishes for $t \notin ]0,2T[.$
To prove such $\tg_w$ is an admissible extension, it is sufficient to verify 
for some $\ga_0 >1,$ 
\begin{equation}\label{spc:tgw}
\tg_w \in  L_{p,0,\ga_0}\big(\BBR; W^1_q(\dOm) \cap \BW^{-1}_q (\Om)\big) \cap H^{1\slash 2}_{p,0,\ga_0}\big( \BBR; L_q(\dOm)\big).
\end{equation}
In fact, the first part of \eqref{spc:tgw} is guaranteed by \eqref{es:HD1},
\begin{equation}\label{es:gw1}
\|\tg_w\|_{L_{p,0,\ga_0}(\BBR;W^1_q(\dOm))}  
\lesssim \|\tg_w\|_{L_p(0,2T;W^1_q(\dOm))} 
\lesssim  T^{1\slash{p'} + \sigma_{p,q}}L^2.
\end{equation}
\smallbreak

Next, to prove
$\tg_w \in  L_{p,0,\ga_0}\big(\BBR; \BW^{-1}_q (\Om)\big),$ 
it is sufficient to verify 
\begin{equation*}
(\tg_w, \varphi)_{\Omega} =(\di \tBR_w, \varphi)_{\Omega}= - (\tBR_w, \nabla \varphi)_{\Omega},
\,\,\, \forall \,\, \varphi \in W^1_{q',\Gamma_+}(\Omega).
\end{equation*}
Note that 
$\tBR_w =\big( E_{_{(T)}}(\BBI-\sA^{\top}_{_{W}}) \big) \tBW =\0 
  \,\,\,\hbox{on}\,\,\, \Gamma_-$ as 
$\tBW |_{\Gamma_-} =\0.$ 
On the other hand, we claim
\begin{equation}\label{eq:jumptRw}
\jump{\tBR_w \cdot \Bn}
 = \jump{\big( E_{_{(T)}}(\BBI-\sA^{\top}_{_{W}}) \big) \tBW \cdot \Bn}=0 
  \,\,\,\hbox{on}\,\,\, \Gamma,
\end{equation}
which yields our desired result.To complete the proof of \eqref{eq:jumptRw},
assume only $t \in ]0,T]$ without loss of generality.
By the continuity of $\jump{\tBW}$ across $\Gamma$ and  \cite[Remark 3.1]{DeniSol1995}, we see that
\begin{equation*}
\jump{\sA^{\top}_{_{W}}\BW \cdot \Bn} = \BW \cdot  \jump{\sA_{_W}\Bn}=0 
  \,\,\,\hbox{on}\,\,\, \Gamma, \quad \forall\, t\in ]0,T].
\end{equation*}
\smallbreak

Finally, it remains to check
$\tg_w \in H^{1\slash 2}_{p,0,\ga_0} \big(\BBR_+;L_q(\dOm)\big).$ 
However, for technical reason, we divide the proof into two cases:
\begin{equation*}
(I) :=\{(p,q)\in ]2,\infty[\times ]N,\infty[\}
\quad \hbox{and}\quad 
(II) :=\{(p,q)\in ]1,2]\times ]N,\infty[\,:1 \slash p + N \slash q >3\slash 2\}.
\end{equation*}
For the case (I), we have the embedding
\begin{equation}\label{eq:embed_1}
\CD^{2-2\slash p}_{q,p}(\dOm) \hookrightarrow W^1_q(\dOm) \hookrightarrow L_\infty(\dOm).
\end{equation}
Then Lemma \ref{lem:Hhalf}, \eqref{ob:key} and \eqref{es:tW} imply that
\begin{align}\label{es:gw21}
\|\tg_w\|_{H^{1\slash 2}_{p,0,\ga_0} (\BBR;L_q(\dOm))} 
 & \lesssim \|\BBI-\sA^{\top}_{_{W}}\|_{L_\infty(\dOm\times ]0,T[)}^{1\slash 2} 
\Big( \| \BBI-\sA^{\top}_{_{W}}\|_{L_\infty(0,T;W^1_q(\dOm))}
 \\ \nonumber
 & + T^{(q-N)\slash (pq)} \|\pa_t   \sA^{\top}_{_{W}}\|_{L_\infty(0,T;L_q(\dOm))}^{1-N\slash (2q)}
  \|\pa_t \sA^{\top}_{_{W}}\|_{L_p(0,T;H^1_q(\dOm))}^{N\slash (2q)}
\Big)^{1\slash 2}
 \| \nabla_{\xi}^{\top} \tBW \|_{H^{1\slash 2, 1\slash 2}_{q,p}(\dOm\times \BBR)}
 \\ \nonumber
& \lesssim  T^{ (1\slash {p'} +\sigma_{p,q})\slash 2} L^{3\slash 2} 
\Big( T^{1\slash {p'}} L + T^{(q-N)\slash (pq)}
   \|\BW\|_{L_\infty(0,T;W^1_q(\dOm))}^{1-N\slash (2q)} L^{N\slash (2q)}  
\Big)^{1\slash 2} 
\\ \nonumber
&\lesssim T^{(1-N\slash (pq)+\sigma_{p,q})\slash 2} L^2.
\end{align}
\smallbreak

On the other hand, if $(p,q) \in (II),$  the comments below Proposition \ref{prop_uDu} yield that,
\begin{equation}\label{es:DtAW}
\|\pa_t \sA_{_W}\|_{L_{p \slash \theta}(0,T;L_\beta(\dOm))}
\lesssim \|\nabla_\xi \BW\|_{L_{p \slash \theta}(0,T;L_\beta(\dOm))} \lesssim L
\end{equation}
for some $\theta$ and $\beta$ satisfying $1-2(1-\theta)\slash p = N\slash q -N\slash \beta.$
Thus we can infer from Lemma \ref{lem:Hhalf2} and \eqref{ob:key} that 
\begin{align}\label{es:gw22}
\|\tg_w\|_{H^{1\slash 2}_{p,0,\ga_0} (\BBR;L_q(\dOm))} 
 & \lesssim 
\|\BBI-\sA^{\top}_{_{W}}\|^{1\slash 2}_{L_\infty(\dOm\times ]0,T[)} 
\Big( \| \BBI-\sA^{\top}_{_{W}}\|_{L_\infty(0,T;W^1_q(\dOm))}
\\ \nonumber
&\quad + \|\pa_t  \sA_{_W}^{\top}\|_{L_{p \slash \theta}(0,T;L_\beta(\dOm))}
\Big)^{1\slash 2}
 \| \nabla_{\xi}^{\top} \tBW \|_{H^{1\slash 2 ,1\slash 2}_{q,p}(\dOm\times \BBR)}
\lesssim  T^{ (1\slash {p'} +\sigma_{p,q})\slash 2} L^{2}.
\end{align}

Finally, combining the estimates \eqref{es:gw1}, \eqref{es:gw21} and \eqref{es:gw22} furnishes that
\begin{equation}\label{es:gw}
\|\tg_w\|_{ H^{1,1\slash 2}_{q,p,0,\ga_0} (\dOm \times \BBR)} 
\lesssim (T^{1\slash {p'} + \sigma_{p,q}} + T^{s_{p,q}}) L^2
 \lesssim  T^{s_{p,q}} L^2,
\end{equation}
where the index $s_{p,q}$ is given by
\begin{equation*}
s_{p,q}:= 
\begin{cases}
 \big(1-N\slash (pq)+\sigma_{p,q} \big)\slash 2 
                               & \hbox{for}\quad (p,q) \in (I),\\
 (1\slash {p'} +\sigma_{p,q})\slash 2
                         &  \hbox{for} \quad (p,q) \in (II).
  \end{cases}
\end{equation*}

%%%%%%%%%%%%%%%%%%%%%%%%%%%%%%%%%%%%%%%%%%%%%%%%%%%%%%%%%%%%%%%%%%%%%%%%%%%%%%%%%%%%%%%%%%
\subsubsection*{\underline{Bound of $\tBR_w$}}
Recall the definition
$\tBR_w= \big( E_{_{(T)}}(\BBI-\sA^{\top}_{_{W}}) \big) \tBW$ in \eqref{eq:tgwRw}. Obviously $\tBR_w(\cdot,t) = \oBR_w(\cdot,t)$ for $t\in ]0,T[.$ 
Besides, employing Lemma \ref{lem:A_DA} and \eqref{ob:key} implies that 
\begin{equation}\label{es:DtRw1}
\|\pa_t \tBR_{_W}\|_{L_{p,0,\ga_0}(\BBR;L_q(\dOm))} 
 \lesssim \|\tW^{j} \pa_{\xi_{_\ell}} E_{_{(T)}}W^k\|_{L_p(0,2T;L_q(\dOm))}  +T^{1\slash {p'} + \sigma_{p,q}}L^2 .
\end{equation}
So to handle the nonlinear term $\tW^{j} \pa_{\xi_{_\ell}} E_{_{(T)}}W^k,$
let us first consider the case where $N \slash q + 2 \slash p \leq 1.$ 
In this situation, we have the embedding \eqref{eq:embed_1}  and thus
\begin{equation}\label{es:DtRw2}
\|\tW^{j} \pa_{\xi_{_\ell}} E_{_{(T)}}W^k\|_{L_p(0,2T;L_q(\dOm))} 
\lesssim T^{1\slash p}  \|\tBW\|_{L_\infty(\dOm \times ]0,2T[)} \|\nabla_\xi \BW\|_{L_\infty(0,T; L_q(\dOm))}
\lesssim T^{1\slash p} L^2.
\end{equation}
\smallbreak

On the other hand, we infer from \eqref{es:decayW} by assuming $N \slash q + 2 \slash p > 1,$ 
\begin{equation}\label{es:DtRw3}
\|\tW^{j} \pa_{\xi_{_\ell}} E_{_{(T)}}W^k\|_{L_p(0,T;L_q(\dOm))} 
\lesssim \|\nabla_\xi \BW\|_{L_p(0,T; L_\infty(\dOm))}  \|\tBW\|_{L_\infty(0,2T; L_q(\dOm))} 
\lesssim T^{\sigma_{p,q}}L^2.
\end{equation}
At last, inserting \eqref{es:DtRw2} and \eqref{es:DtRw3} into \eqref{es:DtRw1} yields for all $(p,q) \in ]1,\infty[\times ]N,\infty[,$
\begin{equation}\label{es:DtRw}
\|\pa_t \tBR_{_W}\|_{L_{p,0,\ga_0}(\BBR;L_q(\dOm))}  
 \lesssim  T^{\tsigma_{p,q}} L^2 \quad 
 \hbox{with}\,\,\,
  \tsigma_{p,q} := \min \big\{ 1\slash p, (1-N\slash q)\slash 2 \big\} >0. \\
\end{equation}

%%%%%%%%%%%%%%%%%%%%%%%%%%%%%%%%%%%%%%%%%%%%%%%
\subsubsection*{\underline{Bound of $\tBh_{w,P}$}}
Recall the symbol $\BBH_{_W} = \BBD(\BW) - \BBD_{_W}(\BW)$ and write out
\begin{equation*}
\oBh_{w,P} = \mu(\rho_0) \BBH_{_W} \Bn + \mu(\rho_0)\BBH_{_W} (\oBn_{_W}- \Bn)
           - \mu(\rho_0) \BBD(\BW)(\oBn_{_W}- \Bn) + \Theta\,(\oBn_{_W}- \Bn)
\end{equation*} 
Then this motivates us to introduce that
\begin{equation*}
\tBh_{w,P} := \mu(\rho_0) \tBBH_{_W} \Bn + \mu(\rho_0)\tBBH_{_W} \ET(\oBn_{_W}- \Bn)
           - \mu(\rho_0) \BBD(\tBW)\ET(\oBn_{_W}- \Bn) + \Pi\,\ET(\oBn_{_W}- \Bn).
\end{equation*}
Clearly $\jump{\tBh_{w,P}(\cdot, t) } = \jump{\oBh_{w,P}(\cdot, t)}$ 
on $\Gamma$ for any $t\in ]0,T[.$ 
In fact, we will see that
\begin{equation} \label{spc:thwP}
\tBh_{w,P} \in L_{p,0,\ga_0}\big(\BBR; W^1_q(\dOm)^{N}\big) \cap 
                H^{1\slash 2}_{p,0,\ga_0}\big(\BBR; L_q(\dOm)^{N}\big)                 
                \quad
                \hbox{for some}\,\,\,\ga_0 >1,
\end{equation}
and thus $\tBh_{w,P}$  is exactly one desired extension.
\medskip

So let us study the claim \eqref{spc:thwP}.
Firstly, the inequalities \eqref{es:HD1}, \eqref{es:HD2} and Lemma \ref{lem:normal}  imply that 
\begin{equation}\label{es:thwP11}
\|\mu(\rho_0) (\tBBH_{_W}, \tBBH_{_W} \Bn)  \|_{L_{p,0,\ga_0}(\BBR;W^1_q(\dOm))}
\lesssim T^{1\slash {p'}+ \sigma_{p,q}}L^2
  \big(\|\mu\|_{L_\infty(\dOm)} + T^{\sigma_{p,q}} \|\nabla \mu\|_{L_q(\dOm)} \big).
\end{equation}
Now thanks to \eqref{ob:key},  \eqref{es:thwP11} and \eqref{es:decaytW}  yield respectively,
\begin{equation}\label{es:thwP12}
\|\mu(\rho_0)\tBBH_{_W} \ET(\oBn_{_W}- \Bn)\|_{L_{p,0,\ga_0}(\BBR;W^1_q(\dOm))}
\lesssim T^{2\slash {p'} + 2\sigma_{p,q}}L^3
  \big(\|\mu\|_{L_\infty(\dOm)} + T^{\sigma_{p,q}} \|\nabla \mu\|_{L_q(\dOm)} \big),
\end{equation}
\begin{equation}\label{es:thwP13}
\|\mu(\rho_0) \BBD(\tBW)\ET(\oBn_{_W}- \Bn)\|_{L_{p,0,\ga_0}(\BBR;W^1_q(\dOm))}
\lesssim T^{1\slash {p'} + \sigma_{p,q}}L^2
  \big(\|\mu\|_{L_\infty(\dOm)} + T^{\sigma_{p,q}} \|\nabla \mu\|_{L_q(\dOm)} \big).
\end{equation}
The estimate of the pressure term in $\tBh_{w,P}$ is immediate from \eqref{ob:key},
\begin{equation}\label{es:thwP14}
\| \Pi\,\ET(\oBn_{_W}- \Bn)\|_{L_{p,0,\ga_0}(\BBR;W^1_q(\dOm))} 
\lesssim T^{1\slash {p'}} L^2. 
\end{equation}
Therefore, the bounds \eqref{es:thwP11}, \eqref{es:thwP12}, \eqref{es:thwP13}, \eqref{es:thwP14} and the condition \eqref{cdt:Tsmall1} imply that
\begin{equation}\label{es:thwP1}
\|  \tBh_{w,P}  \|_{L_{p,0,\ga_0}(\BBR; W^1_q(\dOm))}
         \lesssim T^{1\slash {p'} + \sigma_{p,q}}L^2
                     \big(\|\mu\|_{L_\infty(\dOm)} + T^{\sigma_{p,q}} \|\nabla \mu\|_{L_q(\dOm)} \big) +T^{1\slash {p'}} L^2 ,
\end{equation}
which completes the proof of the first part of \eqref{spc:thwP}.
\medskip

To verify the rest property in \eqref{spc:thwP}, we first recall the proof of \eqref{es:gw} and gain that
\begin{equation}\label{es:tHhalf}
\|\tBBH_{_W}\|_{H^{1,1\slash 2}_{q,p,0,\ga_0} (\dOm\times\BBR)}
 \lesssim T^{s_{p,q}}L^2.
\end{equation}
Then the bound \eqref{es:HD2} and Lemma \ref{lem:normal} imply that
\begin{equation}\label{es:thwP21}
\|\mu(\rho_0)(\tBBH_{_W},\tBBH_{_W}\Bn) \|_{H^{1,1\slash 2}_{q,p,0,\ga_0} (\dOm\times\BBR)} 
\lesssim 
\|\mu\|_{L_\infty(\dOm)} T^{s_{p,q}}L^2 +\|\nabla \mu\|_{L_q(\dOm)}T^{1\slash {p'} + 2\sigma_{p,q}}L^2.
\end{equation}
Now,  thanks to Lemma \ref{lem:normal} and \eqref{es:DtAW}, note that
\begin{equation*}
\|\pa_t(\oBn_{_W}- \Bn) \|_{L_{p \slash \theta}(0,T;L_\beta(\dOm))}
\lesssim \|\pa_t\sA_{_W} \|_{L_{p \slash \theta}(0,T;L_\beta(\dOm))} \lesssim L.
\end{equation*}
Thus by means of \eqref{ob:key}, we can produce similar arguments as \eqref{es:gw21} and \eqref{es:gw22} with replacing $\sA_{_W}-\BBI$ by $\oBn_{_W}- \Bn.$ Indeed, the estimates below hold true,
\begin{equation}\label{es:thwP22}
\|\mu(\rho_0)\tBBH_{_W} \ET(\oBn_{_W}- \Bn)\|_{H^{1\slash 2}_{p,0,\ga_0} (\BBR;L_q(\dOm))} 
\lesssim
\|\mu\|_{L_\infty(\dOm)} T^{2s_{p,q}}L^3 +\|\nabla \mu\|_{L_q(\dOm)}T^{1\slash {p'} + 2\sigma_{p,q} + s_{p,q}}L^3.
\end{equation}
\begin{equation}\label{es:thwP23}
\|\mu(\rho_0) \BBD(\tBW)\ET(\oBn_{_W}- \Bn)\|_{H^{1\slash 2}_{p,0,\ga_0} (\BBR;L_q(\dOm))} 
\lesssim \|\mu\|_{L_\infty(\dOm)} T^{s_{p,q}}L^2.
\end{equation}
\begin{equation}\label{es:thwP24}
\|\Pi\,\ET(\oBn_{_W}- \Bn)\|_{H^{1\slash 2}_{p,0,\ga_0} (\BBR;L_q(\dOm))} 
\lesssim T^{s_{p,q}}L^2.
\end{equation}
Lastly, keeping the following convention in mind
\begin{equation}\label{cdt:Tsmall2}
 CT^{s_{p,q}} L \leq 1 \quad \hbox{for some constant}\,\,\,C,
\end{equation}
we infer from \eqref{es:thwP21}, \eqref{es:thwP22},  \eqref{es:thwP23} and \eqref{es:thwP24} that
\begin{equation}\label{es:thwP2}
\|\tBh_{w,P}\|_{H^{1\slash 2}_{p,0,\ga_0} (\BBR_+;L_q(\dOm))} 
\lesssim
(\|\mu\|_{L_\infty(\dOm)} +1) T^{s_{p,q}}L^2 +\|\nabla \mu\|_{L_q(\dOm)}T^{1\slash {p'} + 2\sigma_{p,q}}L^2.
\end{equation}

Hence combine the estimates \eqref{es:thwP1} and \eqref{es:thwP2} and we end up with
\begin{equation}\label{es:thwP}
\|\tBh_{w,P}\|_{H^{1,1\slash 2}_{q,p,0,\ga_0} (\dOm\times \BBR)} 
\lesssim
\|\mu\|_{L_\infty(\dOm)} T^{s_{p,q}}L^2 
+\|\nabla \mu\|_{L_q(\dOm)}T^{1\slash {p'} + 2\sigma_{p,q}}L^2
 + T^{\min\{s_{p,q} ,1\slash {p'}\}} L^2.
\end{equation}
%%%%%%%%%%%%%%%%%%%%%%%%%%%%%%%%%%%%%%%%%%%%%%%%%%%%%%%%%%%%%%%%%%%%%%%%%%%%%%%%%%%%%%%%%%
\subsubsection*{ \underline{Bound of $\tBk_{w_+,P_+}$} }
Inspired by previous step, we introduce the following notations,
\begin{align*}
\tBBK_{_{W_+}} &:=  \nabla_{\xi}^{\top} \tBW_+ E_{_{(T)}}(\BBI-\sA_{_{W_+}}^{\top})
           +E_{_{(T)}}(\BBI-\sA_{_{W_+}}) \nabla_\xi \tBW_+^{\top},\\
\tBk_{w_+,P_+} &:= \mu(\rho_0) \tBBK_{_{W_+}} \Bn_+ + \mu(\rho_0)\tBBK_{_{W_+}} \ET(\oBn_{_{W_+}}- \Bn_+)\\
          & \hspace{0.5cm}- \mu(\rho_0) \BBD(\tBW_+)\ET(\oBn_{_{W_+}}- \Bn) 
           + \Pi_+\,\ET(\oBn_{_{W_+}}- \Bn_+). 
\end{align*}
In fact, $\tBk_{w_+,P_+}$ defined above is one desired extension by taking advantage of similar arguments for  $\tBh_{w,P}.$ Indeed, we can find
\begin{equation}\label{es:tkwP}
\|\tBk_{w_+,P_+}\|_{H^{1,1\slash 2}_{q,p,0,\ga_0} (\Omega_+\times \BBR)} 
\lesssim
\|\mu\|_{L_\infty(\dOm)} T^{s_{p,q}}L^2 
+\|\nabla \mu\|_{L_q(\dOm)}T^{1\slash {p'} + 2\sigma_{p,q}}L^2
 + T^{\min\{s_{p,q} ,1\slash {p'}\}} L^2.
\end{equation}
\medskip

Finally, putting the bounds \eqref{es:tfwP}, \eqref{es:gw},  \eqref{es:DtRw}, \eqref{es:thwP} and \eqref{es:tkwP} together yields
\begin{multline*}
\|(\tBf_{w,P}, \tg_w, \tBR_w,  \tBh_{w,P}, \tBk_{w_+,P_+})\|_{\CY_{p,q,\ga_0}} \lesssim \|\Bf\|_{L_p(0,T;L_q(\dOm))}
+ \|\eta-\rho_0\|_{L_\infty(\dOm \times ]0,T[)} L \\
+\|\mu\|_{L_\infty(\dOm)} T^{s_{p,q}}L^2 
+\|\nabla \mu\|_{L_q(\dOm)}T^{1\slash {p'} + 2\sigma_{p,q}}L^2
 + T^{\min\{\tsigma_{p,q}, s_{p,q} , 1\slash {p'} \}} L^2.
\end{multline*} 
Thus by taking $\eta-\rho_0$ and $T$ small enough as in \eqref{cdt:Tsmall1} and \eqref{cdt:Tsmall2}, we conclude
\begin{equation}\label{es:tUtQ}
\|\tBU\|_{W^{2,1}_{q,p,0,\gamma_0}  (\dOm \times \BBR)}
 + \|\nabla_\xi \tQ\|_{L_{p,0,\gamma_0}(\BBR; L_q(\dOm))}
 \leq cL + C_{p,q,\gamma_0} T^{\min\{\tsigma_{p,q}, s_{p,q} ,1\slash {p'} \}} L^2.
\end{equation}
Recall the definitions \eqref{eq:Xi} and \eqref{eq:Xi+} and we have
\begin{equation}\label{es:Xi}
\|\Xi\|_{H^{1,1\slash 2}_{q,p,0,\ga_0}(\dOm\times \BBR)} 
 +\|\Xi_+\|_{H^{1,1\slash 2}_{q,p,0,\ga_0}(\Omega_+\times \BBR)}
  \leq cL + C_{p,q,\gamma_0} T^{\min\{\tsigma_{p,q}, s_{p,q} ,1\slash {p'} \}} L^2.
\end{equation}
Therefore, according to \eqref{es:tUtQ} and \eqref{es:Xi}, 
$(\BU,Q,\Xi,\Xi_+) :=(\tBU|_{]0,T[},\tQ|_{]0,T[},\Xi,\Xi_+) $ belongs to  $\sE_{L}(T)$  for any
$(\Bw,P,\Pi, \Pi_+)$ in $\sE_{L}(T)$
as long as $T$ small.

\subsection{Stability of $\eqref{eq:UQ}$}\label{ssec:local_3}
In this part, we will show that the solution of \eqref{eq:UQ} is unique by contradiction arguments.
According to Section \ref{ssec:local_2}, suppose that $(\BU_k, Q_k),$ $k=1,2,$ are two distinct solutions of \eqref{eq:UQ} on $[0,T] \subset [0,1[$ such that 
\begin{equation*}
\|(\BU_1,\BU_2)\|_{W^{2,1}_{q,p}  (\dOm \times ]0,T[)}
 + \|\nabla (Q_1,Q_2)\|_{L_p(0,T; L_q(\dOm))} \leq L <\infty,
\end{equation*}
\begin{equation*}
\FL:=\|\delta \BU\|_{W^{2,1}_{q,p}  (\dOm \times ]0,T[)}
 + \|\nabla \delta Q\|_{L_p(0,T; L_q(\dOm))} >0
\end{equation*}
for $(\delta \BU, \delta Q) := (\BU_2-\BU_1,Q_2-Q_1).$ 
Thanks to \eqref{eq:UQ}, we can easily write the equations of $(\delta \BU, \delta Q)$ by
\begin{equation}\label{eq:delta_0}
	\left\{\begin{aligned}
 \pa_t \delta \BU - \eta^{-1}\Di_\xi \BBT(\delta\BU, \delta Q) = \delta \Bf ,\,\,\, \di_\xi \delta\BU = \delta g =\di_\xi \delta\BR 
		 &\quad\mbox{in}\quad \dOm \times ]0,T],  \\
		\jump{\BBT(\delta\BU, \delta Q)\Bn} = \jump{\delta\Bh }, \,\,\,\jump{\delta\BU}=\0 
		&\quad\mbox{on}\quad \Ga\times  ]0,T],  \\
		\BBT_+(\delta\BU_+, \delta Q_+)\Bn_+ = \delta\Bk  
		&\quad\mbox{on}\quad \Ga_{+}\times  ]0,T], \\
				\delta\BU_- = \0 &\quad\mbox{on}\quad \Ga_{-}\times  ]0,T], \\
	\delta\BU|_{t=0}= \0 &\quad\mbox{in}\quad \dOm,
	\end{aligned}\right.
\end{equation}
where 
$(\delta\Bf, \delta g, \delta\BR,  \delta\Bh, \delta\Bk)$ are given by
\begin{align*}
(\delta\Bf, \delta g, \delta\BR) &:= (\Bf_{u_2,\Fq_2} -\Bf_{u_1,\Fq_1}, g_{u_2}- g_{u_1}, \BR_{u_2}-\BR_{u_1}),\\
( \delta\Bh ,\delta\Bk)&:=(\Bh_{u_2,\Fq_2}-\Bh_{u_1,\Fq_1}, \Bk_{u_{2,+},\Fq_{2,+}} - \Bk_{u_{1,+},\Fq_{1,+}}), \\
(\Bu_k,\Fq_k)&:= (\BuL+ \BU_k, \FqL + Q_k).
\end{align*}
In fact, we claim that there exist suitable
$(\delta\tBf, \delta\tg, \delta\tBR,\delta\tBh,\delta\tBk ) \in \CY_{p,q,\ga_0}$ 
for some $\gamma_0 >0$ such that  
\begin{equation*}
(\delta\tBf, \delta\tg, \delta\tBR,\delta\tBh,\delta\tBk )|_{]0,T]} 
= (\delta\Bf, \delta g, \delta\BR,  \delta\Bh, \delta\Bk).
\end{equation*}
Then we can consider the following system,
\begin{equation}\label{eq:delta}
	\left\{\begin{aligned}
		 \pa_t \delta \tBU - \eta^{-1}\Di_\xi \BBT(\delta\tBU, \delta\tQ) = \delta \tBf ,\,\,\, \di_\xi \delta\tBU = \delta\tg =\di_\xi \delta\tBR 
		 &\quad\mbox{in}\quad \dOm \times \BBR_+,  \\
		\jump{\BBT(\delta\tBU, \delta\tQ)\Bn} = \jump{\delta\tBh }, \,\,\,\jump{\delta\tBU}=\0 &\quad\mbox{on}\quad \Ga\times \BBR_+,  \\
		\BBT_+(\delta\tBU_+, \delta\tQ_+)\Bn_+ = \delta\tBk  &\quad\mbox{on}\quad \Ga_{+}\times \BBR_+, \\
				\delta\tBU_- = \0 &\quad\mbox{on}\quad \Ga_{-}\times \BBR_+, \\
	\delta\tBU|_{t=0}= \0 &\quad\mbox{in}\quad \dOm.
	\end{aligned}\right.
\end{equation}
Thanks to Theorem \ref{thm:stokes},  there exists a unique solution $(\delta \tBU, \delta\tQ)$ of \eqref{eq:delta} such that 
\begin{equation*}
\FL = \|\delta \tBU\|_{W^{2,1}_{q,p}  (\dOm \times ]0,T[)}
 + \|\nabla \delta \tQ\|_{L_p(0,T; L_q(\dOm))} 
 \leq C_{p,q,\ga_0} \|(\delta\tBf, \delta\tg, \delta\tBR,\delta\tBh,\delta\tBk )\|_{\CY_{p,q,\ga_0}}.
\end{equation*}
On the other hand, the extensions $(\delta\tBf, \delta\tg, \delta\tBR,\delta\tBh,\delta\tBk)$ chosen before also fulfil 
\begin{equation}\label{claim:delta}
\|(\delta\tBf, \delta\tg, \delta\tBR,\delta\tBh,\delta\tBk )\|_{\CY_{p,q,\ga_0}} \leq (2C_{p,q,\ga_0})^{-1} \FL
\end{equation}
as long as $T$ small enough, which is a contradiction with $\FL >0.$ Therefore, the uniqueness of \eqref{eq:UQ} is valid.
\medskip

In order to prove the claim \eqref{claim:delta}, introduce some notations and a priori estimates for convenience.
So set that $\tBu_k:= \tBuL + \ET \BU_k$ for $k=1,2,$ and 
\begin{equation*}
(\delta\tBu, \delta \sA_{u}, \delta \oBn):=(\ET \delta \BU, \sA_{u_2} -\sA_{u_1}, \oBn_{u_2}-\oBn_{u_1}).
\end{equation*}
Besides, the facts below will be constantly used in the rest of this section,
\begin{equation}\label{es:decaydw}
\|\nabla_\xi \delta\BU\|_{L_{1} (0,T; L_\infty(G))} 
\lesssim  T^{1\slash {p'}}
\|\nabla_\xi \delta\BU\|_{L_{p} (0,T; L_\infty(G))}  
\lesssim T^{1\slash {p'}+ \sigma_{p,q}} \FL.
\end{equation}
Combining this decay property and Condition \eqref{cdt:Tsmall1}, we infer from Lemma \ref{lem:deltaA} and Lemma \ref{lem:normal} that
\begin{align}\label{ob:keydelta}
\|(\delta \sA_{u},\delta \oBn)\|_{L_\infty(\dOm\times]0,T[)} 
&\lesssim T^{1\slash {p'}+\sigma_{p,q}}\FL,&
\|\nabla_\xi (\delta \sA_{u},\delta \oBn)\|_{L_\infty(0,T;L_q(\dOm))}
&\lesssim T^{1\slash {p'}} \FL,\\ \nonumber
\|\pa_t (\delta \sA_{u},\delta \oBn)\|_{L_p(0,T;L_\infty(\dOm))}
 &\lesssim T^{\sigma_{p,q}}\FL,&
\|\pa_t \nabla_\xi (\delta \sA_{u},\delta \oBn)\|_{L_p(0,T;L_q(\dOm))}&\lesssim  \FL.
\end{align}

\subsubsection*{\underline{Bounds of $(\delta \tBf,\delta \tg,\delta \tBR)$}}
Recall the proof in Section \ref{ssec:local_2}, and we define 
\begin{align*}
\delta \tBBH := \tBBH_{u_2}-\tBBH_{u_1}&= \nabla_{\xi}^{\top} \ET \delta\BU \cdot\ET (\BBI-\sA_{u_2}^{\top}) 
              - \nabla_{\xi}^{\top}\tBu_1 \cdot \ET \delta \sA_{u}^{\top}\\
             & \quad + \ET (\BBI-\sA_{u_2})\cdot\nabla_\xi \ET \delta\BU^{\top}
             - \ET \delta \sA_{u} \cdot \nabla_\xi\tBu_1^{\top} ,\\
\delta \tBBD  := \tBBD_{u_2}-\tBBD_{u_1}
&=   \BBD(\ET \delta\BU) \cdot \ET (\BBI-\sA_{u_2} )
                  -\BBD(\tBu_1)\cdot  \ET \delta \sA_{u}.           
\end{align*}
Then it is natural to study 
\begin{align*}
\eta \,\delta\tBf &:= \rho_0 \ET \Big(\Bf\big( \BX_{u_2}(\xi,t) ,t\big) 
                        -\Bf\big( \BX_{u_1}(\xi,t) ,t\big)\Big)
                        + (\eta-\rho_0) \ET \pa_t \delta \BU\\
  &\quad  -\Di_\xi \big(\mu(\rho_0) (\delta\tBBH +\delta\tBBD) \big)
+\Di_\xi \Big(\mu(\rho_0) \big(\delta \tBBH \cdot  \ET (\BBI-\sA_{u_2})
       - \tBBH_{u_1} \cdot \ET \delta\sA_{u} \big) \Big) \\
 &\quad+\Di_\xi \Big(\ET \big( \delta Q (\BBI-\sA_{u_2})
       -  \Fq_1 \delta \sA_{u}   \big)  \Big),\\
       \delta\tg   &:= \nabla_{\xi}^{\top} E_{_{(T)}}\delta\BU:\ET (\BBI-\sA_{u_2}^{\top}) 
                -\nabla_{\xi}^{\top} \tBu_1: \ET \delta \sA_{u}^{\top},\\
\delta\tBR  &:=\ET (\BBI-\sA_{u_2}^{\top}) \,\ET \delta\BU 
                 - \ET \delta \sA_{u}^{\top}\,\tBu_1.
\end{align*}
According to \eqref{ob:keydelta} and \eqref{cdt:Tsmall1}, it is not hard to establish that
\begin{equation}\label{es:deltatf}
\|\delta\tBf\|_{L_{p,0,\ga_0}(\BBR;L_q(\dOm))} \lesssim
 c \FL +T^{1\slash {p'}+ \sigma_{p,q}}L \FL\big(\|\mu\|_{L_\infty(\dOm)}
 + T^{\sigma_{p,q}} \|\nabla \mu\|_{L_q(\dOm)} \big) 
  +T^{1\slash {p'}}L\FL,
\end{equation}
\begin{equation}\label{es:deltag}
\|\delta\tg\|_{H_{q,p,0,\ga_0}^{1,1\slash 2} (\dOm \times \BBR)}
 \lesssim T^{\min \{ 1\slash {p'} + \sigma_{p,q},s_{p,q} \}}L\FL
\lesssim T^{s_{p,q}}L\FL,
\end{equation}
\begin{equation}\label{es:deltaR}
\|\pa_t \delta \tBR\|_{L_{p,0,\ga_0}(\BBR;L_q(\dOm))}  
 \lesssim  T^{\tsigma_{p,q}} L \FL.
\end{equation}

\subsubsection*{\underline{Bounds of $(\delta \tBh, \delta \tBk)$}}
Thanks to the similar formulations of $\delta \tBh$ and $\delta \tBk,$ it is sufficient to study $\delta \tBh$ for simplicity.
Now we write out $\delta \tBh$  as below
\begin{align*}
\delta \tBh &= \mu(\rho_0)\,\delta\tBBH \,\Bn
               + \mu(\rho_0)\,\delta\tBBH \,\ET (\oBn_{u_2} -\Bn)
                       + \mu(\rho_0)\,\tBBH_{u_1} \,\ET\delta \oBn\\
& \quad -\mu(\rho_0) \,\BBD( \ET \delta \BU)\, \ET (\oBn_{u_2} -\Bn)
    -\mu(\rho_0)\, \BBD (\tBu_1) \,\ET \delta\oBn\\
& \quad -\delta \Pi\,\ET (\oBn_{u_2} -\Bn) + \Pi_1\, \ET \delta\oBn,   
\end{align*}
where $\delta \Pi := \Pi_2-\Pi_1$ and  
$$\Pi_k := \Big( \mu (\rho_0) \big(\BBD(\tBu_k)-\wt\BBH_{u_k} \big) \oBn_{u_k}\Big) \cdot \oBn_{u_k} 
\,\,\,\hbox{for}\,\, k=1,2.$$
Then we claim that
\begin{equation}\label{es:deltath}
\|\delta \tBh\|_{H^{1,1\slash 2}_{q,p,0,\ga_0}(\dOm\times\BBR)}\lesssim 
\big( \|\mu\|_{L_\infty(\dOm)} T^{s_{p,q}}L 
+\|\nabla \mu\|_{L_q(\dOm)}T^{s_{p,q}+ \sigma_{p,q}}L
 + T^{\min\{s_{p,q}, 1\slash {p'}\}} L \big)\FL.
\end{equation}
Based on our previous discussions and Lemma \ref{lem:normal}, it is not hard to show \eqref{es:deltath}. 
For instance, let us only consider the $H^{1\slash 2}_{p,0,\gamma_0}\big(\BBR;L_q(\dOm)\big)$ bound of 
$\Pi_1 \,\ET\delta \oBn$ for $(p,q) \in (II).$  According to Lemma \ref{lem:deltaA},\eqref{es:decaydw} and \eqref{cdt:Tsmall1}, we have
\begin{equation}\label{es:DtdeltaAW}
\|\pa_t \delta \sA_{_W}\|_{L_{p \slash \theta}(0,T;L_\beta(\dOm))}
\lesssim \FL
\end{equation}
with some $\theta$ and $\beta$ fulfilling $1-2(1-\theta)\slash p = N\slash q -N\slash \beta.$
Thus Lemma \ref{lem:normal}, \eqref{es:DtdeltaAW}, \eqref{es:DtAW} and \eqref{ob:keydelta} imply that
\begin{equation}\label{es:Dtdeltatn}
\|\pa_t \delta \oBn\|_{L_{p \slash \theta}(0,T;L_\beta(\dOm))}
\lesssim \FL + T^{1\slash {p'}+\sigma_{p,q}}L\FL \lesssim \FL.
\end{equation}
On the other hand, recall \eqref{es:thwP21} and we can infer from the conditions \eqref{cdt:Tsmall1} and \eqref{cdt:Tsmall2},
\begin{equation*}
\big\|\mu (\rho_0) \big(\BBD(\tBu_k)-\wt\BBH_{u_k} \big) \big\|_{H^{1,1\slash 2}_{q,p}(\dOm \times \BBR)} \lesssim L (\|\mu\|_{L_\infty(\dOm)} +T^{\sigma_{p,q}} \|\nabla \mu\|_{L_q(\dOm)}),
\end{equation*}
which immediately yields that
\begin{equation}\label{es:Pik_2}
\big\|\Pi_k \big\|_{H^{1,1\slash 2}_{q,p}(\dOm \times \BBR)} \lesssim L (\|\mu\|_{L_\infty(\dOm)} +T^{\sigma_{p,q}} \| \nabla \mu\|_{L_q(\dOm)}).
\end{equation}
Thus we can conclude the desired bound of $\Pi_1 \,\ET\delta \oBn$  from \eqref{ob:keydelta},\eqref{es:Dtdeltatn}, \eqref{es:Pik_2} and Lemma \ref{lem:Hhalf2},
\begin{equation*}
\|\Pi_1 \,\ET\delta \oBn\|_{H^{1\slash 2}_{p,0,\ga_0} (\BBR;L_q(\dOm))} \lesssim T^{s_{p,q}}L\FL 
 (\|\mu\|_{L_\infty(\dOm)} +T^{\sigma_{p,q}} \|\nabla \mu\|_{L_q(\dOm)})
 \,\,\,\hbox{for}\,\, (p,q) \in (II).
\end{equation*}
\medskip

Finally, combining the estimates \eqref{es:deltatf},  \eqref{es:deltag}, \eqref{es:deltaR} and \eqref{es:deltath}  implies that
\begin{align*}
\|(\delta\tBf, \delta\tg, \delta\tBR,  \delta\tBh, \delta\tBk)\|_{\CY_{p,q,\ga_0}} &\lesssim c \FL 
+\|\mu\|_{L_\infty(\dOm)} T^{s_{p,q}}L\FL
+\|\nabla \mu\|_{L_q(\dOm)}T^{s_{p,q} + \sigma_{p,q}}L\FL \\
& \quad + T^{\min\{\tsigma_{p,q}, s_{p,q}, 1\slash {p'} \}} L\FL. \hspace{3cm}
\end{align*}
Therefore \eqref{claim:delta} holds true by taking $c$ and $T$ small.

%%%%%%%%%%%%%%%%%%%%%%%%%%%%
%\newpage
%%%%%%%%%%%%%%%%%%%%%%%%%%%%
\section{Some exponential stability in the case of $(\Omega_1)$}\label{sec:decay}
In this section, we shall address some exponential decay property of two phase Stokes system under the physical setting $(\Omega_1).$ So  consider \eqref{eq:S} with assuming $\Gamma_- = \emptyset,$
\begin{equation}\label{eq:S2}
	\left\{\begin{aligned}
		 \pa_t\Bu -\eta^{-1} \Di \BBT(\Bu,\Fq) =  \Bf,\,\,\,
		\di\Bu = g=\di \BR    
		&\quad\mbox{in}\quad \dOm \times \BBR_+,  \\
		\jump{\BBT(\Bu,\Fq)\Bn} = \jump{\Bh}, \,\,\,\jump{\Bu}=\0 &\quad\mbox{on}\quad \Ga\times \BBR_+,  \\
		\BBT_+(\Bu_+,\Fq_+)\Bn_+ = \Bk &\quad\mbox{on}\quad \Ga_{+}\times \BBR_+, \\
	\Bu|_{t=0}= \Bu_0 &\quad\mbox{in}\quad \dOm.
	\end{aligned}\right.
\end{equation}
To study \eqref{eq:S2}, we introduce following functional spaces for convenience.
\begin{itemize}
\item  
Recall the rigid motion space $\CR_d$ and its basis $\FP$ used in Theorem \ref{thm:main_global}.
Then for any Banach space $E(\dOm)$ defined over $\dOm,$ 
\begin{equation*}
\tE (\dOm) := \{\Bu \in E(\dOm):  (\eta \Bu, \Bp_\alpha)_{\dOm} =0
\,\,\hbox{for any}\,\, \Bp_\alpha \in \FP
\,\,\hbox{and}\,\, 1 \leq \alpha \leq M\},
\end{equation*}
with the norm $\|\cdot\|_{\tE (\dOm)} := \|\cdot\|_{E (\dOm)}.$ For instance, the symbols $\wt{J}_q(\dOm)$ and $\wt{\CD}_{q,p}^{2-2\slash p}(\dOm)$ stand for the subspaces of $J_q(\dOm)$ and $\CD_{q,p}^{2-2\slash p}(\dOm)$ without rigid motion respectively.

\item In addition, we say $(\Bf,g,\BR,\Bh,\Bk) \in \CZ_{p,q,\ep}$ for some $1<p,q<\infty$ and $\ep >0,$ if $\Bf,g,\BR,\Bh$ and $\Bk$ satisfy the conditions,  
\begin{gather*}
e^{\ep t} \Bf \in L_{p,0}\big(\BBR; L_q(\dOm)\big)^N, \,\,
e^{\ep t} g \in H_{q,p,0}^{1,1\slash 2}(\dOm \times \BBR) \cap L_{p,0}\big(\BBR;\BW^{-1}_q(\Omega)\big), \,\,\\
e^{\ep t} (\pa_t\BR, \BR) \in L_{p,0}\big(\BBR; L_q(\dOm)\big)^{2N},\,
e^{\ep t} \Bh \in H_{q,p,0}^{1,1\slash 2}(\dOm \times \BBR)^N
\,\,\,\hbox{and}\,\,\,
e^{\ep t} \Bk \in H_{q,p,0}^{1,1\slash 2}(\Omega_+ \times \BBR)^N.
\end{gather*} 
Moreover, the norm $\|\cdot\|_{\CZ_{p,q,\ep}}$ is given by
\begin{align*}
\|(\Bf,g,\BR,\Bh,\Bk)\|_{\CZ_{p,q,\ep}}& :=
\|e^{\ep t} (\Bf, \BR, \pa_t \BR)\|_{L_p(\BBR_+; L_q(\dOm))} 
+\|e^{\ep t} (g, \Bh)\|_{H_{q,p}^{1,1\slash 2}(\dOm \times \BBR)} \\
& \quad +\|e^{\ep t} \Bk\|_{H_{q,p}^{1,1\slash 2}(\Omega_+ \times \BBR)}.
\end{align*}
\end{itemize}
With above symbols, our main results in this part on the decay properties in the framework of the $L_p-L_q$ maximal regularity  read as follows.
\begin{theo}\label{thm:decay_S2}
Assume that $1<p,q<\infty,$ $N<r<\infty$ and $r \geq \max\{q,q\slash (q-1)\}.$ 
Suppose that $\mu$ satisfies $(\CH 3')$ and $\Omega = \dOm \cup \Gamma$ be a bounded domain with the closed hypersurfaces $\Gamma,$ $\Gamma_+$ being of class $W^{2-1\slash r}_r.$
Let $\eta:= \eta_+ \mathds{1}_{\Omega_+}+\eta_- \mathds{1}_{\Omega_-}$ for some $\eta_\pm >0,$
$\Bu_0 \in \wt{\CD}_{q,p}^{2-2\slash p}(\dOm)$ and $(\Bf, g, \BR, \Bh, \Bk) \in \CZ_{p,q,\ep}$ for some $\ep>0.$
Then \eqref{eq:S2} admits a unique solution $(\Bu, \Fq)$ with 
\begin{equation*}
\Bu \in W^{2,1}_{q,p}(\dOm \times \BBR_+) 
\,\,\,\hbox{and}\,\,\,
\Fq \in L_p\big(\BBR_+; W^1_q(\dOm) + \widehat{W}^1_{q,\Gamma_+}(\Omega)\big).
\end{equation*}
Moreover, there exist constants $C$ and $\ep_0 \, (\leq \ep)$ such that
\begin{multline*}
\|e^{\ep_0 t} (\pa_t \Bu, \Bu, \nabla \Bu, \nabla^2 \Bu)\|_{L_p(0,T; L_q(\dOm))} 
+\|e^{\ep_0 t} \Fq\|_{L_p(0,T;W^1_q(\dOm))}
 \leq C \Big( \|\Bu_0\|_{\wt{\CD}_{q,p}^{2-2\slash p}(\dOm)} \\
 + \|(\Bf,g,\BR,\Bh,\Bk)\|_{\CZ_{p,q,\ep_0}}
 +  \sum_{\alpha=1}^M  \Big(  \int_0^T e^{p\ep_0t} |(\eta \Bu, \Bp_\alpha)_{\dOm}|^p dt   \Big)^{1\slash p} \Big).
\end{multline*}
for any $T>0.$
\end{theo}

\begin{rema}\label{rmk:othogonal}
Let us give some simple but useful comment on the orthogonal condition in Theorem \ref{thm:decay_S2}.
Assume that $\Bv$ is a smooth vector field in $\Omega$ such that $\jump{\Bv} = \0$ on $\Gamma.$ Then we have
\begin{equation}\label{eq:tensor_v}
\big(\Di \BBT(\Bu,\Fq), \Bv \big)_{\dOm} 
=\big(\jump{\BBT\big(\Bu,\Fq \big)\Bn}, \Bv \big)_{\Gamma} 
+ \big( \BBT_+\big(\Bu_+,\Fq_+ \big)\Bn_+, \Bv \big)_{\Gamma_+} 
- \frac{1}{2} \sum_{i,j=1}^N \int_{\dot{\Omega}} \BBT(\Bu, \Fq)^i_j \BBD(\Bv)^i_j \,dx.
\end{equation}
For any $\Bp_\alpha \in \FP,$  note the fact that $\BBD(\Bp_\alpha)^i_j =0.$
Then by multiplying $\eta \Bp_\alpha$ on the both sides of $\eqref{eq:S2}_1$ and integrating over $\dOm,$
we infer from \eqref{eq:tensor_v} that,
\begin{align*}
\pa_t (\eta \Bu, \Bp_\alpha)_{\dOm} 
 =(\jump{\Bh}, \Bp_\alpha)_{\Gamma}
    + (\Bk, \Bp_\alpha)_{\Gamma_+}
     + (\eta \Bf, \Bp_\alpha)_{\dOm},
\end{align*}
In particular, as long as 
$(\eta \Bu_0, \Bp_\alpha)_{\dOm} =
(\jump{\Bh}, \Bp_\alpha)_{\Gamma}  + (\Bk, \Bp_\alpha)_{\Gamma_+}
     + (\eta \Bf, \Bp_\alpha)_{\dOm} =0,$
the solution $\Bu$ of \eqref{eq:S2} satisfies $(\eta \Bu, \Bp_\alpha)_{\dOm} = 0$ for any $\alpha=1,...,M.$
\end{rema}

To prove Theorem \ref{thm:decay_S2}, we suppose that the solution of \eqref{eq:S2} can be decomposed by 
$$(\Bu,\Fq) = (\BuL, \Fq_{_L}) + (\Bw, P),$$
with $(\BuL, \Fq_{_L}) $ and  $(\Bw, P)$ solving the following systems respectively,
\begin{equation}\label{eq:S2_1}
	\left\{\begin{aligned}
	\pa_t\BuL  -  	\eta^{-1}\Di \BBT(\BuL,\Fq_{_L}) = \0, \,\,\,\di\BuL = 0   &\quad\mbox{in}\quad \dOm \times \BBR_+,  \\
		\jump{\BBT(\BuL,\Fq_{_L})\Bn} = \0, \,\,\,\jump{\BuL}=\0 
		&\quad\mbox{on}\quad \Ga \times \BBR_+,  \\
		\BBT(\BuL,\FqL)\Bn_+ = \0 
		&\quad\mbox{on}\quad \Ga_{+}\times \BBR_+, \\
	\BuL|_{t=0}= \Bu_0 &\quad\mbox{in}\quad \dOm,
	\end{aligned}\right.
\end{equation}
\begin{equation}\label{eq:S2_2}
	\left\{\begin{aligned}
		\pa_t\Bw -\eta^{-1} \Di \BBT(\Bw,P) = \Bf,  \,\,\, \di\Bw = g=\di \BR
		 &\quad\mbox{in}\quad \dOm \times \BBR_+,  \\
		\jump{\BBT(\Bw, P)\Bn} = \jump{\Bh}, \,\,\,\jump{\Bw}=\0 
		&\quad\mbox{on}\quad \Ga\times \BBR_+,  \\
		\BBT_+(\Bw,P)\Bn_+ = \Bk &\quad\mbox{on}\quad \Ga_{+}\times \BBR_+, \\
	\Bw|_{t=0}= \0 &\quad\mbox{in}\quad \dOm.
	\end{aligned}\right.
\end{equation}
In the rest of this section, we will treat \eqref{eq:S2_1} and \eqref{eq:S2_2} separately.

\subsection{Analysis of \eqref{eq:S2_1}}

\begin{theo}\label{thm:decay_S1}
Assume that $1<p,q<\infty,$ $N<r<\infty$ and $r \geq \max\{q,q\slash (q-1)\}.$ 
Suppose that $\mu$ satisfies $(\CH 3')$ and $\Omega = \dOm \cup \Gamma$ be a bounded domain with the closed hypersurfaces $\Gamma,$ $\Gamma_+$ being of class $W^{2-1\slash r}_r.$
Let $\eta:= \eta_+ \mathds{1}_{\Omega_+}+\eta_- \mathds{1}_{\Omega_-}$ for any fixed $\eta_\pm >0$ and
$\Bu_0 \in \wt{\CD}_{q,p}^{2-2\slash p}(\dOm).$
Then \eqref{eq:S2_1} admits a unique solution $(\Bu_{_L}, \Fq_{_L})$ with 
\begin{equation*}
\Bu_{_L} \in \wt W^{2,1}_{q,p}(\dOm \times \BBR_+) 
\,\,\,\hbox{and}\,\,\,
\Fq_{_L} \in L_p\big(\BBR_+; W^1_q(\dOm) + \widehat{W}^1_{q,\Gamma_+}(\Omega)\big).
\end{equation*}
Moreover, there exists positive constants $\ep_0$ and $C$ such that
\begin{equation*}
\|e^{\ep_0 t} (\pa_t \Bu_{_L}, \Bu_{_L} , \nabla \Bu_{_L} , \nabla^2 \Bu_{_L})\|_{L_p(\BBR_+; L_q(\dOm))} 
+\|e^{\ep_0 t} \Fq_{_L} \|_{L_p(\BBR_+;W^1_q(\dOm))}
\leq C  \|\Bu_0\|_{\wt{\CD}_{q,p}^{2-2\slash p}(\dOm)}.
\end{equation*}
\end{theo}

Thanks to \cite[Theorem 2.7] {MarSai2017} by taking $\Gamma_- =\emptyset,$ the Stokes operator $\CA_q$ generates some analytic $\CC_0$ semigroup
 $\big\{e^{-\CA_q t}\big\}_{t \geq 0}$ in $J_q(\dOm).$ 
Moreover, for some $0<\ep< \pi\slash 2 $ and $\lambda_0 >0,$ 
we have
\begin{equation*}
\Sigma_{\ep,\lambda_0} := \{ \lambda \in \BBC : |\arg \lambda | \leq \pi-\ep , |\lambda| \geq \lambda_0 \} \subset \,\, \hbox{the resolvent set}\,\, \rho(\CA_q).
\end{equation*}
On the other hand, as a direct consequce of Remark \ref{rmk:othogonal},  the closed subspace $\wt{J}_q(\dOm)$ is $e^{-\CA_q t}-$invariant,
 i.e. $e^{-\CA_q t} \wt{J}_q(\dOm) \subset \wt{J}_q(\dOm)$ 
 for any $t \geq 0.$
Therefore, if we set the restriction operator $\wt{\CA}_q  := \CA_q |_{\wt{J}_q(\dOm)}$  with its domain given by
$\CD(\wt{\CA}_q) := \CD(\CA_q) \cap \wt{J}_q(\dOm),$
then $\wt{\CA}_q $ is the generator of the induced $\CC_0-$semigroup 
$\big\{e^{-\wt{\CA}_q t}\big\}_{t \geq 0}:= 
\big\{e^{-\CA_q t}|_{\wt{J}_q(\dOm)}\big\}_{t \geq 0}.$
\medskip

Now, let us study the resolvent points of two phase Stokes operator. Thus we consider
\begin{equation}\label{eq:resolvent}
	\left\{\begin{aligned}
		\lambda \Bu - \eta^{-1}\Di \BBT\big(\Bu,K(\Bu)\big) = \Bf  
		&\quad\mbox{in}\quad \dOm ,  \\
		\jump{\BBT\big(\Bu,K(\Bu)\big)\Bn} = \0, \,\,\,\jump{\Bu}=\0 &\quad\mbox{in}\quad \Ga,  \\
		\BBT_+\big(\Bu_+,K(\Bu)_+\big)\Bn_+ = \0 &\quad\mbox{on}\quad \Ga_{+}.\\
	\end{aligned}\right.
\end{equation}
By the similar arguments as in Remark \ref{rmk:othogonal}, we have 
\begin{equation*}
\lambda(\eta \Bu, \Bp_\alpha)_{\dOm} = (\eta \Bf, \Bp_\alpha)_{\dOm}, 
\,\,\, \hbox{for any}\,\,\, \Bp_\alpha \in \FP,
\end{equation*}
which yields the $R(\lambda, \CA_q)$-invariance of subspace 
$\tJ_q(\dOm)$ for any $\lambda \in \Sigma_{\ep,\lambda_0},$ 
and thus $\Sigma_{\ep,\lambda_0} \subset \rho(\tCA_q).$ 
Therefore, combining above discussions yields the following property.
\begin{prop} \label{prop:analytic_wt}
Let $1<q<\infty,$ $N < r <\infty,$ and $\max \{ q, q'\}\leq r$ with $q':= q \slash (q-1).$ Suppose that $\mu$ satisfies $(\CH 3')$ and $\Omega$ is a bounded domain and $\Gamma, \Gamma_+$ are closed hypersurface of $W^{2-1\slash r}_r$ class.
Assume that $\eta:= \eta_+ \mathds{1}_{\Omega_+}+\eta_- \mathds{1}_{\Omega_-}$ for any fixed $\eta_\pm >0.$
Then the induced  Stokes operator $\tCA_q:= \CA_q |_{\tJ_q (\dOm)}$ generates a $\CC_0-$semigroup $\big\{e^{-\tCA_q t} \big\}_{t \geq 0}$ on $\tJ_q(\dOm),$ which is analytic.
\end{prop}

As $\dOm$ is bounded, we can show that $0 \in \rho (\tCA_q).$ To this end, let us start with the following property. 
\begin{lemm}\label{lem:unique_wt}
Let $q, \eta,$ $\mu$ and $\dot\Omega$ fulfil the same assumptions in Proposition \ref{prop:analytic_wt}. 
Suppose that $\Bu \in W^2_q(\dot \Omega)^N \cap \wt J_q(\dOm)$ satisfies 
\eqref{eq:resolvent} with $\Bf=0$ for some $\lambda \in\BBC \backslash\, ]-\infty,0[.$ Then $\Bu =\0.$ 
\end{lemm}
\begin{proof}

Rewrite $\eqref{eq:resolvent}_1$ by
$(2\lambda_0+|\lambda|) \Bu - \eta^{-1}\Di \BBT\big(\Bu,K(\Bu)\big) = (2\lambda_0+ |\lambda| - \lambda)  \Bu. $
By the bootstrap arguments and the estimates for resolvent problem \eqref{eq:resolvent} (see \cite[Theorem 1.6]{MarSai2017}), it is sufficient to check the case $q=2$ for any finite $N.$

Now we take $\Bv \in W^{1}_{2}(\dot\Omega)^N \cap J_2(\dOm)$ such that $\jump{\Bv}=\0$ on $\Gamma.$ 
Then \eqref{eq:tensor_v} implies that 
\begin{align}\label{es:L2_tensor}
 \Big(\lambda \eta \Bu - \Di \BBT\big(\Bu, K(\Bu)\big), \Bv\Big)_{\dot{\Omega}}
 = \lambda(\eta\Bu ,\Bv)_{\dOm} +\frac{1}{2} \sum_{i,j=1}^N \int_{\dot{\Omega}} \mu \BBD(\Bu)^i_j \BBD(\Bv)^i_j \,dx =0.
\end{align}
Take $\Bv=\Bu$ in \eqref{es:L2_tensor},
\begin{equation*}
\Re \lambda \|\sqrt{\eta} \Bu \|_{L_2(\dOm)}^2 + \frac{1}{2} \|\sqrt{\mu} \BBD(\Bu) \|_{L_2(\dOm)}^2  + i\Im \lambda \|\sqrt{\eta} \Bu \|_{L_2(\dOm)}^2 =0.
\end{equation*}
Hence $\Bu =\0$ for $\Re \lambda >0$ or $\Im \lambda \not=0.$
Otherwise, if $\lambda =0,$
then $\Bu_\pm \in \CR_d.$ Noting that $\jump{\Bu} =\0,$ we have $\Bu \in \CR_d.$
Finally, the fact that $(\eta\Bu,\Bp_\alpha)=0$ yields $(\sqrt{\eta} \Bu,\sqrt{\eta} \Bu)_{\dOm}=0$ and thus $\Bu =\0.$
\end{proof}

By Lemma \ref{lem:unique_wt}, we can establish the following result.
\begin{prop}\label{prop:resolvent_wtA}
Let $q, \eta$ $\mu$ and $\dot\Omega$ fulfil the same assumptions in Proposition \ref{prop:analytic_wt} and $0<\ep < \pi \slash 2.$ 
Then the resolvent set $ \rho (\wt \CA_q)$ of $\wt\CA_q$ contains $\Sigma_{\ep} \cup \{0\}.$
\end{prop}
\begin{proof}
By Proposition \ref{prop:analytic_wt},  there exists $\lambda_0>0$ such that $\Sigma_{\ep, \lambda_0} \subset \rho (\wt \CA_q)$ for $0 < \ep< \pi \slash 2.$ It remains to show our result for
$\lambda \in  (\Sigma_{\ep} \backslash \Sigma_{\ep, \lambda_0} ) \cup \{0\}.$
As $2\lambda_0 \in \Sigma_{\ep, \lambda_0},$ we denote 
$$R_0 := R(2\lambda_0;\wt \CA_q) : \wt J_q(\dOm) \rightarrow  W^2_q(\dOm) \cap \wt J_q(\dOm) \subset \wt J_q(\dOm).$$
$\dOm$ is bounded and then $R_0$ is compact operator from $\wt J_q(\dOm) $ into itself by Rellich's Theorem.
For any $\lambda \in  (\Sigma_{\ep} \backslash \Sigma_{\ep, \lambda_0} ) \cup \{0\},$  take 
$\Bg \in Ker \big( (\lambda -2\lambda_0)^{-1} + R_0 \big) \subset \wt J_q(\dOm),$ that is 
\begin{equation}\label{eq:Ker}
(\lambda- 2\lambda_0)^{-1} \Bg + R_0 \Bg =\0 \,\,\, \hbox{for any} 
\,\,\,\Bg \in \wt J_q (\dOm) .
\end{equation}
By definition of $R_0,$ $\Bu :=R_0 \Bg $ satisfies \eqref{eq:resolvent} with $\Bf =0.$ Thus Lemma \ref{lem:unique_wt} yields that $\Bu =R_0\Bg =0$ and then $\Bg=0$ by  \eqref{eq:Ker}. 
Thanks to the Fredholm Alternative Theorem, we can conclude the desired results. 
\end{proof}

In particular, Proposition \ref{prop:resolvent_wtA} implies that there exists $\ep_0 >0$ such that 
\begin{equation} \label{es:Anal_1}
 \|\Bu_{_L}(\cdot, t )\|_{\tJ_q (\dOm)} \leq C e^{-2\ep_0 t}
  \|\Bu_0\|_{\tJ_q(\dOm)},\,\,\,
\forall \Bu_0 \in \tJ_q(\dOm).
\end{equation}
Then we can prove Theorem \ref{thm:decay_S1}  by standard arguments in \cite[Theorem 3.9]{ShiShi2008}.

\subsection{Analysis of \eqref{eq:S2_2}}
This subsection is dedicated to the study of \eqref{eq:S2_2} and the result is summarized by the following theorem.
\begin{theo}\label{thm:decay_S22}
Assume that $1<p,q<\infty,$ $N<r<\infty$ and $r \geq \max\{q,q\slash (q-1)\}.$ 
Suppose that $\mu$ satisfies $(\CH 3')$ and $\Omega = \dOm \cup \Gamma$ be a bounded domain with the closed hypersurfaces $\Gamma,$ $\Gamma_+$ being of class $W^{2-1\slash r}_r.$
Let $\eta:= \eta_+ \mathds{1}_{\Omega_+}+\eta_- \mathds{1}_{\Omega_-}$ for any $\eta_\pm >0$ and $(\Bf, g, \BR, \Bh, \Bk) \in \CZ_{p,q,\ep_0}$ for the same constant $\ep_0$ in Theorem \ref{thm:decay_S1}.
Then \eqref{eq:S2_2} admits a unique solution $(\Bw, P)$ with 
\begin{equation*}
\Bw \in  W^{2,1}_{q,p}(\dOm \times \BBR_+) 
\,\,\,\hbox{and}\,\,\,
P \in L_p\big(\BBR_+; W^1_q(\dOm) + \widehat{W}^1_{q,\Gamma_+}(\Omega)\big).
\end{equation*}
Moreover, there exists a constant $C$ such that
\begin{multline*}
\|e^{\ep_0 t} (\pa_t \Bw, \Bw, \nabla \Bw, \nabla^2 \Bw)\|_{L_p(0,T; L_q(\dOm))} 
+\|e^{\ep_0 t} P\|_{L_p(0,T;W^1_q(\dOm))}\\
 \leq C  \|(\Bf,g,\BR,\Bh,\Bk)\|_{\CZ_{p,q,\ep_0}}
 + C \sum_{\alpha=1}^M  \Big(  \int_0^T e^{p\ep_0t} |(\eta\Bw, \Bp_\alpha)_{\dOm}|^p dt   \Big)^{1\slash p}
\end{multline*}
for any $T>0.$
\end{theo}

To prove Theorem \ref{thm:decay_S22}, recall the weak problem \eqref{def:welliptic} for any $\Bf \in L_q(\Omega)^N.$  Set that $Q_{q,\eta} \Bf := \eta^{-1}\nabla \theta$ and  $P_{q,\eta}\Bf:= \Bf -Q_{q,\eta} \Bf \in J_q(\dOm)$
with  $\theta \in \widehat{W}^1_{q,\Gamma_+}(\Omega)$ satisfying
\begin{equation*}
(\eta^{-1} \nabla \theta, \nabla \varphi)_{\dOm} = (\Bf,\nabla \varphi)_{\Omega}
\,\,\, \hbox{for all} \,\,\, \varphi \in \widehat{W}^1_{q',\Gamma_+}(\Omega).
\end{equation*}
Thanks to the definition of $ \widehat{W}^1_{q,\Gamma_+}(\Omega)$ and the Divergence Theorem, we have
\begin{equation}\label{eq:Qf}
(\eta Q_{q,\eta} \Bf, \Bp_\alpha)_{\dOm}=(\nabla \theta, \Bp_\alpha)_{\dOm} =0.
\end{equation}
Then \eqref{eq:Qf} implies that $(\eta P_{q,\eta}\Bf, \Bp_\alpha)_{\dOm}=0,$ as long as $(\eta \Bf, \Bp_\alpha)_{\dOm}= 0$ for any $\alpha=1,\dots,M.$ 
On the other hand, write $\wt \Bw:= \Bw - \sum_{\alpha=1}^M  (\eta \Bw, \Bp_\alpha)_{\dOm} \,\Bp_\alpha$ for any vector $ \Bw.$
Next we consider the following systems with above notations,
\begin{equation}\label{eq:S2_21}
	\left\{\begin{aligned}
		\pa_t\Bw_1 +2\lambda_0 \Bw_1- \eta^{-1}\Di \BBT(\Bw_1,P_1) =  \Bf,  \,\,\, \di\Bw_1 = g=\di \BR &\quad\mbox{in}\quad \dOm \times \BBR_+,  \\
		\jump{\BBT(\Bw_1, P_1)\Bn} = \jump{\Bh}, \,\,\,\jump{\Bw_1}=\0 
		&\quad\mbox{on}\quad \Ga\times \BBR_+,  \\
		\BBT_+(\Bw_1,P_1)\Bn_+ = \Bk &\quad\mbox{on}\quad \Ga_{+}\times \BBR_+, \\
	\Bw_1|_{t=0}= \0 &\quad\mbox{in}\quad \dOm,
	\end{aligned}\right.
\end{equation}
\begin{equation}\label{eq:S2_22}
	\left\{\begin{aligned}
		\pa_t\Bw_2 +2\lambda_0 \Bw_2-\eta^{-1} \Di \BBT(\Bw_2,P_2) =2\lambda_0 Q_{q,\eta} \wt \Bw_1,  \,\,\, \di\Bw_2 = 0 &\quad\mbox{in}\quad \dOm \times \BBR_+,  \\
		\jump{\BBT(\Bw_2, P_2)\Bn} = \jump{\Bw_2}=\0 
		&\quad\mbox{on}\quad \Ga\times \BBR_+,  \\
		\BBT_+(\Bw_2,P_2)\Bn_+ = \0 
		&\quad\mbox{on}\quad \Ga_{+}\times \BBR_+, \\
	\Bw_2|_{t=0}= \0 &\quad\mbox{in}\quad \dOm,
	\end{aligned}\right.
\end{equation}
\begin{equation}\label{eq:S2_23}
	\left\{\begin{aligned}
		\pa_t\Bw_3 - \eta^{-1}\Di \BBT(\Bw_3,P_3) =2\lambda_0  (P_{q,\eta}\wt \Bw_1 +\Bw_2),  \,\,\, \di\Bw_3 = 0
		 &\quad\mbox{in}\quad \dOm \times \BBR_+,  \\
		\jump{\BBT(\Bw_3, P_3)\Bn} =\jump{\Bw_3}=\0 
		&\quad\mbox{on}\quad \Ga\times \BBR_+,  \\
		\BBT_+(\Bw_3,P_3)\Bn_+ = \0 &\quad\mbox{on}\quad \Ga_{+}\times \BBR_+, \\
	\Bw_3|_{t=0}= \0 &\quad\mbox{in}\quad \dOm.
	\end{aligned}\right.
\end{equation}
Now set that 
$\Bw:= \sum_{\ell=1}^3 \Bw_\ell+\sum_{\alpha=1}^M \int_0^t 2\lambda_0 \big(\eta \Bw_1(s), \Bp_\alpha\big)_{\dOm}\, \Bp_\alpha ds $
and  $P:=  P_1+P_2+P_3.$    
Then the fact $\BBD(\Bp_\alpha) =\0$ implies that
$(\Bw, P)$ is (at lease formally) a solution of \eqref{eq:S2_2}.
Thus it is sufficient to construct the solutions of \eqref{eq:S2_21}-\eqref{eq:S2_23} one by one.

\subsubsection*{\underline{Study of \eqref{eq:S2_21}}}
The result on the solutions of \eqref{eq:S2_21} is given by the following proposition.
\begin{prop}\label{prop:decay_S2_21}
Let $\eta,$ $\mu$ and $\dot\Omega$ fulfil the same assumptions in Theorem \ref{thm:decay_S2}. 
Assume that $\Bf, g, \BR, \Bh, \Bk$ belong to $\CZ_{p,q,\ep_0}$ for some $0<\ep_0 (< \lambda_0).$
Then \eqref{eq:S2_21} admits a unique solution $(\Bw_1, P_1)$ with 
\begin{equation*}
\Bw_1 \in  W^{2,1}_{q,p}(\dOm \times \BBR_+) 
\,\,\,\hbox{and}\,\,\,
P_1 \in L_p\big(\BBR_+; W^1_q(\dOm) + \widehat{W}^1_{q,\Gamma_+}(\Omega)\big).
\end{equation*}
Moreover, there exists a constant $C$ such that
\begin{equation*}
\|e^{\ep_0 t} (\pa_t \Bw_1, \Bw_1, \nabla \Bw_1, \nabla^2 \Bw_1)\|_{L_p(\BBR_+; L_q(\dOm))} 
+\|e^{\ep_0 t} P_1\|_{L_p(\BBR_+;W^1_q(\dOm))}
\leq C\|(\Bf,g,\BR,\Bh,\Bk)\|_{\CZ_{p,q,\ep_0}}.
\end{equation*}
In addition, if $\Bf, \Bh$ and $\Bk$ satisfies 
 $(\jump{\Bh}, \Bp_\alpha)_{\Gamma}
    + (\Bk, \Bp_\alpha)_{\Gamma_+}
     + (\eta \Bf, \Bp_\alpha)_{\dOm} =0$
for any $\Bp_\alpha$ in $\FP,$ then we have $(\eta \, \Bw_1, \Bp_\alpha)_{\dOm}=0.$
\end{prop}
\begin{proof}
The construction of the solution of \eqref{eq:S2_21}  is similar to \cite[Lemma 4.6]{Sai2016}. Here we only check 
the fact that $(\eta \Bw_1, \Bp_\alpha)_{\dOm}  =0$ for any $t>0.$
By the Remark \ref{rmk:othogonal}, we have for any $\Bp_\alpha \in \FP$ and any $t>0,$
\begin{equation*}
\frac{d}{dt} (\eta \Bw_1, \Bp_\alpha)_{\dOm} +2\lambda_0 (\eta \Bw_1, \Bp_\alpha)_{\dOm} =(\jump{\Bh}, \Bp_\alpha)_{\Gamma}
    + (\Bk, \Bp_\alpha)_{\Gamma_+}
     + (\eta \Bf, \Bp_\alpha)_{\dOm} =0.
\end{equation*}
Thus $\frac{d}{dt} \big(e^{2\lambda_0 t}(\eta \Bw_1, \Bp_\alpha)_{\dOm} \big)=0$ yields our desired result. 
\end{proof}

\subsubsection*{\underline{Study of \eqref{eq:S2_22}}}
By Proposition \ref{prop:decay_S2_21}, we can easily establish the solution of \eqref{eq:S2_22}. Indeed,
note that $(\eta \wt\Bw_1, \Bp_\alpha)_{\dOm}  =0$ by the definition of $\wt \Bw_1.$ 
Then 
$(\eta Q_{q,\eta}\wt \Bw_1, \Bp_\alpha)_{\dOm} = (\eta P_{q,\eta}\wt \Bw_1, \Bp_\alpha)_{\dOm} =0$ by our previous comments.
Thanks to Proposition \ref{prop:decay_S2_21}, \eqref{eq:S2_22} admit a unique solution 
\begin{equation*}
(\Bw_2,P_2) \in \wt W^{2,1}_{q,p}(\dOm \times \BBR_+) 
\times  L_p\big(\BBR_+; W^1_q(\dOm) + \widehat{W}^1_{q,\Gamma_+}(\Omega)\big).
\end{equation*}
Moreover, we can bound $(\Bw_2,P_2)$ by
\begin{multline} \label{es:w2_exp}
\|e^{\ep_0 t} (\pa_t \Bw_2, \Bw_2, \nabla \Bw_2, \nabla^2 \Bw_2)\|_{L_p(\BBR_+; L_q(\dOm))} 
+\|e^{\ep_0 t} P_2\|_{L_p(\BBR_+;W^1_q(\dOm))}\\ 
\lesssim\|2\lambda_0 e^{\ep_0 t} Q_{q,\eta} \wt \Bw_1 \|_{L_p(\BBR_+; L_q(\dOm))} 
\lesssim \|(\Bf,g,\BR,\Bh,\Bk)\|_{\wt\CZ_{p,q,\ep_0}}.
\end{multline}

\subsubsection*{\underline{Study of \eqref{eq:S2_23}}}  
With $\Bw_1$ and $\Bw_2$ at hand, let us check \eqref{eq:S2_23}.
As $\di \Bw_2 =0,$  $\jump{\Bw_2} =0$ on $\Gamma$ and  $(\eta P_{q,\eta}\Bw_1 ,\Bp_\alpha)_{\dOm}= ( \eta\Bw_2,\Bp_\alpha)_{\dOm}=0$ for any $\alpha =1,\dots,M,$  we have 
$\BW:=P_{q,\eta}\wt \Bw_1 +\Bw_2 \in \wt J_q(\dOm).$ 
Moreover, Proposition \ref{prop:decay_S2_21} and  \eqref{es:w2_exp} yield that 
\begin{equation}\label{es:W_w3}
\| e^{\ep_0 t} \BW \|_{L_p(\BBR_+; L_q(\dOm))} 
\lesssim \|(\Bf,g,\BR,\Bh,\Bk)\|_{\CZ_{p,q,\ep_0}}.
\end{equation}
On the other hand, we infer from the Duhamel principle that
\begin{equation*}
\Bw_{3} (\cdot, t)=2\lambda_0  \int_0^t  e^{(t-s)\CA_q}\BW  (\cdot,s) \,ds.
\end{equation*}
By \eqref{es:Anal_1} and above formula of $\Bw_3,$  we have 
\begin{equation*}
e^{\ep_0 t}\|\Bw_{3}(t) \|_{L_q(\dOm)} \leq 2 \lambda_0 C \int_0^t  e^{-\ep_0(t-s)} \big( e^{\ep_0 s}\|\BW(s)\|_{L_q(\dOm)} \big) \,ds,
\end{equation*}
which, combining Young's inequality and  \eqref{es:W_w3}, implies that 
\begin{equation*}
\|e^{\ep_0 t} \Bw_3\|_{L_p(\BBR_+; L_q(\dOm))}  \lesssim  \frac{\lambda_0}{\ep_0}
\|e^{\ep_0 t} \BW \|_{L_p(\BBR_+; L_q(\dOm))} 
\lesssim \|(\Bf,g,\BR,\Bh,\Bk)\|_{\CZ_{p,q,\ep_0}}.
\end{equation*}
Then applying Proposition \ref{prop:decay_S2_21} to 
\begin{equation*}
\pa_t\Bw_3 +2\lambda_0 \Bw_3- \eta^{-1}\Di \BBT(\Bw_3,P_3) =2\lambda_0  (P_{q,\eta}\Bw_1 +\Bw_2) +2\lambda_0 \Bw_3
\end{equation*}
furnishes the desired bounds of $\Bw_3$ and $P_3.$ That is, 
\begin{equation}\label{es:w3_exp}
\|e^{\ep_0 t} (\pa_t \Bw_3, \Bw_3, \nabla \Bw_3, \nabla^2 \Bw_3)\|_{L_p(\BBR_+; L_q(\dOm))} 
+\|e^{\ep_0 t} P_3\|_{L_p(\BBR_+;W^1_q(\dOm))}
\leq C\|(\Bf,g,\BR,\Bh,\Bk)\|_{\CZ_{p,q,\ep_0}}.
\end{equation}
\medskip

Finally, let us derive the estimates of $(\Bw, P).$ In fact, it is sufficient to study  
$$\Bw= \sum_{\ell=1}^3 \Bw_\ell+\sum_{\alpha=1}^M \int_0^t 2\lambda_0 \big(\eta \Bw_1(s), \Bp_\alpha\big)_{\dOm}\, \Bp_\alpha ds,$$ 
since the bound of $P=\sum_{\ell=1}^3 P_\ell$ is immediate from the above discussions.
As an consequence of the second Korn's inequality, we first note that
$$\|\Bw\|_{W^1_q(\Omega)} \leq C_{\Omega,q} \big( \|\BBD(\Bw)\|_{L_q(\Omega)} + \sum_{\alpha=1}^M |(\eta \Bw, \Bp_\alpha)_{\Omega}| \big)$$
 for any $C^1$ bounded domain $\Omega.$
Then Proposition \ref{prop:decay_S2_21}, \eqref{es:w2_exp}  and \eqref{es:w3_exp} imply that
\begin{equation}\label{es:w_low_exp}
\|e^{\ep_0 t} (\pa_t\Bw,\Bw,\nabla \Bw)\|_{L_p(0,T; L_q(\dOm))} 
 \leq C  \|(\Bf,g,\BR,\Bh,\Bk)\|_{\CZ_{p,q,\ep_0}}
 + C \sum_{\alpha=1}^M  \Big(  \int_0^T e^{p\ep_0t} |(\eta \Bw, \Bp_\alpha)_{\dOm}|^p dt   \Big)^{1\slash p}
\end{equation}
for any $T>0.$
Then apply Theorem 1.6 in \cite{MarSai2017} to \eqref{eq:S2_2} with $\lambda =2\lambda_0,$ and we obtain that
\begin{equation*}
\| \nabla^2 \Bw\|_{ L_q(\dOm)} 
\leq C\|(\Bf,g,\nabla g,\BR,\Bh,\nabla \Bh,\Bk,\nabla \Bk, \pa_t \Bw,\Bw)\|_{L_q(\dOm)} ,
\end{equation*}
from which the bound of $\nabla^2 \Bw$ is attained by \eqref{es:w_low_exp}.
\medskip

\begin{proof}[The completion of the proof of Theorem \ref{thm:decay_S2}] Thanks to Theorem \ref{thm:decay_S1} and Theorem \ref{thm:decay_S22}, there exist constants $\ep_0,C >0$ such that 
\begin{multline*}
\|e^{\ep_0 t} (\pa_t \Bu, \Bu, \nabla \Bu, \nabla^2 \Bu)\|_{L_p(0,T; L_q(\dOm))} 
+\|e^{\ep_0 t} \Fq\|_{L_p(0,T;W^1_q(\dOm))}
 \leq C \Big( \|\Bu_0\|_{\wt{\CD}_{q,p}^{2-2\slash p}(\dOm)} \\
 + \|(\Bf,g,\BR,\Bh,\Bk)\|_{\CZ_{p,q,\ep_0}}
 +  \sum_{\alpha=1}^M  \Big(  \int_0^T e^{p\ep_0t} |(\eta \Bw, \Bp_\alpha)_{\dOm}|^p dt   \Big)^{1\slash p} \Big).
\end{multline*}
for any $T>0.$ Then Theorem \ref{thm:decay_S1}  and the following inequality
\begin{equation*}
 \Big( \int_0^T e^{p\ep_0t} |(\eta \Bw, \Bp_\alpha)_{\dOm}|^p dt   \Big)^{1\slash p} 
 \leq  \Big( \int_0^T e^{p\ep_0t} |(\eta \Bu, \Bp_\alpha)_{\dOm}|^p dt   \Big)^{1\slash p} 
 + \|e^{\ep_0 t}  \BuL\|_{L_p(0,T;L_q(\dOm))} \|\eta\, \Bp_\alpha\|_{L_{q'}(\dOm)},
\end{equation*}
yield the desired result.
\end{proof}

%%%%%%%%%%%%%%%%%%%%%%%%%%%%%%%%
%\newpage
%%%%%%%%%%%%%%%%%%%%%%%%%%%%%%%%

\section{Global solvability of \eqref{eq:INSL} in $(\Omega_1)$}\label{sec:global}
In this section, we would like to tackle the long time issue of \eqref{eq:INSL} under the case of $(\Omega_1)$ with the piecewise constant density $\rho_0 = \eta$ and the external force $\Bf=\0.$ 
Firstly, some useful auxiliary results for the global issue are given. Next, we outline the main idea to establish the global in time solution by admitting some a priori estimates of the solutions. At last, these a priori bounds are checked. 

\subsection{Some auxiliary estimates}
Let us first derive some useful properties for the system \eqref{eq:INS} in Eulerian coordinates.
For convenience, we recall here by assuming $\Gamma_- = \emptyset$ but general density.
\begin{equation}\label{eq:INS2}
	\left\{\begin{aligned}
		\pa_t(\rho \Bv)+\Di (\rho \Bv \otimes \Bv)  - \Di \BBT(\Bv,\Fp) = \rho \Bf  &\quad\mbox{in}\quad \dOm_t,   \\
		\pa_t \rho + \di (\rho \Bv) =0,  \,\,\,\di\Bv = 0    &\quad\mbox{in}\quad \dOm_t,  \\
		\jump{\BBT(\Bv,\Fp)\Bn_t} = \jump{\Bv}=\0, \,\,\,V_t=\Bv\cdot \Bn_t &\quad\mbox{on}\quad \Ga_t,  \\
		\BBT(\Bv_+,\Fp_+)\Bn_{+,t} = \0 ,\,\,\,V_{+,t}=\Bv_+\cdot \Bn_{+,t} &\quad\mbox{on}\quad \Ga_{+,t}, \\
		(\rho, \Bv)|_{t=0}=(\rho_0, \Bv_0) &\quad\mbox{in}\quad \dOm,
	\end{aligned}\right.
\end{equation}
Concerning \eqref{eq:INS2}, we have the following results.
\begin{prop}\label{prop:conserve}
For the smooth solution $(\rho, \Bv, \Fp)$ of \eqref{eq:INS2}, the following assertions hold true.
\begin{enumerate}
\item Define that $\BBA:\BBB := trace (\BBA \BBB)$ and we have
 \begin{equation*}
\frac{1}{2}\frac{d}{dt} \int_{\dot \Omega_t} \rho |\Bv|^2 \, dx
+\frac{1}{2} \int_{\dOm_t} \mu \BBD(\Bv) :\BBD(\Bv) \,dx
=(\rho \Bf, \Bv)_{\dOm_t}.
\end{equation*}
\item For any $\Ba \in \CR_d,$ we have
\begin{equation*}
\frac{d}{dt} \int_{\dot\Omega_t} \rho \Bv \cdot \Ba \ dx = \int_{\Omega_t}  \rho \Bf \cdot \Ba \,dx.
\end{equation*}
In particular, if $\Bf= \0,$ then $\int_{\Omega_t} \rho \Bv \cdot \Ba = \int_{\Omega} \rho_0 \Bv_0 \cdot \Ba \,d \xi.$
\item The barycenter satisfies
\begin{equation*}
\frac{d}{dt} \int_{\dot\Omega_t} \rho \Bx \ dx= \int_{\dot \Omega_t} \rho \Bv \,dx
\end{equation*}
\end{enumerate}
\end{prop}
\begin{proof}
Assume that $\Ba$ is any smooth vector field defined on $\dOm_t$ and $D_t:= \pa_t + \Bv \cdot \nabla$ stands for the material derivative. 
By Lagrangian mapping and incompressibility condition, we have
\begin{equation}\label{eq:energy_1}
\frac{d}{dt} \int_{\dot \Omega_t} \rho \Bv \cdot \Ba \, dx
= \frac{d}{dt} \int_{\dOm} (\rho \Bv \cdot \Ba ) \big(\BX_u(\xi,t),t \big) \ d\xi  
= \int_{\dOm_t} \rho\big(  D_t \Bv \cdot \Ba + \Bv \cdot D_t \Ba \big) \, dx.
\end{equation}
Set $\Ba: = \Bv$ in \eqref{eq:energy_1} and we infer from \eqref{eq:INS2} that,
\begin{equation*}
\frac{1}{2}\frac{d}{dt} \int_{\dot \Omega_t} \rho |\Bv|^2 \, dx
= \big( \Di \BBT(\Bv,\Fp) + \rho \Bf, \Bv \big)_{\dOm_t} 
=(\rho \Bf, \Bv)_{\dOm_t}-\frac{1}{2} \int_{\dOm_t} \mu \BBD(\Bv) :\BBD(\Bv) \,dx.
\end{equation*}
On the other hand, take $\Ba  \in \CR_d$ in \eqref{eq:energy_1}, i.e. $\BBD(\Ba)=0,$ and we see
\begin{align*}
\frac{d}{dt} \int_{\dot \Omega_t} \rho \Bv \cdot \Ba \, dx
&= \big( \Di \BBT(\Bv,\Fp) + \rho \Bf, \Ba \big)_{\dOm_t}  + \frac{1}{2} \int_{\Omega_t} \rho (\Bv \otimes \Bv) : \BBD(\Ba) \,dx 
= \int_{\Omega_t}  \rho \Bf \cdot \Ba dx,
\end{align*}
where $\BBA : \nabla \Ba^\top = \sum_{i,j=1}^N A^i_j \pa_j a^i = \big( \BBA: \BBD(\Ba) \big)\slash 2 =0$ for any  symmetric  matrix $\BBA.$
Furthermore, the variation of the barycenter is due to the conservation of mass by similar proof of \eqref{eq:energy_1}.
\end{proof}
Apart from Proposition \ref{prop:conserve},  another useful tool is the following bootstrap argument.
\begin{lemm}\label{lem:bootstrap}
Assume that $X(t) \geq 0$ is a continuous function on $[0,T] \subset [0,\infty[$ satisfying 
$$X(t) \leq a + b X(t)^2 + b X(t)^{3}
\quad \forall \, t \in [0,T],$$
where $a,b>0$ such that 
\begin{equation}\label{eq:arb}
a < r_b(2-b\, r_b)  \slash 3, \quad  X(0) \leq r_b, \quad r_b := (-1 + \sqrt{1 + 3b^{-1} })\slash 3.
\end{equation}
Then we have $X(t) \leq 2 a.$
\end{lemm}
\begin{proof}
Set $f(x):= bx^3 +bx^2-x +a $ and note the fact that $f'(r_b)=3br_b^2 + 2b r_b-1 =0$ by the definition of $r_b.$ In other words, $f(x)$ decreases on $[0,r_b]$ and increases on $]r_b,T].$ Moreover, $f(r_b) <0$ by \eqref{eq:arb}. Then assume that $x_0$ is the (unique) root of $f$ on $]0, r_b[.$ Furthermore, we have 
$$a = x_0 - b(x_0^3 +x_0^2) = x_0 \Big( 1 - \frac{x_0^2 +x_0}{3r_b^2 + 2 r_b} \Big) > \frac{x_0}{2},$$
which yields the desired result by continuity of $X(t).$
\end{proof}

%%%%%%%%%%%%%%%%%%%%%%%%%%%%%%%%%%%%%%%%%%%%%%%%%%%
\subsection{Construction of global solution}
From now on, assume that $\Omega = \dOm \cup \Gamma$ is some $W^{2-1\slash r}_r$ ($r\geq q$) bounded droplet as in $(\Omega_1),$ and $\rho = \eta = \eta_+\mathds{1}_{\Omega_+} +\eta_-\mathds{1}_{\Omega_-}$ for some $\eta_\pm >0.$ According to $(\CH 3),$ the viscosity coefficient $\mu$ is reduced to $\mu = \mu_+\mathds{1}_{\Omega_+} +\mu_-\mathds{1}_{\Omega_-}$ for $\mu_\pm := \mu(\eta_\pm)>0.$
By the continuity of $\sA_u \Bn$ across $\Gamma$ and $|\sA_u \Bn|, |\sA_{u_+}\Bn_+| \not= 0,$ 
\eqref{eq:INS2} in Lagrangian coordinates reads as follows,
\begin{equation}\label{eq:INSLL_2}
	\left\{\begin{aligned}
		\pa_t\Bu - \eta^{-1}\Di_\xi \BBT(\Bu,\Fq) = \Bf_{u,\Fq},\,\,\,
		\di_\xi \Bu = g_u=\di_\xi \BR_u    
		                         &\quad\mbox{in}\quad \dOm \times ]0,T^\star[,  \\
\jump{\BBT(\Bu,\Fq)\Bn} = \jump{\Bh_{u,\Fq}}, \quad \jump{\Bu}=\0
                                &\quad\mbox{on}\quad \Ga\times ]0,T^\star[,  \\		
\BBT(\Bu_+,\Fq_+)\Bn_+ = \Bk_{u_+,\Fq_+}  
                                &\quad\mbox{on}\quad \Gamma_+ \times  ]0,T^\star[,  \\	
	\Bu|_{t=0}= \Bv_0 &\quad\mbox{in}\quad \dOm,
	\end{aligned}\right.
\end{equation}
where the nonlinear terms 
$(\Bf_{u,\Fq}, g_u, \BR_u,  \Bh_{u,\Fq}, \Bk_{u_+\Fq_+})$ are defined by 
\begin{equation*}
\eta \Bf_{u,\Fq}:=
-  \Di_\xi  \big( \BBT(\Bu,\Fq)- \BBT_u(\Bu,\Fq) \sA_u \big) ,
\end{equation*}
\begin{equation*}
g_u:= \nabla_{\xi}^{\top}\Bu: (\BBI - \sA^{\top}_u ) ,\quad \BR_u:= (\BBI - \sA^{\top}_u ) \Bu,
\end{equation*}
\begin{equation*}
\Bh_{u,\Fq}:=  \BBT(\Bu,\Fq)\Bn - \BBT_u(\Bu,\Fq)\sA_u \Bn
\quad\hbox{and}\quad
\Bk_{u_+,\Fq_+}:= \BBT(\Bu_+,\Fq_+)\Bn_+ - \BBT_{u_+}(\Bu_+,\Fq_+)\sA_{u_+}\Bn_+.
\end{equation*}
Moreover, $T^{\star}$ stands for the lifespan of the local solution $(\Bu, \Fq)$ of \eqref{eq:INSLL_2}
by Theorem \ref{thm:main_local}.
Then given some small initial datum,
\begin{equation}\label{cdt:v0_small}
\|\Bv_0\|_{\wt{\CD}_{q,p}^{2-2\slash p}(\dOm)} \ll 1,
\end{equation}
we are going to show that the local solution of \eqref{eq:INSLL_2} (according to Theorem \ref{thm:main_local}) does not blow up in finite time.
To this end, introduce the following notation throughout the proof, 
\begin{equation*}
\CI_{\ep,\Bu}(a,b):= \|e^{\ep t} (\pa_t \Bu, \Bu, \nabla \Bu, \nabla^2 \Bu)\|_{L_p(a,b; L_q(\dOm))},
\end{equation*}
where $\Bu$ is any vector and the interval $]a,b[ \subset \BBR.$
\medskip

Now, let us reveal our strategy to above issue by applying the linear results in Section \ref{sec:decay}.
Choose any $0<T < T^{\star}$ and define $(\Bw, P):= (\Bu - \BuL, \Fq-\Fp_{_L})$
with $(\BuL, \FpL)$ solving \eqref{eq:S2_1}. 
Therefore $(\Bw, P)$ satisfies \eqref{eq:S2_2} on $]0,T]$ with given 
$(\Bf, g, \BR,  \Bh, \Bk)=(\Bf_{u,\Fq}, g_u, \BR_u,  \Bh_{u,\Fq}, \Bk_{u_+\Fq_+}).$
In next subsection, we will find some suitable extension  
$(\wt \Bf_{u,\Fq}, \wt g_u, \wt \BR_u, \wt  \Bh_{u,\Fq}, \wt \Bk_{u_+\Fq_+}) \in \CZ_{p,q,\ep_0}$
 such that 
\begin{equation*}
(\wt \Bf_{u,\Fq}, \wt g_u, \wt \BR_u, \wt  \Bh_{u,\Fq}, \wt \Bk_{u_+\Fq_+})|_{ t \in ]0,T] } 
= (\Bf_{u,\Fq}, g_u, \BR_u,  \Bh_{u,\Fq}, \Bk_{u_+\Fq_+}).
\end{equation*}
Moreover, the following estimate holds true for any $0<T<T^{\star},$
\begin{equation}\label{es:priori_global}
\|(\wt \Bf_{u,\Fq}, \wt g_u, \wt \BR_u, \wt  \Bh_{u,\Fq}, \wt \Bk_{u_+\Fq_+}) \|_{\CZ_{p,q,\ep_0}}
 \lesssim   \big( \|\Bv_0\|_{\wt{\CD}_{q,p}^{2-2\slash p}(\dOm)}^2+ X(T)^2\big)\big(X(T) +1 \big),
\end{equation}
with $X(T):= \CI_{\ep_0,\Bw} (0,T) 
+ \|e^{\ep_0 t }P\|_{L_p(0,T;W^1_q(\dOm))}.$
\medskip

Next, we consider the following problem, 
\begin{equation}\label{eq:UQ_2}
	\left\{\begin{aligned}
	 \pa_t\BU -	\eta^{-1} \Di_\xi \BBT(\BU, Q) = \wt \Bf_{u,\Fq},\,\,\,
		\di_\xi \BU =\wt g_u =\di_\xi \wt \BR_u
		  &\quad\mbox{in}\quad \dOm \times \BBR_+,  \\
		\jump{\BBT(\BU, Q)\Bn} = \jump{ \wt\Bh_{u,\Fq} }, \,\,\,\jump{\BU}=\0 &\quad\mbox{on}\quad \Ga\times \BBR_+,  \\
		\BBT_+(\BU_+, Q_+)\Bn_+ = \wt \Bk_{u_+\Fq_+} &\quad\mbox{on}\quad \Ga_{+}\times \BBR_+, \\
	\BU|_{t=0}= \0 &\quad\mbox{in}\quad \dOm.
	\end{aligned}\right.
\end{equation}
Then apply Theorem \ref{thm:decay_S22} and \eqref{es:priori_global} by noting the uniqueness of \eqref{eq:INSLL_2} on $]0,T],$ 
\begin{equation}\label{es:XT_1}
X(T) \lesssim  \big( \|\Bv_0\|_{\wt{\CD}_{q,p}^{2-2\slash p}(\dOm)}^2+ X(T)^2\big)\big(X(T) +1 \big) 
+\sum_{\alpha=1}^M  \Big(  \int_0^T e^{p\ep_0t} |(\eta\Bw, \Bp_\alpha)_{\dOm}|^p dt   \Big)^{1\slash p}.
\end{equation}
To handle the last term on the r.h.s. of \eqref{es:XT_1}, we take advantage of \eqref{eq:embedding} and Theorem \ref{thm:decay_S1}, 
\begin{equation}\label{es:uq_L_exp}
\|e^{\ep_0 t}\Bu_{_L} \|_{L_\infty(\BBR_+;\CD_{q,p}^{2-2\slash p}(\dOm))}
\lesssim \CI_{\ep_0, \Bu_{_L}}(0,\infty)
+\|e^{\ep_0 t} \Fq_{_L} \|_{L_p(\BBR_+;W^1_q(\dOm))}
\lesssim \|\Bv_0\|_{\wt{\CD}_{q,p}^{2-2\slash p}(\dOm)}.
\end{equation}
Immediately, \eqref{es:uq_L_exp} yields that 
\begin{equation}\label{es:u_exp}
\|e^{\ep_0 t}\Bu \|_{L_\infty(0,T;\CD_{q,p}^{2-2\slash p}(\dOm))}
\lesssim \CI_{\ep_0, \Bu} (0,T)
\lesssim \|\Bv_0\|_{\wt{\CD}_{q,p}^{2-2\slash p}(\dOm)}
+ \CI_{\ep_0,\Bw} (0,T).
\end{equation}
On the other hand, thanks to $\rho_0 =\eta,$ $(\eta \,\Bv_0, \Bp_\alpha )_{\dOm} =0$ and Proposition \ref{prop:conserve}, we have 
$$ \int_{\dOm} \eta \Bu(t) \cdot \Bp_\alpha \big(\BX_u (\xi,t) \big) \, d\xi  = ( \eta\Bv_0 , \Bp_\alpha )_{\dOm} =0
\,\,\,\hbox{for any}\,\,\, 0<t < T.
$$ 
Recall that $q>N$ and $\Fp_\alpha \in \CR_d$ and we have 
\begin{equation*}
\| \Bp_\alpha(\xi) - \Bp_\alpha \big(\BX_u(\xi,t)\big) \|_{L_\infty(\dOm)} 
\leq \|\nabla \Bp_\alpha\|_{L_\infty(\BBR^N)} \int_0^t \|\Bu(s)\|_{W^1_q(\dOm)} \,ds 
\leq C_{\ep_0,p} \|e^{\ep_0s}\Bu(s)\|_{L_p(0,t;W^1_q(\dOm))}.
\end{equation*}
Then H\"older inequality implies that 
\begin{equation*}
\big|\big(\eta u(t), \Bp_\alpha\big)_{\dOm} \big| 
= \big|\big(\eta u(t), \Bp_\alpha - \Bp_\alpha \circ X_u (t)\big)_{\dOm} \big|
 \leq C_{\ep_0,p,q,\eta,\dOm} \|\Bu(t)\|_{L_q(\dOm)}  \CI_{\ep_0, \Bu} (0,t),
\end{equation*}
from which we obtain
\begin{equation}\label{es:etau_exp}
 \sum_{\alpha=1}^M  \Big(  \int_0^T e^{p\ep_0 t} \big|\big(\eta\Bu(t), \Bp_\alpha\big)_{\dOm}\big|^p dt \Big)^{1\slash p}
\lesssim  \CI_{\ep_0,\Bu}(0,T)^2 
\lesssim\|\Bv_0\|_{\wt{\CD}_{q,p}^{2-2\slash p}(\dOm)}^2+ X(T)^2.
\end{equation}
Therefore, \eqref{cdt:v0_small}, \eqref{es:uq_L_exp} and \eqref{es:etau_exp} furnish that
\begin{equation}\label{es:XT_2}
\sum_{\alpha=1}^M  \Big(  \int_0^T e^{p\ep_0t} |(\eta\Bw, \Bp_\alpha)_{\dOm}|^p dt   \Big)^{1\slash p}
\lesssim \|\Bv_0\|_{\wt{\CD}_{q,p}^{2-2\slash p}(\dOm)}+ X(T)^2.
\end{equation}
Combining the condition \eqref{cdt:v0_small},  \eqref{es:XT_1} and \eqref{es:XT_2}, we have
\begin{equation}\label{es:XT}
X(T)\lesssim \|\Bv_0\|_{\wt{\CD}_{q,p}^{2-2\slash p}(\dOm)}+ X(T)^2 + X(T)^3. 
\end{equation}
By \eqref{es:XT},  \eqref{cdt:v0_small} and Lemma \ref{lem:bootstrap},  we have $X(T) \leq 2C\|\Bv_0\|_{\wt{\CD}_{q,p}^{2-2\slash p}(\dOm)},$ which allows us to extend $(\Bw, P)$ beyond $T^{\star}$ by the standard bootstrap arguments as in \cite{Shi2015}.

\subsection{A priori estimates}
To complete our proof of Theorem \ref{thm:main_global}, it remains to verify the claim \eqref{es:priori_global}. As before, we first introduce some extensions operators and then check the bound of \eqref{es:priori_global} for the extended non-homogeneous terms.
\subsubsection*{Extension operators}
Assume that $\varphi(s) \in C^{\infty}(\BBR)$ is some cut-off function  such that
$\varphi(s) =1$ for $s\leq 0$ and $\varphi(s) =0$ for $s\geq 1.$ Then denote $\varphi_t(s):= \varphi(s-t)$ for any $t \in ]0,T^{\star}[$ and recall the definition of $E_{(t)}$ in Section \ref{sec:local}.
For smooth $E$-valued  $\Fh,$ it is not hard to check
\begin{align*}
\|e^{\ga s}\varphi_t(s)E_{(t)}\Fh(\cdot, s) \|_{L_p (\BBR; E)}^p 
&\leq \|e^{\ga s}E_{(t)}\Fh(\cdot, s) \|_{L_p (0,t; E)}^p 
+\int_t^{\min\{2t, t+1\} } e^{p\ga s} |\varphi_t(s)|^p \|\Fh(\cdot,2t-s)\|_{E}^p \,ds\\
&\leq  \|e^{\ga s}E_{(t)}\Fh(\cdot, s) \|_{L_p (0,t; E)}^p  
+ e^{2p\ga}\int_{\max\{0,t-1\}}^t e^{p\ga \wt s}\|\Fh(\cdot, \wt s)\|_{E}^p \,d\wt s,
\end{align*}
where the change of variables  $\wt s:= 2t-s$ is applied in the last inequality. 
Therefore,
\begin{equation}\label{es:Fh_2} 
\|e^{\ga s}\varphi_t(s) E_{(t)}\Fh(\cdot, s) \|_{L_p (\BBR; E)} 
\leq (1+e^{2p\ga})^{1\slash p} \|e^{\ga s}\Fh(\cdot, s) \|_{L_p (0,t; E)} 
\quad \hbox{for any}\,\,\, \ga, t > 0.
\end{equation}
Furthermore, the formula of $\pa_s E_{(t)}\Fh(\cdot, s)$ and \eqref{es:Fh_2} yield that
\begin{align}\label{es:DFh_2}
\|e^{\ga s} \pa_s \big(\varphi_t(s) E_{(t)} \Fh(\cdot, s)\big) \|_{L_p (\BBR; E)}
 &\leq e^{2\gamma} \sup|\varphi'| 
\|e^{\ga s}  \Fh(\cdot, s) \|_{L_p (0,T; E)}\\ \nonumber
& \quad + (1+e^{2p\ga})^{1\slash p} \|e^{\ga s}\pa_s \Fh(\cdot, s) \|_{L_p (0,t; E)}
\end{align}
for $\Fh(\cdot,0) =0$ and $\ga, t > 0.$
Then according to \eqref{es:Fh_2} and \eqref{es:DFh_2},  we immediately know that
\begin{equation}\label{es:extend_w_exp}
\|e^{\ep_0 t}\varphi_{_T}(t)\ET \Bw(\cdot, t) \|_{W^{2,1}_{q,p} (\dOm \times \BBR)} 
\lesssim  \CI_{\ep_0,\Bw} (0,T).
\end{equation}
\medskip

Now, let us consider the extension of the solution $\BuL$ of \eqref{eq:S2_1}. Define that
\begin{equation*}
\oBuL(\cdot,t) :=  e^{-|t| \CA_q} \Bv_0(\cdot) 
\quad\hbox{for}\,\,\, t\in \BBR.
\end{equation*}
By \eqref{es:uq_L_exp} and the mixed derivative theorem, we have
\begin{equation}\label{es:oBuL}
\|e^{\ep_0 |t|}\oBuL\|_{H^{1\slash 2}_p(\BBR; W^1_q(\dOm))}
 \lesssim \|e^{\ep_0 |t|}\oBuL\|_{W^{2,1}_{q,p}(\dOm \times \BBR)}
\lesssim \|\Bv_0\|_{\wt{\CD}_{q,p}^{2-2\slash p}(\dOm)}.
\end{equation}
With the definitions of $\varphi_{_T}$ and $\oBuL,$ let us take  
$\wt\Bu(\cdot ,t ) := \oBuL (\cdot, t) + \varphi_{_T}(t)\ET \Bw(\cdot, t).$
Then \eqref{es:extend_w_exp} \eqref{es:oBuL} and mixed derivative theorem furnish that,
\begin{equation}\label{es:tBu_exp}
\|e^{\ep_0 t}\wt\Bu\|_{H^{1\slash 2}_p(\BBR; W^1_q(\dOm))}
 \lesssim \|e^{\ep_0 t}\wt\Bu\|_{W^{2,1}_{q,p}(\dOm \times \BBR)}
\lesssim \|\Bv_0\|_{\wt{\CD}_{q,p}^{2-2\slash p}(\dOm)} +  \CI_{\ep_0,\Bw} (0,T).
\end{equation}
Moreover, the pressure $\wt \Fq := \Fp_{_L} + \varphi_{_T}(t)\ET P (\cdot, t)$  for any $t \in [0,T].$
Then by \eqref{es:Fh_2} and  \eqref{es:uq_L_exp}, we have
\begin{equation}\label{es:tq_exp}
\|e^{\ep_0 t}\wt \Fq\|_{L_p(\BBR_+;W^1_q(\dOm))} 
\lesssim  \|\Bv_0\|_{\wt{\CD}_{q,p}^{2-2\slash p}(\dOm)}
+ \|e^{\ep_0 t }P\|_{L_p(0,T;W^1_q(\dOm))}.
\end{equation}
\medskip

Before going into the details of the calculations, it is necessary to address some bounds for $\sA_u.$ 
Firstly, we infer from Lemma \ref{lem:A_DA} and \eqref{es:u_exp} that,  
\begin{align} \label{ob:key_exp_1}
\|\BBI -\sA_u\|_{L_\infty(0,T; W^1_q(\dOm))} &\lesssim \int_0^T \|\nabla_\xi \Bu\|_{W^{1}_q(\dOm)} \,dt
\lesssim \|\Bv_0\|_{\wt{\CD}_{q,p}^{2-2\slash p}(\dOm)}
+ \CI_{\ep_0,\Bw} (0,T),
\\ \nonumber
\end{align}
On the other hand, Lemma \ref{lem:A_DA}, Proposition \ref{prop_uDu} and \eqref{es:u_exp} yield
\begin{align}\label{ob:key_exp_2}
\|\pa_t \sA_u\|_{L_{p\slash \theta} (0,T;L_\beta(\dOm))} 
 \lesssim \|\nabla \Bu\|_{L_{p\slash \theta} (0,T;L_\beta(\dOm))} 
 &\lesssim \|\Bu\|_{L_{\infty} \big(0,T; \CD^{2-2\slash p}_{q,p}(\dOm) \big)}^{1-\theta} \|\nabla^2 \Bu\|_{L_p (0,T;L_{q}(\dOm))}^{\theta}
 \\ \nonumber
&\lesssim \|\Bv_0\|_{\wt{\CD}_{q,p}^{2-2\slash p}(\dOm)}
+ \CI_{\ep_0,\Bw} (0,T). 
\end{align}
for some $ (\theta, \beta) \in ]0,1[ \times ]q,\infty]$ satisfying $1 - 2(1-\theta)\slash p = N \slash q - N \slash \beta.$
In the following, we give the definitions of suitable $(\wt \Bf_{u,\Fq}, \wt g_u, \wt \BR_u, \wt  \Bh_{u,\Fq}, \wt \Bk_{u_+\Fq_+})$ fulfilling \eqref{es:priori_global}.

\subsubsection*{\underline{Bounds for $\wt\Bf_{u,\Fq}$}}
Recall $ \BBH_{u} :=  \nabla_\xi^{\top}\Bu (\BBI-\sA_{u}^{\top}) +(\BBI-\sA_{u})\nabla_\xi \Bu^{\top}$ and the formula
\begin{equation*}
\eta \Bf_{u,\Fq} =
                 -\Di_\xi \Big(\mu \big(\BBH_{u}
                 +\BBD(\Bu)(\BBI-\sA_{u}) \big) \Big)
      +\Di_\xi \big( \mu \BBH_{u}(\BBI-\sA_{u})\big) 
                 +\Di_\xi \big(\Fq\,(\BBI-\sA_{u})\big).
\end{equation*}
Then inspired by previous discussions on the short time issue, we set
\begin{align*}
\eta \wt \Bf_{u,\Fq} &:=
                 -\Di_\xi \Big(\mu \big(\wt\BBH_{u} +\tBBD_u \big) \Big)
      +\Di_\xi \big( \mu  \wt\BBH_{u} \cdot \varphi_{_T}(t)E_{_{(T)}}(\BBI-\sA_{u}) \big) \\
                &\quad  +\Di_\xi \big( \wt\Fq \cdot   \varphi_{_T}(t)E_{_{(T)}}(\BBI-\sA_{u})\big),\\
\tBBH_{u} &:=  \nabla_\xi^{\top} \wt\Bu \cdot \varphi_{_T}(t) E_{_{(T)}}(\BBI-\sA_{u}^{\top})
             + \varphi_{_T}(t)E_{_{(T)}}(\BBI-\sA_{u}) \cdot \nabla_\xi \wt\Bu^{\top},\\
 \tBBD_{u} &:=\BBD(\wt\Bu) \cdot  \varphi_{_T}(t)E_{_{(T)}} (\BBI -\sA_{u}).         
\end{align*}
Thanks to \eqref{es:tBu_exp}, \eqref{es:tq_exp} and \eqref{ob:key_exp_1}, the following $W^1_q$ estimates are attained,
\begin{equation}\label{es:tHtD1_exp}
\|e^{\ep_0 t}(\tBBH_{u}, \tBBD_{u})\|_{L_{p}(\BBR;W^1_q(\dOm))} 
\lesssim \|\Bv_0\|_{\wt{\CD}_{q,p}^{2-2\slash p}(\dOm)}^2
+ X(T)^2,
\end{equation}
\begin{equation}\label{es:tHtD2_exp}
\| e^{\ep_0 t} \wt\BBH_{u} \cdot \varphi_{_T}(t)E_{_{(T)}}(\BBI-\sA_{u})\|_{L_{p}(\BBR;W^1_q(\dOm))} 
\lesssim \Big( \|\Bv_0\|_{\wt{\CD}_{q,p}^{2-2\slash p}(\dOm)}
+X(T)\Big)^3,
\end{equation}
\begin{equation}\label{es:tq_f_exp}
\|e^{\ep_0 t} \wt\Fq \cdot   \varphi_{_T}(t)E_{_{(T)}}(\BBI-\sA_{u})\|_{L_{p}(\BBR;W^1_q(\dOm))} 
\lesssim \|\Bv_0\|_{\wt{\CD}_{q,p}^{2-2\slash p}(\dOm)}^2
+ X(T)^2.
\end{equation}
Thus we can conclude the desired bound of $\wt \Bf_{u,\Fq}$ in \eqref{es:priori_global} from \eqref{es:tHtD1_exp}, \eqref{es:tHtD1_exp} and \eqref{es:tq_f_exp}, 
\begin{equation*}
\|e^{\ep_0 t} \wt \Bf_{u,\Fq} \|_{L_p(\BBR_+;L_q(\dOm))}
 \lesssim 
 \big( \|\Bv_0\|_{\wt{\CD}_{q,p}^{2-2\slash p}(\dOm)}^2+ X(T)^2\big)
 \big(X(T) +1 \big).
\end{equation*}

\subsubsection*{\underline{Bounds for $\wt g_u$}}
Note the formula $g_u =  \nabla_{\xi}^{\top} \Bu:  (\BBI - \sA^{\top}_u ) 
 = \di_\xi \BR_u =\di_\xi  \big( (\BBI - \sA^{\top}_{u} ) \Bu \big),$
and it is reasonable to set $\wt g_u$ and $\tBR_u$ by
\begin{equation*}
\tg_u:= \nabla_{\xi}^{\top} \wt \Bu:  \varphi_{_T}(t)E_{_{(T)}}(\BBI - \sA^{\top}_u ) 
\,\,\, \hbox{and}\,\,\,
\tBR_u :=  \varphi_{_T}(t)E_{_{(T)}}(\BBI - \sA^{\top}_u )  \wt \Bu.
\end{equation*}
Then according to \eqref{es:tHtD1_exp}, the $W^1_q(\dOm)$ norm of $\tg_u$ is easily given by,
\begin{equation}\label{es:tg1_exp}
\|e^{\ep_0 t}\tg_u \|_{L_{p}(\BBR;W^1_q(\dOm))} 
\lesssim \|\Bv_0\|_{\wt{\CD}_{q,p}^{2-2\slash p}(\dOm)}^2
+ X(T)^2.
\end{equation}
Moreover, the fact that $\jump{\tBR_u \cdot \Bn} =0$ on $\Gamma$ yields that $\tg_u \in \BW^{-1}_q(\Omega)$ from the proof in  Section \ref{sec:local}.
\medskip

To find the bound in $H^{1\slash 2}_p \big(\BBR; L_q(\dOm)\big),$  firstly assume that $(p,q) \in (I),$ 
and employ Lemma \ref{lem:A_DA} and \eqref{es:u_exp},
\begin{equation*}
\|\pa_t \sA_u \|_{L_\infty(0,T;L_q(\dOm))} 
\lesssim \|\nabla_\xi \Bu\|_{L_\infty(0,T;L_q(\dOm))} 
 \lesssim \|\Bu\|_{L_\infty(0,T;D_{q,p}^{2-2\slash p}(\dOm))} 
 \lesssim \|\Bv_0\|_{\wt{\CD}_{q,p}^{2-2\slash p}(\dOm)} +X(T).
\end{equation*}
Thus by Lemma \ref{lem:Hhalf}, \eqref{es:tBu_exp} and \eqref{ob:key_exp_1} , we obtain that
\begin{align*}
\|e^{\ep_0 t}\tg_u \|_{H^{1\slash 2}_p (\BBR; L_q(\dOm))} 
\lesssim \|\BBI - \sA_{u}^{\top}\|_{W^{1,1}_{q,\infty}(\dOm \times ]0,T[)}
\|e^{\ep_0 t} \nabla_{\xi}^{\top} \wt \Bu\|_{H^{1\slash 2, 1\slash 2}_{q,p}(\dOm \times \BBR)}
\lesssim  \|\Bv_0\|_{\wt{\CD}_{q,p}^{2-2\slash p}(\dOm)}^2+ X(T)^2.
\end{align*}
On the othere hand, if $(p,q) \in (II),$ then Lemma \ref{lem:Hhalf3} and \eqref{ob:key_exp_2} yield
\begin{align*}
\|e^{\ep_0 t}\tg_u \|_{H^{1\slash 2}_p (\BBR; L_q(\dOm))} 
&\lesssim \|e^{\ep_0 t} \nabla_{\xi}^{\top} \wt \Bu\|_{H^{1\slash 2, 1\slash 2}_{q,p}(\dOm \times \BBR)} \|\BBI -\sA_u\|_{L_\infty(0,T; W^1_q(\dOm))}^{1\slash 2} \\
& \times \Big( \|\BBI -\sA_u\|_{L_\infty(0,T; W^1_q(\dOm))} 
 +\big\| \pa_t \big(\varphi_{_T}(t)E_{_{(T)}}(\BBI - \sA^{\top}_u ) \big)  \big\|_{L_{p\slash \theta} (\BBR;L_\beta(\dOm))}   \Big)^{1\slash 2}
\\
&\lesssim  \|\Bv_0\|_{\wt{\CD}_{q,p}^{2-2\slash p}(\dOm)}^2+ X(T)^2.
\end{align*}
Therefore, we can conclude all the necessary bounds of $\tg_u$ from above discussions for $(p,q) \in (I) \cup (II),$
\begin{equation*}
\|e^{\ep_0 t} \wt g_u\|_{H^{1,1\slash 2}_{q,p} (\dOm \times \BBR)}
\lesssim  \|\Bv_0\|_{\wt{\CD}_{q,p}^{2-2\slash p}(\dOm)}^2+ X(T)^2.
\end{equation*}

\subsubsection*{\underline{Bounds for $\wt \BR_u$}}
As $\tBR_u :=  \varphi_{_T}(t)E_{_{(T)}}(\BBI - \sA^{\top}_u )  \wt \Bu,$ 
we directly apply \eqref{es:tBu_exp}, \eqref{ob:key_exp_1} and Lemma \ref{lem:A_DA},
\begin{align*}
\|e^{\ep_0 t} (\tBR_u, \pa_t \tBR_u)\|_{L_p(\BBR_+;L_q(\dOm))} 
&\lesssim \|\BBI - \sA^{\top}_u\|_{L_\infty(\dOm \times ]0,T[)}
\|e^{\ep_0 t} (\wt\Bu, \pa_t \wt\Bu)\|_{L_p(\BBR_+;L_q(\dOm))} \\
& \quad +\big\|\pa_t  \big( \varphi_{_T}(t)E_{_{(T)}}(\BBI - \sA^{\top}_u )\big) \big\|_{L_p(0,2T;L_\infty(\dOm))}
\|e^{\ep_0 t} \wt\Bu\|_{L_\infty(0,2T;L_q(\dOm))}\\
& \lesssim   \|\Bv_0\|_{\wt{\CD}_{q,p}^{2-2\slash p}(\dOm)}^2+ X(T)^2,
\end{align*}
which implies the bound of $\tBR_u$ in \eqref{es:priori_global}.

\subsubsection*{\underline{Bounds for $\wt\Bh_{u,\Fq}$ and $\wt\Bk_{u,\Fq}$}}
According to our previous experience, the tricks for extensions $\wt\Bh_{u,\Fq}$ and $\wt\Bk_{u,\Fq}$ are similar. For simplicity, we only consider $\wt\Bh_{u,\Fq}$ here. 
As $\Fq$ is a solution of \eqref{eq:INSLL_2}, it is easy to see that
\begin{equation*}
\Fq = |\sA_u \Bn|^{-2} \big( \mu \BBD_u(\Bu) \sA_u \Bn
 \cdot \sA_u \Bn \big) 
 \,\,\, \hbox{on}\,\,\, \Gamma \times ]0,T[.
\end{equation*}
Inspired by the discussions on $\wt \Bf_{u,\Fq},$ we introduce that
\begin{align*}
\wt\Bh_{u,\Fq} &:=
               \mu \big(\wt\BBH_{u} +\tBBD_u \big) \Bn 
      -\big( \mu  \wt\BBH_{u} \cdot \varphi_{_T}(t)E_{_{(T)}}(\BBI-\sA_{u}) \big) \Bn\\
                &\quad  - \big( \wt\Fp \cdot   \varphi_{_T}(t)E_{_{(T)}}(\BBI-\sA_{u})\big) \Bn.\\
  \wt \Fp&:= |\wt\sA_u \Bn|^{-2}  \mu  \big( \wt \BBH_{u}  + \BBD(\wt\Bu)\big)  \wt\sA_u \Bn
 \cdot \wt\sA_u \Bn \big) \\
 \wt \sA_u &:=   \varphi_{_T}(t)\big(E_{_{(T)}}(\sA_{u}-\BBI)+ \BBI\big) 
\end{align*} 
As $W^1_q(\dOm)$ is an algebra for $q>N,$ we infer from \eqref{es:tHtD1_exp}, \eqref{es:tHtD1_exp} and \eqref{es:tq_f_exp} that
\begin{equation}\label{es:th_exp_1}
\|e^{\ep_0 t}\wt\Bh_{u,\Fq}\|_{L_p(\BBR; W^1_q(\dOm))} \lesssim  
\big( \|\Bv_0\|_{\wt{\CD}_{q,p}^{2-2\slash p}(\dOm)}^2+ X(T)^2\big)
 \big(X(T) +1 \big).
\end{equation}
By the above computations of $\tg_u,$ $\wt\BBH_{u}$ and $\wt\BBD_{u}$ are bounded in $H^{1,1\slash 2}_{q,p} (\dOm \times \BBR)$ by
\begin{equation}\label{es:HD_exp}
\|e^{\ep_0 t} (\wt\BBH_{u}, \wt\BBD_{u}) \|_{H^{1,1\slash 2}_{q,p} (\dOm \times \BBR)}
\lesssim  \|\Bv_0\|_{\wt{\CD}_{q,p}^{2-2\slash p}(\dOm)}^2+ X(T)^2.
\end{equation}
Combine the estimates \eqref{ob:key_exp_1} and \eqref{es:HD_exp},
\begin{multline}\label{es:th_exp_21}
\big\| e^{\ep_0 t}\mu \big(\wt\BBH_{u} +\tBBD_u \big) \Bn \big\|_{H^{1\slash 2}_{p} (\BBR;L_q(\dOm))}
      +\big\| e^{\ep_0 t}\big( \mu  \wt\BBH_{u} \cdot \varphi_{_T}(t)E_{_{(T)}}(\BBI-\sA_{u}) \big) \Bn \big\|_{H^{1\slash 2}_{p} (\BBR;L_q(\dOm))}\\
       \lesssim  
\big( \|\Bv_0\|_{\wt{\CD}_{q,p}^{2-2\slash p}(\dOm)}^2+ X(T)^2\big)
 \big(X(T) +1 \big).
\end{multline} 
\begin{equation}\label{es:tp_exp}
\|e^{\ep_0 t }\wt\Fp\|_{H^{1\slash 2}_{p} (\BBR;L_q(\dOm))}
\lesssim  \big( \|\Bv_0\|_{\wt{\CD}_{q,p}^{2-2\slash p}(\dOm)}+ X(T)\big)  \big( 1+ \|\Bv_0\|_{\wt{\CD}_{q,p}^{2-2\slash p}(\dOm)}+ X(T)\big).
\end{equation}
Then keeping \eqref{es:tp_exp} in mind, we can employ the similar procedure as the bound of $\|e^{\ep_0 t}\wt g_u\|_{H^{1\slash 2}_p (\BBR; L_q(\dOm))}$ and then obtain that 
\begin{equation}\label{es:th_exp_22}
\big\|  e^{\ep_0 t}\big( \wt\Fp \cdot   \varphi_{_T}(t)E_{_{(T)}}(\BBI-\sA_{u})\big) \Bn \big\|_{H^{1\slash 2}_{p} (\BBR;L_q(\dOm))}
 \lesssim  \big( \|\Bv_0\|_{\wt{\CD}_{q,p}^{2-2\slash p}(\dOm)}^2+ X(T)^2\big)\big(X(T) +1 \big).
\end{equation}
Thus \eqref{es:th_exp_21} and \eqref{es:th_exp_22} immediately yield
\begin{equation}\label{es:th_exp_2}
\|e^{\ep_0 t }\wt\Bh_{u,\Fq}\|_{H^{1\slash 2}_{p} (\BBR;L_q(\dOm))}
 \lesssim  \big( \|\Bv_0\|_{\wt{\CD}_{q,p}^{2-2\slash p}(\dOm)}^2+ X(T)^2\big)\big(X(T) +1 \big).
\end{equation}
Finally, \eqref{es:th_exp_1} and  \eqref{es:th_exp_2} furnish that
\begin{equation*}
\|e^{\ep_0 t }\wt\Bh_{u,\Fq}\|_{H^{1,1\slash 2}_{q,p} (\dOm \times \BBR)}
 \lesssim  \big( \|\Bv_0\|_{\wt{\CD}_{q,p}^{2-2\slash p}(\dOm)}^2+ X(T)^2\big)\big(X(T) +1 \big),
\end{equation*}
which is admissible in \eqref{es:priori_global}.

%%%%%%%%%%%%%%%%%%%%%%%%%%%%%%%%%
%\newpage
%%%%%%%%%%%%%%%%%%%%%%%%%%%%%%%%%

\begin{appendix}
\section{Technical lemma}

To display the Lagrangian coordinates approach, let us recall some technical results here. Assume that $\Bu \in W^{2,1}_{q,p}(G\times ]0,T[)$ for some open (not necessary bounded) domain $G \subset \BBR^N$ and $T \in ]0,\infty].$ Denote that
\begin{equation*}
\BX(\xi,t) := \xi +\int^t_0 \Bu (\xi, \tau) \,d \tau \quad \hbox{for} \,\,\, \xi \in G. 
\end{equation*}
If  $\|\int^t_0 \nabla_\xi \Bu (\xi, t') \, dt' \|_{L_\infty(G)}$ is strictly smaller than $1$,  then the following definition makes sense,
\begin{equation*}
\sA := (\nabla_{\xi}^{\top} \BX^{-1})^{\top}= (\nabla_\xi \BX)^{-1} = \sum_{k=0}^{\infty} \big(-\int^t_0 \nabla_\xi \Bu (\xi, t') \, dt'\big)^k.
\end{equation*}
In fact, we have the following Lemma concerning the estimates of $\sA$ and $\sA^{\top}.$
\begin{lemm}\label{lem:A_DA}
Assume that $\Bu$ is some smooth enough vector field satisfying 
$$\|\nabla_\xi \Bu\|_{L_1(0,T;L_\infty(G))} \leq  \kappa <1,$$
and $\FA$ stands for $\sA$ or $ \sA^{\top}$ as we defined above. Then following assertions hold true. 
\begin{enumerate}
\item For some terms of $\FA,$ there exists a constant $C_{N,\kappa}$ such that
\begin{equation*}
\|\FA\|_{L_\infty (G \times ]0,T[)} \leq C_{N,\kappa} ,  
\end{equation*}
\begin{equation*}
\|(\FA- \BBI,\FA\FA- \BBI, \FA\FA^{\top}- \BBI)\|_{L_\infty (G \times ]0,T[)} \leq C_{N,\kappa} \|\nabla_\xi \Bu\|_{L_1(0,T;L_\infty(G))}, 
\end{equation*}
\item For the first order derivative terms of $\FA$ and $q \in ]1,\infty[,$ there exists a constant $C_{N,k}$ such that 
\begin{equation*}
\|\nabla_\xi (\FA, \FA\FA, \FA\FA^{\top})\|_{L_\infty (0,T; L_q(G))} \leq C_{N,\kappa} \|\nabla^2_\xi \Bu\|_{L_1(0,T;L_q(G))}.
\end{equation*}
\item For the time derivative of $\FA,$ we have for any suitable $(\widetilde{p},\widetilde{q}) \in [1,\infty]^2$
\begin{equation*}
\|\pa_t \FA\|_{L_{\widetilde{p}}(0,T; L_{\widetilde q}(G))}
 \leq C_{N,\kappa} \| \nabla_\xi \Bu\|_{L_{\widetilde{p}}(0,T;L_{\widetilde q}(G))},
\end{equation*}
\begin{equation*}
\|\pa_t \nabla_\xi \FA\|_{L_{\tp}(0,T; L_{\tq}(G))} 
\leq C_{N,\kappa}\Big(
\|\nabla_\xi^2 \Bu \|_{L_{\widetilde{p}}(0,T; L_{\widetilde q}(G))} 
+ \|\nabla_\xi \Bu\|_{L_{\tp}(0,T; L_{\infty}(G))} 
\|\nabla_\xi^2 \Bu \|_{L_{1}(0,T; L_{\widetilde q}(G))}\Big).
\end{equation*}
\end{enumerate}
Above all constants  $C_{N,\kappa}$ go to $\infty$ as $\kappa$ tends to $1.$ 
\end{lemm}
\begin{proof}
For simplicity, we only concentrate to the proof concerning the terms $\sA,$ $\sA-\BBI$ and $\sA^{\top}\sA-\BBI$. Recall the definition of $\sA,$ we have 
\begin{equation*}
\|\sA\|_{L_\infty (G \times ]0,T[)} 
\leq \frac{C_N}{1-\| \int_0^t \nabla_\xi \Bu \, dt'\|_{L_\infty (G\times ]0,T[)}}  
\leq \frac{C_N}{1-\kappa} \,\cdot
\end{equation*} 
Note that $\| \int_0^t \nabla_\xi \Bu \, dt'\|_{L_\infty (G\times ]0,T[)} \leq \|\nabla_\xi \Bu\|_{L_1(0,T;L_\infty(G))} <1 ,$ then  we have
\begin{equation*}
\|\sA -\BBI\|_{L_\infty(G\times ]0,T[)}  
\leq \frac{\| \int_0^t \nabla_\xi \Bu \, dt'\|_{L_\infty (G\times ]0,T[)}}{1-\| \int_0^t \nabla_\xi \Bu \, dt'\|_{L_\infty (G\times ]0,T[)}}  
\leq \frac{C_N}{1- \kappa} \|\nabla_\xi \Bu\|_{L_1(0,T;L_\infty(G))}.
\end{equation*} 
Rewrite the term $\sA^{\top}\sA-\BBI$ by
\begin{equation}\label{eq:AA-I}
\sA^{\top}\sA-\BBI = (\sA^{\top}-\BBI)(\sA-\BBI) + (\sA^{\top}-\BBI) + (\sA-\BBI),
\end{equation}
and combine the results of $\FA -\BBI$ yield the bounds for $\sA^{\top}\sA-\BBI $ with taking $C_{N,\kappa} =  C_N \frac{1+\kappa}{1-\kappa} \max \{\frac{1}{(1-\kappa)}, 1\}.$
\smallbreak

Let us consider the first order derivative term $\nabla_\xi \sA,$ it is not hard to see from definition of $\sA$
\begin{equation*}
\|\nabla_\xi \sA \|_{L_\infty(0,T; L_q(G))} \leq \frac{C_N}{(1-\kappa)^2}\|\nabla^2_\xi \Bu \|_{L_1(0,T; L_q(G))}.
\end{equation*}
Combine above estimate and \eqref{eq:AA-I}, we have
\begin{equation*}
\|\nabla_\xi (\sA^{\top}\sA -\BBI)\|_{L_\infty(0,T; L_q(G))}  \leq C_N \frac{1+\kappa}{(1-\kappa)^2} 
\|\nabla^2_\xi \Bu \|_{L_1(0,T; L_q(G))}.
\end{equation*}
Lastly, for the time derivative, we just take advantage the definition of $\sA,$
\begin{equation*}
\|\pa_t \sA\|_{L_{\widetilde{p}}(0,T; L_{\widetilde q}(G))} 
\leq \frac{ C_N }{(1-\kappa)^2} 
\|\nabla_\xi \Bu \|_{L_{\widetilde{p}}(0,T; L_{\widetilde q}(G))},\quad \hbox{for }\,\,\,\widetilde{p},\widetilde{q} \in [1,\infty].
\end{equation*}
\begin{equation*}
\|\pa_t\nabla_\xi \sA\|_{L_{\tp}(0,T; L_{\tq}(G))} 
\leq \frac{ C_N }{(1-\kappa)^2} 
\|\nabla_\xi^2 \Bu \|_{L_{\widetilde{p}}(0,T; L_{\widetilde q}(G))} 
+ \frac{ C_N }{(1-\kappa)^3}\|\nabla_\xi \Bu\|_{L_{\tp}(0,T; L_{\infty}(G))} 
\|\nabla_\xi^2 \Bu \|_{L_{1}(0,T; L_{\widetilde q}(G))}.
\end{equation*}
\end{proof}
\medskip

To study the stability problem, we need the following technical results similar to Lemma \ref{lem:A_DA}.
\begin{lemm}\label{lem:deltaA}
Assume that $\Bu_k$ with $k=1,2$ are some smooth enough vector field satisfying 
\begin{equation*}
\|(\nabla_\xi \Bu_1,\nabla_\xi \Bu_2) \|_{L_1(0,T;L_\infty(G))} \leq  \kappa <1.
\end{equation*}
Define the corresponding mapping  
$\BX_k(\xi,t) := \xi +\int^t_0 \Bu_k (\xi, \tau) \,d \tau$ 
associated to $\Bu_k,$ and $\sA_k := (\nabla_\xi \BX_k)^{-1}.$ 
For simplicity, $\FA_k$ stands for $\sA_k$ or $ \sA_k^{\top},$ and the notations on difference are given by
$(\delta\Bu, \delta \FA):= (\Bu_2-\Bu_1, \FA_2-\FA_1).$ 
Then following assertions hold true. 
\begin{enumerate}
\item For some terms of $\FA,$ there exists a constant $C_{N,\kappa}$ such that
\begin{equation*}
\|(\delta\FA, \FA_2\FA_2-\FA_1\FA_1,\FA_2\FA_2^{\top}-\FA_1\FA_1^{\top})\|_{L_\infty (G \times ]0,T[)} \leq C_{N,\kappa} \|\nabla_\xi \delta\Bu\|_{L_1(0,T;L_\infty(G))}, 
\end{equation*}
\item For the first order derivative terms of $\FA$ and $q \in ]1,\infty[,$ there exists a constant $C_{N,k}$ such that 
\begin{multline*}
\|\nabla_\xi (\delta \FA, \FA_2\FA_2-\FA_1\FA_1,\FA_2\FA_2^{\top}-\FA_1\FA_1^{\top})\|_{L_\infty (0,T; L_q(G))}
\leq C_{N,\kappa}\Big(\|\nabla_\xi^2 \delta\Bu\|_{L_1(0,T;L_q(G))}\\
+ \|\nabla_\xi^2 (\Bu_1,\Bu_2)\|_{L_1(0,T;L_q(G))}
\|\nabla_\xi \delta\Bu\|_{L_1(0,T;L_\infty(G))}\Big),
\end{multline*}
\item For the time derivative of $\FA,$ we have for any suitable $(\widetilde{p},\widetilde{q}) \in [1,\infty]^2$
\begin{equation*}
\|\pa_t \delta \FA\|_{L_{\widetilde{p}}(0,T; L_{\widetilde q}(G))}
 \leq C_{N,\kappa} \Big(\|\nabla_\xi \delta \Bu\|_{L_{\widetilde{p}}(0,T;L_{\widetilde q}(G))} 
  +\|\nabla_\xi (\Bu_1,\Bu_2)\|_{L_{\widetilde{p}}(0,T;L_{\widetilde q}(G))} 
 \|\nabla_\xi \delta\Bu\|_{L_1(0,T;L_\infty(G))}\Big),
\end{equation*}
\begin{multline*}
\|\pa_t \nabla_\xi \delta \sA\|_{L_{\widetilde{p}}(0,T; L_{\widetilde q}(G))} 
 \leq  C_{N,\kappa}\Big(  \|\nabla_\xi^2 \delta \Bu\|_{L_{\widetilde{p}}(0,T;L_{\widetilde q}(G))} 
 + \|\nabla_\xi^2 (\Bu_1,\Bu_2)\|_{L_1(0,T;L_{\widetilde q}(G))} 
 \|\nabla_\xi \delta\Bu\|_{L_{\widetilde{p}}(0,T;L_\infty(G))}\\
  + \|\nabla_\xi (\Bu_1,\Bu_2)\|_{L_{\tp}(0,T;L_\infty(G))} 
 \|\nabla_\xi^2 \delta\Bu\|_{L_{1}(0,T;L_{\tq}(G))}
    + \|\nabla_\xi^2 (\Bu_1,\Bu_2)\|_{L_{\tp}(0,T;L_{\tq}(G))} 
 \|\nabla_\xi \delta\Bu\|_{L_{1}(0,T;L_{\infty}(G))}\\
 +\|\nabla_\xi (\Bu_1,\Bu_2)\|_{L_{\tp}(0,T;L_{\infty}(G))}
   \|\nabla_\xi^2 (\Bu_1,\Bu_2)\|_{L_{1}(0,T;L_{\tq}(G))}
   \|\nabla_\xi \delta\Bu\|_{L_{1}(0,T;L_{\infty}(G))} \Big).
\end{multline*}
\end{enumerate}
Above all constants  $C_{N,\kappa}$ go to $\infty$ as $\kappa$ tends to $1.$ 
\end{lemm}
\begin{proof}
For simplicity, we only concentrate on the case $\FA_k = \sA_k$ for $k=1,2.$ The proof is based on the following expression of $\delta \sA$ (see \cite[Appendix]{DanM2012} for instance),
\begin{equation}\label{eq:deltaA}
\delta\sA(t) = \big(\int_0^t \nabla_\xi\delta\Bu \,d\tau\big) 
\sum_{\ell \geq 1}(-1)^{\ell} \sum_{j=0}^{\ell-1}C_1(t)^jC_2^{\ell-1-j}(t),
\end{equation}
where $C_k (t):=\int_0^t \nabla_\xi \Bu_k \, d \tau$ for $k=1,2.$ Thus our assumption $\kappa<1$ and \eqref{eq:deltaA} imply the bound 
\begin{equation}\label{es:deltaFA}
\|\delta\sA\|_{L_\infty (G \times ]0,T[)} 
\leq \frac{C_N}{(1-\kappa)^2}
 \|\nabla_\xi \delta\Bu\|_{L_1(0,T;L_\infty(G))}. 
\end{equation}
To complete the proof of the first part, without loss of generality, let us only consider the difference 
\begin{equation}\label{eq:AA-AA}
\sA_2\sA_2^{\top} -\sA_1\sA_1^{\top}
 = \delta \sA \sA_2^{\top} + \sA_1 \delta \sA^{\top}.
\end{equation}
Then Lemma \ref{lem:A_DA} (1) and \eqref{es:deltaFA} yield
\begin{equation*}
\|\sA_2\sA_2^{\top} -\sA_1\sA_1^{\top}\|_{L_\infty (G \times ]0,T[)} 
\leq \frac{C_N}{(1-\kappa)^3} 
\|\nabla_\xi \delta \Bu\|_{L_1(0,T;L_\infty(G))}. 
\end{equation*}
\medskip

Now, the proof of Part (2) also relies on the formulas \eqref{eq:deltaA} and \eqref{eq:AA-AA}. For instance, we have  
\begin{align}\label{es:DdeltaA}
\|\nabla_\xi \delta\sA(t)\|_{L_\infty(0,T;L_q(G))} 
&\leq \frac{C_N}{(1-\kappa)^2}\|\nabla_\xi^2 \delta\Bu\|_{L_1(0,T;L_q(G))}\\
& \quad  +\frac{C_N}{(1-\kappa)^3} \|\nabla_\xi^2 (\Bu_1,\Bu_2)\|_{L_1(0,T;L_q(G))}
\|\nabla_\xi \delta\Bu\|_{L_1(0,T;L_\infty(G))}. \nonumber
\end{align}
Then combining above inequalities \eqref{es:deltaFA}, \eqref{es:DdeltaA} and Lemma \ref{lem:A_DA} yields
\begin{align*}
\|\nabla_\xi (\sA_2\sA_2^{\top} -\sA_1\sA_1^{\top})\|_{L_\infty(0,T;L_q(G))} 
&\leq \frac{C_N}{(1-\kappa)^3}\|\nabla_\xi^2 \delta\Bu\|_{L_1(0,T;L_q(G))}\\
&\quad +\frac{C_N}{(1-\kappa)^4} \|\nabla_\xi^2 (\Bu_1,\Bu_2)\|_{L_1(0,T;L_q(G))}
\|\nabla_\xi \delta\Bu\|_{L_1(0,T;L_\infty(G))}.
\end{align*}
The proof for other terms in Part (2) are similar and hence it remains to study $\pa_t \sA$ in Part (3).

Finally, it is not hard to see that 
\begin{align*}
\|\pa_t \delta \sA\|_{L_{\widetilde{p}}(0,T; L_{\widetilde q}(G))}
&\leq \frac{C_N}{(1-\kappa)^2} \|\nabla_\xi \delta \Bu\|_{L_{\widetilde{p}}(0,T;L_{\widetilde q}(G))} \\
 &\quad + \frac{C_N}{(1-\kappa)^3} 
 \|\nabla_\xi (\Bu_1,\Bu_2)\|_{L_{\widetilde{p}}(0,T;L_{\widetilde q}(G))} 
 \|\nabla_\xi \delta\Bu\|_{L_1(0,T;L_\infty(G))}.
\end{align*}
Apply the operator $\pa_t \nabla_\xi$ to \eqref{eq:deltaA} and we have
\begin{align*}
\|\pa_t \nabla_\xi \delta \sA\|_{L_{\widetilde{p}}(0,T; L_{\widetilde q}(G))} 
& \leq  \frac{C_N}{(1-\kappa)^2} \|\nabla_\xi^2 \delta \Bu\|_{L_{\widetilde{p}}(0,T;L_{\widetilde q}(G))} \\
 &  + \frac{C_N}{(1-\kappa)^3} \Big(
 \|\nabla_\xi^2 (\Bu_1,\Bu_2)\|_{L_1(0,T;L_{\widetilde q}(G))} 
 \|\nabla_\xi \delta\Bu\|_{L_{\widetilde{p}}(0,T;L_\infty(G))}\\
 &\quad \quad\quad \quad\quad\quad
    + \|\nabla_\xi (\Bu_1,\Bu_2)\|_{L_{\tp}(0,T;L_\infty(G))} 
 \|\nabla_\xi^2 \delta\Bu\|_{L_{1}(0,T;L_{\tq}(G))}\\
& \quad \quad\quad \quad\quad\quad
    + \|\nabla_\xi^2 (\Bu_1,\Bu_2)\|_{L_{\tp}(0,T;L_{\tq}(G))} 
 \|\nabla_\xi \delta\Bu\|_{L_{1}(0,T;L_{\infty}(G))}\Big)\\
 &+ \frac{C_N}{(1-\kappa)^4} 
   \|\nabla_\xi (\Bu_1,\Bu_2)\|_{L_{\tp}(0,T;L_{\infty}(G))}
   \|\nabla_\xi^2 (\Bu_1,\Bu_2)\|_{L_{1}(0,T;L_{\tq}(G))}
   \|\nabla_\xi \delta\Bu\|_{L_{1}(0,T;L_{\infty}(G))},
\end{align*}
which completes our proof of Lemma \ref{lem:deltaA}.
\end{proof}
\medskip

To handle the free boundary condition in Lagrangian coordinate, we need the following estimates.
\begin{lemm}\label{lem:normal}
Assume that $G$ is a uniform $W^{2-1\slash q}_q$ connected domain in $\BBR^N$ and $\Bn$ is the unit normal along some boundary $\Gamma \subset \pa G.$ With the same conventions on $(q,\tp,\tq,\Bu,\Bu_k,\BX,\BX_k, \sA,\sA_k)$ $(k=1,2)$  as in Lemma \ref{lem:A_DA} and Lemma \ref{lem:deltaA}, we define 
\begin{equation*}
(\oBn, \oBn_k)(\xi,t) := \Big(\frac{\sA \Bn}{|\sA \Bn|}, \frac{\sA_k \Bn}{|\sA_k \Bn|} \Big)(\xi,t) \quad \hbox{for any}\,\,
(\xi,t) \in \Gamma \times ]0,T[.
\end{equation*}
For simplicity, we use $\oFn$ for any element in $\{ \oBn, \oBn_1, \oBn_2\}.$ 
Then the following assertions hold true as long as $0<\kappa \ll 1.$
\begin{enumerate}
\item  $\Bn$ and $\oFn$ can be extended into $W^{1}_q(G)^N.$ Moreover, there is some constant $C_{N,\kappa}$ such that 
\begin{equation}\label{es:n}
\|\Bn\|_{L_\infty(G) \cap W^{1}_q(G)} 
+ \|\oFn\|_{L_\infty (0,T;L_\infty(G) \cap W^{1}_q(G))}\leq C_{N,\kappa}.
\end{equation}
\item If we take  $\FA$ in $\{ \sA , \sA_1,\sA_2 \}$  with respect to the corresponding  $\oFn$ as above, then there exists a constant $C$ such that
\begin{align*}
\|\oFn-\Bn\|_{L_\infty(G\times ]0,T[)} & \leq C_{N,\kappa}
 \|\FA-\BBI\|_{L_\infty(G\times ]0,T[)},\\
 \| \nabla_\xi (\oFn-\Bn)\|_{L_\infty(0,T;L_q(G))} & \leq C_{N,\kappa}\big(
 \| \nabla_\xi \FA\|_{L_\infty(0,T;L_q(G))} + \|\FA-\BBI\|_{L_\infty(G\times ]0,T[)} \big),\\
  \| \pa_t (\oFn-\Bn)\|_{L_{\tp}(0,T;L_{\tq}(G))} & \leq C_{N,\kappa}
  \| \pa_t \FA\|_{L_{\tp}(0,T;L_{\tq}(G))}, \\
   \| \pa_t \nabla_\xi (\oFn-\Bn)\|_{L_{p}(0,T;L_{q}(G))} &
    \leq C_{N,\kappa} \big(
 \|\pa_t \nabla_\xi \FA\|_{L_{p}(0,T;L_{q}(G))} 
  +\|\pa_t \FA\|_{L_{p}(0,T;L_{\infty}(G))} \\ 
 &\hspace{1.5cm} + \|\pa_t \FA\|_{L_{p}(0,T;L_{\infty}(G))} 
  \|\nabla_\xi \FA\|_{L_{\infty}(0,T;L_{q}(G))}  \big). 
\end{align*}
\item Set $\delta \oBn := \oBn_2 -\oBn_1$ and recall $\delta \sA := \sA_2- \sA_1.$ Then there exists a constant $C$ such that
\begin{align*}
\|\delta \oBn\|_{L_\infty(G\times ]0,T[)} & \leq C_{N,\kappa}
      \|\delta \sA\|_{L_\infty(G\times ]0,T[)},\\
\| \nabla_\xi \delta \oBn\|_{L_\infty(0,T;L_q(G))} & \leq C_{N,\kappa}\Big(
       \| \nabla_\xi \delta \sA\|_{L_\infty(0,T;L_q(G))}
       + \|\delta \sA\|_{L_\infty(G\times ]0,T[)} A_{\infty,q}(T)\Big),\\
\| \pa_t \delta \oBn\|_{L_{\tp}(0,T;L_{\tq}(G))} & \leq C_{N,\kappa}\Big(
          \| \pa_t \delta \sA\|_{L_{\tp}(0,T;L_{\tq}(G))} 
              + \|\delta \sA\|_{L_\infty(G\times ]0,T[)} 
            B_{\tp,\tq}(T)\Big), \\
\| \pa_t \nabla_\xi \delta \oBn\|_{L_{p}(0,T;L_{q}(G))} & \leq C_{N,\kappa}\Big(
             \|\pa_t \nabla_\xi \delta \sA\|_{L_{p}(0,T;L_{q}(G))} 
             +\|\pa_t \delta \sA\|_{L_{p}(0,T;L_{\infty}(G))}A_{\infty,q}(T) \\
               &\hspace{1cm}  
              +\|\nabla_\xi \delta \sA\|_{L_{\infty}(0,T;L_{q}(G))} B_{p,\infty}(T)
              +\|\delta \sA\|_{L_{\infty}(G\times ]0,T[)} C_{p,q}(T) \Big), 
\end{align*}
where constants $A_{p,q}(T),$ $B_{p,q}(T)$ and $C_{p,q}(T)$ are given by
\begin{align*}
A_{p,q}(T)&:=1+ \| \nabla_\xi (\sA_1,\sA_2)\|_{L_p(0,T;L_q(G))}, \\
B_{p,q}(T)&:=\|\pa_t (\sA_1,\sA_2)\|_{L_{p}(0,T;L_{q}(G))},\\
C_{p,q}(T)&:=\|\pa_t \nabla_\xi (\sA_1,\sA_2)\|_{L_{p}(0,T;L_{q}(G))} 
             +  A_{\infty,q}(T)\, B_{p,\infty}(T).
\end{align*}
\end{enumerate}
\end{lemm}
\begin{proof}
The proof of \eqref{es:n} is standard and let us verify the second part.
Without loss of generality, we just treat the pair $(\oBn,\sA)$ here and set $\sB:=\sA-\BBI$. Then the difference $\oBn-\Bn$ is formulated as follows,
\begin{equation}\label{eq:n-n}
\oBn-\Bn = \frac{ \sB \Bn}{|\sA \Bn|} 
             +\Big(\frac{1}{|\sA \Bn|}-1\Big)\Bn 
=\frac{ \sB\Bn}{|\sA \Bn|}-
\frac{2\sB\Bn \cdot \Bn+|\sB \Bn|^2}{|\sA \Bn| (1+|\sA \Bn|)} \Bn
\end{equation}
\smallbreak 

Firstly, we know from \eqref{es:n} that
\begin{equation}\label{es:Bn}
\|\sB \Bn\|_{L_\infty(G\times]0,T[)} 
\leq C_{N,\kappa} \|\sB\|_{L_\infty(G\times]0,T[)},
\end{equation} 
which also yields for $\kappa$ small enough,
\begin{equation}\label{es:An}
 0<c_{_{N,\kappa}} \leq  |\Bn|- |\sB\Bn| \leq |\sA \Bn|
 \leq \|\sA\Bn\|_{L_\infty(G\times]0,T[)} \leq C_{N,\kappa}.
\end{equation}
Hence, keeping \eqref{es:Bn} and \eqref{es:An} in mind,  Lemma \ref{lem:A_DA} yields our result for $\|\oBn-\Bn\|_{L_\infty(G\times ]0,T[)}.$
\smallbreak 

To bound $\nabla_\xi (\oBn-\Bn),$ we first note by \eqref{es:n}
\begin{equation}\label{es:DBn}
\|\nabla_\xi(\sB\Bn)\|_{L_\infty(0,T;L_q(G))} 
 \leq C_{N,\kappa}\big( \|\nabla_\xi \sA\|_{L_\infty(0,T;L_q(G))}+\|\sB\|_{L_\infty(G\times]0,T[)}\big).
\end{equation}
Then \eqref{es:DBn} and \eqref{es:n}, together with Lemma \ref{lem:A_DA}, imply that
\begin{equation}\label{es:DAn}
\|\nabla_\xi(\sA\Bn)\|_{L_\infty(0,T;L_q(G))} \leq C_{N,\kappa}\big( 1+ \|\nabla_\xi \sA\|_{L_\infty(0,T;L_q(G))}\big). 
\end{equation}
Thus based on \eqref{es:DBn} and \eqref{es:DAn},   direct computations yield the bound of  $\nabla_\xi (\oBn-\Bn).$
\smallbreak 

Next, the bound of $\pa_t (\oBn-\Bn)$ is immediate from the following inequality,
\begin{equation}\label{es:DtAnBn}
\|\pa_t (\sA\Bn,\sB\Bn)\|_{L_{\tp}(0,T;L_{\tq}(G))} \leq C_{N,\kappa} \|\pa_t \sA\|_{L_{\tp}(0,T;L_{\tq}(G))}. 
\end{equation}
\smallbreak

Now, combining \eqref{es:n} and  above inequality \eqref{es:DtAnBn} implies
  \begin{equation}\label{es:DtDAnBn}
\|\pa_t \nabla_\xi (\sA\Bn,\sB\Bn)\|_{L_{p}(0,T;L_{q}(G))} 
\leq C_{N,\kappa} \big( \|\pa_t \nabla_\xi \sA\|_{L_{p}(0,T;L_{q}(G))}
+ \|\pa_t\sA\|_{L_{p}(0,T;L_{\infty}(G))} \big). 
\end{equation}
Therefore we can conclude the bound of $\pa_t \nabla_\xi (\oFn-\Bn)$ by putting \eqref{es:Bn} - \eqref{es:DtDAnBn} together.
\medskip

In fact, the proof of last part is similar to the previous step and is based on the following expression,
\begin{align*}
\delta \oBn &= \frac{\delta\sA\Bn}{|\sA_2 \Bn|} 
             +\Big(\frac{1}{|\sA_2 \Bn|}-\frac{1}{|\sA_1 \Bn|}\Big)\sA_1\Bn \\
&=\frac{\delta\sA\Bn}{|\sA_2 \Bn|}-
\frac{ (\delta \sA \Bn)\cdot (\sA_2 \Bn)+(\sA_1\Bn)\cdot (\delta\sA \Bn)}{|\sA_1 \Bn| |\sA_2 \Bn|(|\sA_1 \Bn|+|\sA_2 \Bn|)} \sA_1\Bn.
\end{align*}
The details are left to the readers.
\end{proof}

%%%%%%%%%%%%%%%%%%%%%%%%%%%%%%%%%%%%%%%%%%%%%%%%%%%%%%%%%%%%%%%%%%%%%%%%%%%%%%%%%%%%%%%%%%
%%%%%%%%%%%%%%%%%%%%%%%%%%%%%%%%%%%%%%%%%%%%%%%%%%%%%%%%%%%%%%%%%%%%%%%%%%%%%%%%%%%%%%%%%%
\section{Some interpolation property}\label{appendix:interpolation}
For convenience, let us recall the definition of Stokes operator $\CA_q$ with $1<q< \infty$ in Section \ref{sec:intro}. 
We say some vector field $\Bu$ satisfies the two phase compatibility condition in $\dOm$ (either $\Gamma_+$ or $\Gamma_-$ could be $\emptyset$ ) if $\Bu$ enjoys 
\begin{equation}\label{cdt:compt}
\jump{\Bu}|_{\Ga}=\jump{\CT_{\Bn} \big(\mu \BBD (\Bu) \Bn\big)}|_{\Ga}  =  \0, \,\,\,
\CT_{\Bn_+} \big(\mu \BBD (\Bu) \Bn_+\big)|_{\Ga_+}
 =\0   
\quad \hbox {and }\quad \Bu|_{\Ga_-}=0,
\end{equation}
where the projection operator 
\begin{equation*}
\CT_{\Bnu}\Bh := \Bh-(\Bh\cdot \Bnu)\Bnu,
\end{equation*}
for any vector $\Bnu$ and $\Bh$ defined along some surface.
Then $\CA_q$ is given by
\begin{equation*}
\CA_q \Bu := \eta^{-1} \Di \BBT\big(\Bu,K(\Bu)\big), \quad \hbox{for any} \,\,\, \Bu \in 
\CD (\CA_q) := \big\{\Bu \in W^{2}_q (\dOm)^N \cap J_q (\dOm) :  \Bu\,\,\, \hbox{satisfies}\,\,\, \eqref{cdt:compt} \big\}.
\end{equation*}
Now, set the real interpolation space
$\CD^{2\theta}_{q,p}(\dOm) := \big(J_q (\dOm), \CD(\CA_q) \big)_{\theta,p}
$ for $0<\theta<1,$ $1<q<\infty$ and $1\leq p\leq \infty.$
In fact we have ,
\begin{equation*}
\CD^{2-2\slash p}_{q,p}(\dOm)  := 
\left\{\begin{aligned}
				\big\{ \Bu \in B^{2-2\slash p}_{q,p}(\dOm) \cap J_q(\dOm): \Bu \quad \hbox{satisfies} \quad \eqref{cdt:compt} \big\} 
			 &\quad\mbox{if} \quad 2-2\slash p > 1+1\slash q,\\
			\big\{ \Bu \in B^{2-2\slash p}_{q,p}(\dOm) \cap J_q(\dOm): \Bu|_{\Ga_-}=0 \big\} 
			&\quad\mbox{if} \quad 1\slash q<2-2\slash p< 1+1\slash q, \\
		B^{2-2\slash p}_{q,p}(\dOm) \cap J_q(\dOm) &\quad\mbox{if} \quad 0<2-2\slash p < 1\slash q.
	\end{aligned}\right.
\end{equation*}
\smallbreak

From now on, we write some (one or two phase) Stokes operator defined in some uniform $W^{2-1\slash r}_r$ ($N<r<\infty$) domain $G$ by $A_q$ for simplicity.  Here $A_q$ may be associated to some suitable compatibility boundary conditions like \eqref{cdt:compt} (see \cite{Shi2015} for one phase case), and 
$\CD^{2-2\slash p}_{q,p}(G) := \big(J_q (G), \CD(A_q) \big)_{1- 1\slash p,p}$ for some $1<p,q<\infty.$  
Then the following general interpolation properties hold true
 (see \cite{DanZhp2014} for the case $G=\BBR^N_+$).
\begin{prop}\label{prop_uDu}
Assume that $(\theta,\beta, q,p)$ satisfies the following conditions
\begin{equation*}
(\theta,\beta, q,p) \in ]0,1[\times ]1,\infty]\times ]1,\infty[\times]1,\infty[, 
\quad \beta \geq q \quad \mbox{and} \quad 
 1-\frac{2(1-\theta)}{p} = \frac{N}{q}-\frac{N}{\beta}.
\end{equation*} 
Then the following inequality holds true,
\begin{equation}\label{es:interpolation}
\|\nabla \Bu\|_{L_\beta(G)}  \lesssim \|\Bu\|_{\CD^{2-2\slash p}_{q,p}(G)}^{1-\theta} \|\Bu\|_{\CD(A_q)}^{\theta}.
\end{equation}
\end{prop} 

Consider smooth enough $\Bu$ and take $L_{p \slash \theta}(0,t)$ on both sides of \eqref{es:interpolation}, 
\begin{equation*}
\|\nabla \Bu\|_{L_{p\slash \theta} (0,T;L_\beta(G))}  \lesssim \|\Bu\|_{L_{\infty} \big(0,T; \CD^{2-2\slash p}_{q,p}(G) \big)}^{1-\theta} \|\Bu\|_{L_p (0,T;W^2_{q}(G))}^{\theta}\, ,
\end{equation*}
which gives the decay of 
$\|\nabla \Bu\|_{L_{p} (0,T;L_\beta(\Omega))}$ 
for $G$ fulfilling $(\Omega_1)- (\Omega_3)$
as in \cite[Lemma 4.1]{Dan2006},
\begin{equation}\label{es:Du_decay}
\|\nabla \Bu\|_{L_{p} (0,T; L_\beta(G))}  
\lesssim T^{\frac{1}{2}  (1-\frac{N}{q} + \frac{N}{\beta})} \|\Bu\|_{L_{\infty} \big(0,T; \CD^{2-2\slash p}_{q,p}(G) \big)}^{1-\theta} 
\|\Bu\|_{L_p (0,T; W^2_{q}(G))}^{\theta} \,.
\end{equation} 
\begin{proof}[Proof of Proposition \ref{prop_uDu}]
Recall the definition of $\CD^{2-2\slash p}_{q,p}(G),$ i.e.
\begin{equation*}
\CD^{2-2\slash p}_{q,p}(G)  = \big(J_q (G) , \CD(A_q)\big)_{1-1\slash p,p}, \quad \forall p\in ]1,\infty[.
\end{equation*}
For any $(\theta, r)\in ]0,1[\times [1,\infty[,$ Reiteration Theorem (see \cite{Lun2009} for instance) yields that,
\begin{equation*}
\big(\CD^{2-2\slash p}_{q, p}(G), \CD(A_q)\big)_{\theta,r} 
= \big( J_q(G), \CD(A_q)\big)_{(1-\theta)(1-1\slash p)+\theta,r} 
=:\CD^{2(1-\theta)(1-1\slash p)+2\theta}_{q,r}(G),
\end{equation*}
which implies that
\begin{equation*}
\|\Bu\|_{\CD^{2(1-\theta)(1-1\slash p)+2\theta}_{q,r}(G)} 
\lesssim \|\Bu\|_{\CD^{2-2\slash p}_{q,p}(G)}^{1-\theta} \|\Bu\|_{\CD(A_q)}^{\theta}.
\end{equation*}
Note that $(1-\theta)(1-1\slash p)+\theta = 1-(1-\theta)\slash p \in ]0,1[.$ 
Then by definition of $\beta$ and  above inequality, we have
\begin{equation*}
\|\nabla \Bu\|_{B^{0}_{\beta,r}(G)}
\lesssim \|\nabla \Bu\|_{B^{1-2(1-\theta)\slash p}_{q,r}(G)}
\lesssim \|\Bu\|_{\CD^{2-2\slash p}_{q,p}(G)}^{1-\theta}
          \|\Bu\|_{\CD(A_q)}^{\theta}.
\end{equation*}
If we take $r=1,$ then the embedding $B^{0}_{\beta,1}(G) \hookrightarrow L_{\beta}(G)$ (e.g. see \cite{Tri2006}) gives the desired result.
\end{proof}
\medskip

To study the (fractional) time derivative of quasilinear terms in \eqref{eq:INSL}, we need
to recall \cite[Lemma 2.7, Lemma 3.2]{ShiShi2007} by Y.Shibata and S.Shimizu.
To this end,  for any $(s,\sigma,p,q) \in \BBR_+^2 \times [1,\infty]^2,$ set that
\begin{equation*}
H^{s,\sigma}_{q,p}(G\times \BBR) := L_p\big(\BBR; H^{s}_q(G) \big) \cap
H^{\sigma}_p\big(\BBR;L_q(G) \big). 
\end{equation*}
\begin{lemm}\label{lem:Hhalf}
Let $(p,q) \in ]1,\infty[ \times ]N,\infty[.$ 
Assume $g \in H^{1\slash 2 ,1\slash 2}_{q,p} (G \times \BBR)$ 
and $f \in H^{1 ,1}_{q, \infty} (G \times \BBR).$ 
Then there exists a constant $C$ such that 
\begin{equation*}
\| f g\|_{H^{1\slash 2 ,1\slash 2}_{q,p} (G \times \BBR)} 
\leq C \|f\|_{H^{1 ,1}_{q, \infty}(G \times \BBR)} 
 \| g\|_{H^{1\slash 2 ,1\slash 2}_{q,p}(G\times \BBR)}.
\end{equation*}
If $f$ additionally satisfies the following condition for some 
$0<T\leq 1,$
\begin{equation*}
f = 0 \,\,\hbox{for}\,\, t \notin ]0,2T[ \,\,\,\hbox{and} \,\,\,\pa_t f \in L_p\big( 0,2T; H^1_q(G)\big),
\end{equation*}
then we have for some constant $C_{p,q},$
\begin{multline*}
\| f g\|_{H^{1\slash 2 ,1\slash 2}_{q,p} (G \times \BBR)} \leq C_{p,q} 
\|f\|_{L_\infty(G\times ]0,2T[)}^{1\slash 2} 
\Big( \|\nabla f\|_{L_\infty(0,2T;L_q(G))}
 +\|f\|_{L_\infty (G \times ]0,2T[)}\\
  +T^{(q-N)\slash (pq)} \|\pa_t f\|_{L_\infty(0,2T;L_q(G))}^{1-N\slash (2q)}
  \|\pa_t f\|_{L_p(0,2T;H^1_q(G))}^{N\slash (2q)}
\Big)^{1\slash 2}
 \| g\|_{H^{1\slash 2 ,1\slash 2}_{q,p}(G\times \BBR)}.
\end{multline*}
\end{lemm}
In fact, Lemma \ref{lem:Hhalf} works well for the case $p>2$ in the application to \eqref{eq:INSL}. However, if assuming $1<p\leq 2,$ we need the following result based on Proposition \ref{prop_uDu}.
\begin{lemm}\label{lem:Hhalf2}
Let $(\theta_1,\theta_2,\alpha,\beta, q,p,T) \in ]0,1[^2\times [q,\infty]^2 \times ]N,\infty[ \times [1,2] \times ]0,1]$ satisfy 
\begin{equation}\label{cdt:Hhalf2}
\theta_1 + \theta_2 \in ]0, 1], \,\,\,
\frac{1}{q} = \frac{1}{\alpha} + \frac{1}{\beta}, \,\,\,
 1-\frac{1-\theta_1}{p} = \frac{N}{q}-\frac{N}{\alpha} \,\,\, \hbox{and}\,\,\,  1-\frac{2(1-\theta_2)}{p} = \frac{N}{q}-\frac{N}{\beta}.
\end{equation}
 Assume that $g \in H^{1\slash 2 ,1\slash 2}_{q,p} (G \times \BBR)$ and 
 $f \in L_\infty (\BBR; W^1_q(G))$ fulfilling
\begin{equation*}
f = 0 \,\,\hbox{for}\,\, t \notin ]0,2T[ \,\,\,\hbox{and} \,\,\,\pa_t f \in L_{p \slash \theta_2}\big(0,2T; L_\beta(G)\big).
\end{equation*}
Then there exists a constant $C_{p,q}$ such that 
\begin{multline*}
\| f g\|_{H^{1\slash 2 ,1\slash 2}_{q,p} (G \times \BBR)} \leq C_{p,q} 
\|f\|_{L_\infty(G\times ]0,2T[)}^{1\slash 2} 
\Big( 
\|\nabla f\|_{L_\infty(0,2T;L_q(G))}
 +\|f\|_{L_\infty (G \times ]0,2T[)}\\
+T^{\frac{3}{2}-\frac{1}{p}-\frac{N}{2q} -\frac{N}{2\beta}} 
         \|\pa_t f\|_{L_{p \slash \theta_2}(0,2T;L_\beta(G))}
\Big)^{1\slash 2}
 \| g\|_{H^{1\slash 2 ,1\slash 2}_{q,p}(G\times \BBR)}.
\end{multline*}
\end{lemm}

\begin{proof} 
Suppose $g$ belong to $H^{1,1}_{q,p}(G\times\BBR).$ 
It is easily to verify that
\begin{equation}\label{es:fg0}
\|fg\|_{L_p(\BBR;L_q(G))} \lesssim \|f\|_{L_\infty (G \times ]0,2T[)}  
\|g\|_{L_p(\BBR;L_q(G))},
\end{equation}
\begin{equation}\label{es:fg1}
\| f g\|_{H^{1 ,1}_{q,p} (G \times \BBR)} \lesssim
\Big( \|\nabla f\|_{L_\infty(0,2T;L_q(G))}
 +\|f\|_{L_\infty (G \times ]0,2T[)} 
 \Big) \| g\|_{H^{1 ,1}_{q,p}(G\times \BBR)}
+ \|(\pa_t f)\, g \|_{L_p(0,2T;L_q(G))}.
\end{equation}
In the following, we will mainly study the last term on the r.h.s. of \eqref{es:fg1}. 

Now, recall the fact that
$ H^{1,1}_{q,p}(G\times\BBR) \hookrightarrow 
\BUC \big(\BBR_+; B^{1-1\slash p}_{q,p}(G) \big).$ 
Then we obtain from Reiteration Theorem (see \cite{Lun2009}) that  for any $(\theta_1, r)\in ]0,1[\times [1,\infty[,$ 
\begin{equation*}
\big(B^{1-1\slash p}_{q, p}(G), W^1_q(G)\big)_{\theta_1,r} 
= \big( L_q(G), W^1_q(G) \big)_{(1-\theta_1)(1-1\slash p)+\theta_1, r} 
=B^{1-(1-\theta_1)\slash p}_{q,r}(G).
\end{equation*}
Taking  $\alpha \geq q$ such that $1-(1-\theta_1)\slash p = N \slash q-N\slash \alpha$, we have
\begin{equation*}
\|g\|_{L_\alpha(G)} \lesssim \|g\|_{B^{1-(1-\theta_1)\slash p}_{q,1}(G)} 
\lesssim \|g\|_{B^{1-1\slash p}_{q, p}(G)}^{1-\theta_1}
           \|g\|_{W^1_q(G)}^{\theta_1},
\end{equation*}
which  yields that
\begin{equation*}
\|g\|_{L_{p \slash \theta_1}(\BBR_+; L_\alpha(G))} 
\lesssim \|g\|_{H^{1,1}_{q,p}(G\times\BBR)}.
\end{equation*}
Therefore, we infer from above bound and Conditions \eqref{cdt:Hhalf2}
\begin{align}\label{es:fg2}
  \|(\pa_t f)\, g \|_{L_p(0,2T;L_q(G))} 
    &\lesssim T^{(1-\theta_1-\theta_2)\slash p} 
         \|\pa_t f\|_{L_{p \slash \theta_2}(0,2T;L_\beta(G))}
         \|g\|_{L_{p \slash \theta_1}(0,2T;L_\alpha(G))}\\
 & \lesssim    T^{\frac{3}{2}-\frac{1}{p}-\frac{N}{2q} -\frac{N}{2\beta}} 
         \|\pa_t f\|_{L_{p \slash \theta_2}(0,2T;L_\beta(G))}
         \|g\|_{H^{1,1}_{q,p}(G\times\BBR)}. \nonumber
\end{align}
Finally, combining the bounds \eqref{es:fg0}, \eqref{es:fg1}, \eqref{es:fg2} and the fact
 $$H^{1\slash 2 ,1\slash 2}_{q,p}(G\times \BBR) =
 \Big( L_p\big(\BBR;L_q(G)\big),
 H^{1, 1}_{q,p}(G\times \BBR) \Big)_{[1\slash 2]},$$
yield the desired result.
\end{proof}

Let us end this part with some comments on Lemma \ref{lem:Hhalf2}.
 In fact, we have 
$$N\slash q +1\slash p = 3\slash 2 +N\slash (2\alpha) >3\slash 2$$
 by fixing $\theta_1+\theta_2=1$ in \eqref{cdt:Hhalf2}, which gives our definition of $(II)$ in our main theorem.
On the other hand, if ignoring the purpose of applying Lemma \ref{lem:Hhalf2} to \eqref{eq:INSL}, then the proof of Lemma \ref{lem:Hhalf2} yields the following slightly general result.
 \begin{lemm}\label{lem:Hhalf3}
Let $(\theta,\alpha,\beta, q,p) \in ]0,1[\times [q,\infty]^2 \times ]N,\infty[ \times [1,2]$ satisfy 
\begin{equation*}
\frac{1}{q} = \frac{1}{\alpha} + \frac{1}{\beta}, \,\,\,
 1-\frac{\theta}{p} = \frac{N}{q}-\frac{N}{\alpha} \,\,\, \hbox{and}\,\,\,  1-\frac{2(1-\theta)}{p} = \frac{N}{q}-\frac{N}{\beta}.
\end{equation*}
 Assume that $g \in H^{1\slash 2 ,1\slash 2}_{q,p} (G \times \BBR)$ and 
 $f \in L_\infty (\BBR; W^1_q(G))$ fulfilling
$\pa_t f \in L_{p \slash \theta}\big(\BBR; L_\beta(G)\big).$
Then there exists a constant $C_{p,q}$ such that 
\begin{equation*}
\| f g\|_{H^{1\slash 2 ,1\slash 2}_{q,p} (G \times \BBR)} \leq C_{p,q} 
\|f\|^{1\slash 2}_{L_\infty(G\times \BBR)} 
\Big(  \|f\|_{L_\infty(\BBR;W_q^1(G))}
+ \|\pa_t f\|_{L_{p \slash \theta}(\BBR;L_\beta(G))}
\Big)^{1\slash 2}
 \| g\|_{H^{1\slash 2 ,1\slash 2}_{q,p}(G\times \BBR)}.
\end{equation*}
\end{lemm}
\end{appendix}

%%%%%%%%%%%%%%%%%%%%%%%%%%%%%%%%%%%%%%%%%%%%%%%%%%%%%%%%%%%%%%%%%%%%%%%%%%

\section*{Acknowledgement}
The research of H.S. was supported by JSPS Grant-in-aid for Young Scientists (B) 17K14224.
Y.S.  acknowledges the support from JSPS Grant-in-aid for Scientific Research (A) 17H0109 and Top Global University Project.
 X.Z. would like to thank Professor Rapha\"el Danchin for his discussions on this topic and also for his careful reading on the early version of this manuscript.

%\bibliographystyle{plain}
%\bibliography{bibtex_twophase}

\end{document}